\documentclass[a4paper, british]{amsart}

\usepackage{libertine}
\usepackage[libertine]{newtxmath}

\usepackage{babel}
\usepackage{enumitem}
\usepackage{hyperref}
\usepackage[utf8]{inputenc}
\usepackage{newunicodechar}
\usepackage{mathtools}
\usepackage{varioref}
\usepackage[arrow,curve,matrix]{xy}

\usepackage{colortbl}
\usepackage{graphicx}
\usepackage{tikz}

\usepackage{mathrsfs}

\definecolor{linkred}{rgb}{0.7,0.2,0.2}
\definecolor{linkblue}{rgb}{0,0.2,0.6}

\numberwithin{figure}{section}

\usepackage[hyperpageref]{backref}

\setdescription{labelindent=\parindent, leftmargin=2\parindent}
\setitemize[1]{labelindent=\parindent, leftmargin=2\parindent}
\setenumerate[1]{labelindent=0cm, leftmargin=*, widest=iiii}

\newunicodechar{α}{\ensuremath{\alpha}}
\newunicodechar{β}{\ensuremath{\beta}}
\newunicodechar{χ}{\ensuremath{\chi}}
\newunicodechar{δ}{\ensuremath{\delta}}
\newunicodechar{ε}{\ensuremath{\varepsilon}}
\newunicodechar{Δ}{\ensuremath{\Delta}}
\newunicodechar{η}{\ensuremath{\eta}}
\newunicodechar{γ}{\ensuremath{\gamma}}
\newunicodechar{Γ}{\ensuremath{\Gamma}}
\newunicodechar{ι}{\ensuremath{\iota}}
\newunicodechar{κ}{\ensuremath{\kappa}}
\newunicodechar{λ}{\ensuremath{\lambda}}
\newunicodechar{Λ}{\ensuremath{\Lambda}}
\newunicodechar{ν}{\ensuremath{\nu}}
\newunicodechar{μ}{\ensuremath{\mu}}
\newunicodechar{ω}{\ensuremath{\omega}}
\newunicodechar{Ω}{\ensuremath{\Omega}}
\newunicodechar{π}{\ensuremath{\pi}}
\newunicodechar{Π}{\ensuremath{\Pi}}
\newunicodechar{φ}{\ensuremath{\phi}}
\newunicodechar{Φ}{\ensuremath{\Phi}}
\newunicodechar{ψ}{\ensuremath{\psi}}
\newunicodechar{Ψ}{\ensuremath{\Psi}}
\newunicodechar{ρ}{\ensuremath{\rho}}
\newunicodechar{σ}{\ensuremath{\sigma}}
\newunicodechar{Σ}{\ensuremath{\Sigma}}
\newunicodechar{τ}{\ensuremath{\tau}}
\newunicodechar{θ}{\ensuremath{\theta}}
\newunicodechar{Θ}{\ensuremath{\Theta}}
\newunicodechar{ξ}{\ensuremath{\xi}}
\newunicodechar{Ξ}{\ensuremath{\Xi}}
\newunicodechar{ζ}{\ensuremath{\zeta}}

\newunicodechar{ℓ}{\ensuremath{\ell}}
\newunicodechar{ï}{\"{\i}}

\newunicodechar{𝔸}{\ensuremath{\bA}}
\newunicodechar{𝔹}{\ensuremath{\bB}}
\newunicodechar{ℂ}{\ensuremath{\bC}}
\newunicodechar{𝔻}{\ensuremath{\bD}}
\newunicodechar{𝔼}{\ensuremath{\bE}}
\newunicodechar{𝔽}{\ensuremath{\bF}}
\newunicodechar{𝔾}{\ensuremath{\bG}}
\newunicodechar{ℕ}{\ensuremath{\bN}}
\newunicodechar{ℙ}{\ensuremath{\bP}}
\newunicodechar{ℚ}{\ensuremath{\bQ}}
\newunicodechar{ℝ}{\ensuremath{\bR}}
\newunicodechar{𝕏}{\ensuremath{\bX}}
\newunicodechar{ℤ}{\ensuremath{\bZ}}
\newunicodechar{𝒜}{\ensuremath{\sA}}
\newunicodechar{ℬ}{\ensuremath{\sB}}
\newunicodechar{𝒞}{\ensuremath{\sC}}
\newunicodechar{𝒟}{\ensuremath{\sD}}
\newunicodechar{ℰ}{\ensuremath{\sE}}
\newunicodechar{ℱ}{\ensuremath{\sF}}
\newunicodechar{𝒢}{\ensuremath{\sG}}
\newunicodechar{ℋ}{\ensuremath{\sH}}
\newunicodechar{𝒥}{\ensuremath{\sJ}}
\newunicodechar{ℒ}{\ensuremath{\sL}}
\newunicodechar{𝒪}{\ensuremath{\sO}}
\newunicodechar{𝒬}{\ensuremath{\sQ}}
\newunicodechar{𝒮}{\ensuremath{\sS}}
\newunicodechar{𝒯}{\ensuremath{\sT}}
\newunicodechar{𝒲}{\ensuremath{\sW}}

\newunicodechar{∂}{\ensuremath{\partial}}
\newunicodechar{∇}{\ensuremath{\nabla}}

\newunicodechar{↺}{\ensuremath{\circlearrowleft}}
\newunicodechar{∞}{\ensuremath{\infty}}
\newunicodechar{⊕}{\ensuremath{\oplus}}
\newunicodechar{⊗}{\ensuremath{\otimes}}
\newunicodechar{•}{\ensuremath{\bullet}}
\newunicodechar{Λ}{\ensuremath{\wedge}}
\newunicodechar{↪}{\ensuremath{\into}}
\newunicodechar{→}{\ensuremath{\to}}
\newunicodechar{↦}{\ensuremath{\mapsto}}
\newunicodechar{⨯}{\ensuremath{\times}}
\newunicodechar{∪}{\ensuremath{\cup}}
\newunicodechar{∩}{\ensuremath{\cap}}
\newunicodechar{⊋}{\ensuremath{\supsetneq}}
\newunicodechar{⊇}{\ensuremath{\supseteq}}
\newunicodechar{⊃}{\ensuremath{\supset}}
\newunicodechar{⊊}{\ensuremath{\subsetneq}}
\newunicodechar{⊆}{\ensuremath{\subseteq}}
\newunicodechar{⊂}{\ensuremath{\subset}}
\newunicodechar{⊄}{\ensuremath{\not \subset}}
\newunicodechar{≥}{\ensuremath{\geq}}
\newunicodechar{≠}{\ensuremath{\neq}}
\newunicodechar{≫}{\ensuremath{\gg}}
\newunicodechar{≪}{\ensuremath{\ll}}

\newunicodechar{≤}{\ensuremath{\leq}}
\newunicodechar{∈}{\ensuremath{\in}}
\newunicodechar{∉}{\ensuremath{\not \in}}
\newunicodechar{∖}{\ensuremath{\setminus}}
\newunicodechar{◦}{\ensuremath{\circ}}
\newunicodechar{°}{\ensuremath{^\circ}}
\newunicodechar{…}{\ifmmode\mathellipsis\else\textellipsis\fi}
\newunicodechar{·}{\ensuremath{\cdot}}
\newunicodechar{⋯}{\ensuremath{\cdots}}
\newunicodechar{∅}{\ensuremath{\emptyset}}
\newunicodechar{⇒}{\ensuremath{\Rightarrow}}

\newunicodechar{⁰}{\ensuremath{^0}}
\newunicodechar{¹}{\ensuremath{^1}}
\newunicodechar{²}{\ensuremath{^2}}
\newunicodechar{³}{\ensuremath{^3}}
\newunicodechar{⁴}{\ensuremath{^4}}
\newunicodechar{⁵}{\ensuremath{^5}}
\newunicodechar{⁶}{\ensuremath{^6}}
\newunicodechar{⁷}{\ensuremath{^7}}
\newunicodechar{⁸}{\ensuremath{^8}}
\newunicodechar{⁹}{\ensuremath{^9}}
\newunicodechar{ⁱ}{\ensuremath{^i}}

\newunicodechar{⌈}{\ensuremath{\lceil}}
\newunicodechar{⌉}{\ensuremath{\rceil}}
\newunicodechar{⌊}{\ensuremath{\lfloor}}
\newunicodechar{⌋}{\ensuremath{\rfloor}}

\newunicodechar{≅}{\ensuremath{\cong}}
\newunicodechar{⇔}{\ensuremath{\Leftrightarrow}}
\newunicodechar{∃}{\ensuremath{\exists}}
\newunicodechar{±}{\ensuremath{\pm}}

\DeclareFontFamily{OMS}{rsfs}{\skewchar\font'60}
\DeclareFontShape{OMS}{rsfs}{m}{n}{<-5>rsfs5 <5-7>rsfs7 <7->rsfs10 }{}
\DeclareSymbolFont{rsfs}{OMS}{rsfs}{m}{n}
\DeclareSymbolFontAlphabet{\scr}{rsfs}
\DeclareSymbolFontAlphabet{\scr}{rsfs}

\DeclareFontFamily{U}{mathx}{\hyphenchar\font45}
\DeclareFontShape{U}{mathx}{m}{n}{
      <5> <6> <7> <8> <9> <10>
      <10.95> <12> <14.4> <17.28> <20.74> <24.88>
      mathx10
      }{}
\DeclareSymbolFont{mathx}{U}{mathx}{m}{n}
\DeclareFontSubstitution{U}{mathx}{m}{n}
\DeclareMathAccent{\wcheck}{0}{mathx}{"71}

\DeclareMathOperator{\Aut}{Aut}
\DeclareMathOperator{\codim}{codim}

\DeclareMathOperator{\Hom}{Hom}

\DeclareMathOperator{\img}{img}
\DeclareMathOperator{\Pic}{Pic}
\DeclareMathOperator{\rank}{rank}

\DeclareMathOperator{\reg}{reg}

\DeclareMathOperator{\sing}{sing}
\DeclareMathOperator{\Spec}{Spec}
\DeclareMathOperator{\Sym}{Sym}
\DeclareMathOperator{\supp}{supp}
\DeclareMathOperator{\tor}{tor}

\newcommand{\sA}{\scr{A}}
\newcommand{\sB}{\scr{B}}
\newcommand{\sC}{\scr{C}}
\newcommand{\sD}{\scr{D}}
\newcommand{\sE}{\scr{E}}
\newcommand{\sF}{\scr{F}}
\newcommand{\sG}{\scr{G}}
\newcommand{\sH}{\scr{H}}
\newcommand{\sHom}{\scr{H}\negthinspace om}

\newcommand{\sJ}{\scr{J}}

\newcommand{\sL}{\scr{L}}

\newcommand{\sO}{\scr{O}}

\newcommand{\sQ}{\scr{Q}}

\newcommand{\sS}{\scr{S}}
\newcommand{\sT}{\scr{T}}

\newcommand{\sW}{\scr{W}}

\newcommand{\cM}{\mathcal M}
\newcommand{\cN}{\mathcal N}

\newcommand{\bA}{\mathbb{A}}
\newcommand{\bB}{\mathbb{B}}
\newcommand{\bC}{\mathbb{C}}
\newcommand{\bD}{\mathbb{D}}
\newcommand{\bE}{\mathbb{E}}
\newcommand{\bF}{\mathbb{F}}
\newcommand{\bG}{\mathbb{G}}

\newcommand{\bL}{\mathbb{L}}

\newcommand{\bN}{\mathbb{N}}

\newcommand{\bP}{\mathbb{P}}
\newcommand{\bQ}{\mathbb{Q}}
\newcommand{\bR}{\mathbb{R}}

\newcommand{\bX}{\mathbb{X}}

\newcommand{\bZ}{\mathbb{Z}}

\theoremstyle{plain}
\newtheorem{thm}{Theorem}[section]

\newtheorem{cor}[thm]{Corollary}
\newtheorem{defn}[thm]{Definition}
\newtheorem{fact}[thm]{Fact}
\newtheorem{lem}[thm]{Lemma}

\newtheorem{prop}[thm]{Proposition}
\newtheorem{setup}[thm]{Setup}

\theoremstyle{remark}

\newtheorem{claim}[thm]{Claim}
\newtheorem{c-n-d}[thm]{Claim and Definition}

\newtheorem{example}[thm]{Example}

\newtheorem{notation}[thm]{Notation}

\newtheorem{rem}[thm]{Remark}

\newtheorem*{rem-nonumber}{Remark}
\newtheorem{setting}[thm]{Setting}

\numberwithin{equation}{thm}

\setlist[enumerate]{label=(\thethm.\arabic*), before={\setcounter{enumi}{\value{equation}}}, after={\setcounter{equation}{\value{enumi}}}}

\newcommand{\into}{\hookrightarrow}

\newcommand{\wtilde}{\widetilde}
\newcommand{\what}{\widehat}

\newcommand\CounterStep{\addtocounter{thm}{1}\setcounter{equation}{0}}

\newcommand{\factor}[2]{\left. \raise 2pt\hbox{$#1$} \right/\hskip -2pt\raise -2pt\hbox{$#2$}}
\newcommand{\Preprint}[1]{#1}
\newcommand{\Publication}[1]{}

\newcommand{\subversionInfo}{}
\newcommand{\svnid}[1]{}
\newcommand{\approvals}[2][Approval]{}
\renewcommand{\phi}{\varphi}

\usepackage{tikz}
\usepackage{tikz-cd}
\tikzset{commutative diagrams/arrow style=Latin Modern}

\author{Stefan Kebekus} %
\address{Stefan Kebekus, Mathematisches Institut, Albert-Ludwigs-Universität
  Freiburg, Ernst-Zermelo-Straße 1, 79104 Freiburg im Breisgau, Germany \&
  Freiburg Institute for Advanced Studies (FRIAS), Freiburg im Breisgau,
  Germany} %
\email{\href{mailto:stefan.kebekus@math.uni-freiburg.de}{stefan.kebekus@math.uni-freiburg.de}} %
\urladdr{\url{https://cplx.vm.uni-freiburg.de}} %

\author{Christian Schnell} %
\address{Christian Schnell, Department of Mathematics, Stony Brook University,
  Stony Brook, NY 11794-3651, U.S.A.} %
\email{\href{mailto:christian.schnell@stonybrook.edu}{christian.schnell@stonybrook.edu}} %
\urladdr{\url{https://www.math.stonybrook.edu/~cschnell}}



\title[Extending holomorphic forms from the regular locus of a complex
space]{Extending holomorphic forms from the regular locus of a complex space to
  a resolution of singularities}%
\date{\today}

\makeatletter
\hypersetup{
  pdfauthor={\authors},
  pdftitle={\@title},
  pdfsubject={\@subjclass},
  pdfkeywords={\@keywords},
  pdfstartview={Fit},
  pdfpagelayout={TwoColumnRight},
  pdfpagemode={UseOutlines},
  bookmarks,
  colorlinks,
  linkcolor=linkblue,
  citecolor=linkred,
  urlcolor=linkred
}
\makeatother

\newcommand{\Diff}{\operatorname{Diff}}
\newcommand{\MHM}{\operatorname{MHM}}
\newcommand{\Mmod}{\mathcal{M}}
\newcommand{\Nmod}{\mathcal{N}}
\newcommand{\HM}{\operatorname{HM}}
\newcommand{\Dbcoh}{\operatorname{D}_{\mathit{coh}}^{\mathit{b}}}
\newcommand{\Dtcoh}[1]{\operatorname{D}_{\mathit{coh}}^{#1}}
\newcommand{\Dbc}{\operatorname{D}_{\mathit{c}}^b}
\newcommand{\derR}{\mathbf{R}}
\newcommand{\decal}[1]{\lbrack #1 \rbrack}

\newcommand{\abs}[1]{\lvert #1 \rvert}
\newcommand{\dbar}{\bar{∂}}
\newcommand{\dz}{\mathit{dz}}

\DeclareMathOperator{\Ch}{Ch}
\newcommand{\dK}{\ensuremath{d_\text{\rm Kähler}}}
\newcommand{\drefl}{\ensuremath{d_\text{\rm refl}}}
\DeclareMathOperator{\DR}{DR}
\DeclareMathOperator{\GR}{GR}
\DeclareMathOperator{\gr}{gr}
\DeclareMathOperator{\h}{h}
\DeclareMathOperator{\rat}{rat}
\DeclareMathOperator{\restr}{restriction}
\DeclareMathOperator{\Supp}{Supp}

\newcommand{\define}[1]{\emph{#1}}

\newcommand{\argbl}{-}

\newcommand{\chapref}[1]{\hyperref[#1]{Chapter~\ref*{#1}}}
\newcommand{\lemmaref}[1]{\hyperref[#1]{Lemma~\ref*{#1}}}
\newcommand{\parref}[1]{\hyperref[#1]{Section~\ref*{#1}}}
\newcommand{\theoremref}[1]{\hyperref[#1]{Theorem~\ref*{#1}}}
\newcommand{\definitionref}[1]{\hyperref[#1]{Definition~\ref*{#1}}}
\newcommand{\propositionref}[1]{\hyperref[#1]{Proposition~\ref*{#1}}}
\newcommand{\conjectureref}[1]{\hyperref[#1]{Conjecture~\ref*{#1}}}
\newcommand{\corollaryref}[1]{\hyperref[#1]{Corollary~\ref*{#1}}}
\newcommand{\exampleref}[1]{\hyperref[#1]{Example~\ref*{#1}}}
\newcommand{\exerciseref}[1]{\hyperref[#1]{Exercise~\ref*{#1}}}
\newcommand{\factref}[1]{\hyperref[#1]{Fact~\ref*{#1}}}
\newcommand{\claimref}[1]{\hyperref[#1]{Claim~\ref*{#1}}}
\newcommand{\remarkref}[1]{\hyperref[#1]{Remark~\ref*{#1}}}
\newcommand{\settingref}[1]{\hyperref[#1]{Setting~\ref*{#1}}}
\newcommand{\appendixref}[1]{\hyperref[#1]{Appendix~\ref*{#1}}}

\theoremstyle{plain}
\newtheorem{propdef}[thm]{Proposition and Definition}
\theoremstyle{remark}
\newtheorem*{note}{Note}

\begin{document}

\maketitle
\approvals[Approval for Abstract]{Christian & yes \\Stefan & yes}
\begin{abstract}
We investigate under what conditions holomorphic forms defined on the regular
locus of a reduced complex space extend to holomorphic (or logarithmic) forms on
a resolution of singularities.  We give a simple necessary and sufficient
condition for this, whose proof relies on the Decomposition Theorem and Saito's
theory of mixed Hodge modules.  We use it to generalize the theorem of
Greb-Kebekus-Kovács-Peternell to complex spaces with rational singularities, and
to prove the existence of a functorial pull-back for reflexive differentials on
such spaces.  We also use our methods to settle the ``local vanishing
conjecture'' proposed by Mustaţă, Olano, and Popa.
\end{abstract}

\tableofcontents

\phantomsection\addcontentsline{toc}{part}{Introduction}
%
%
\svnid{$Id: S01-intro.tex 272 2020-01-21 14:29:32Z kebekus $}

\section{Overview of the paper}
\subversionInfo
\label{sec:sintro}

\subsection{Extension of holomorphic forms}
\approvals{Christian & yes \\Stefan & yes}

This paper is about the following ``extension problem'' for holomorphic
differential forms on complex spaces.  Let $X$ be a reduced complex space, and
let $r \colon \wtilde{X} → X$ be a resolution of singularities.  Which
holomorphic $p$-forms on the regular locus $X_{\reg}$ extend to holomorphic
$p$-forms on the complex manifold $\wtilde{X}$?  Standard facts about resolution
of singularities imply that the answer is independent of the choice of
resolution.  (If the exceptional locus of $r$ is a normal crossing divisor $E$,
one can also ask for an extension with at worst logarithmic poles along $E$.)

The best existing result concerning this problem is due to Greb, Kebekus,
Kovács, and Peternell \cite[Thm.~1.4]{GKKP11}.  They show that if $X$ underlies
a normal algebraic variety with Kawamata log terminal (=klt) singularities, then
all $p$-forms on $X_{\reg}$ extend to $\wtilde{X}$, for every $0 ≤ p ≤ \dim X$.
Their theorem has many applications, including hyperbolicity of moduli, the
structure of minimal varieties with trivial canonical class, the nonabelian
Hodge correspondence for singular spaces, and quasi-étale uniformisation.
\parref{ssec:app} recalls some of these in more detail and gives references.

In this paper, we use the Decomposition Theorem and Saito's theory of mixed
Hodge modules to solve the extension problem in general.  Our main result is a
simple necessary and sufficient condition for a holomorphic $p$-form on
$X_{\reg}$ to extend to a holomorphic (or logarithmic) $p$-form on $\wtilde{X}$.
One surprising consequence is that the extension problem for forms of a given
degree also controls what happens for forms of smaller degrees.  Another
consequence is that if $X$ is a complex space with rational singularities, then
all $p$-forms on $X_{\reg}$ extend to $\wtilde{X}$, for every $0 ≤ p ≤ \dim X$.
This result is a crucial step in the recent work of Bakker and Lehn \cite{BL18}
on the global moduli theory of symplectic varieties.

\subsection{Main result}
\approvals{Christian & yes \\Stefan & yes}

Let $X$ be a reduced complex space of pure dimension $n$.  It is well-known that
a holomorphic $n$-form $α ∈ H⁰(X_{\reg}, Ω^n_X)$ extends to a holomorphic
$n$-form on any resolution of singularities of $X$ if and only if
$α Λ \overline{α}$ is locally integrable on $X$.  Griffiths
\cite[§IIa]{griffiths-abel} gave a similar criterion for extension of $p$-forms
in terms of integrals over $p$-dimensional analytic cycles in $X$, but his
condition is not easy to verify in practice.  Our first main result is the
following intrinsic description of those holomorphic forms on $X_{\reg}$ that
extend holomorphically to one (and hence any) resolution of singularities.

\begin{thm}[Holomorphic forms]\label{thm:extension-Kahler}
  Let $X$ be a reduced complex space of pure dimension $n$, and
  $r \colon \wtilde{X} → X$ a resolution of singularities.  A holomorphic
  $p$-form $α ∈ H⁰(X_{\reg}, Ω^p_X)$ extends to a holomorphic $p$-form on
  $\wtilde{X}$ if, and only if, for every open subset $U ⊆ X$, and for every
  pair of Kähler differentials $β ∈ H⁰(U, Ω^{n-p}_X)$ and
  $γ ∈ H⁰(U, Ω^{n-p-1}_X)$, the holomorphic $n$-forms $α Λ β$ and $d α Λ γ$ on
  $U_{\reg}$ extend to holomorphic $n$-forms on $r^{-1}(U)$.
\end{thm}

Our proof of this result is based on the Decomposition Theorem for Hodge
modules.  We think that it would also be very interesting to have an analytic
proof, in terms of $L²$-Hodge theory for the $\dbar$-operator.  The following
analogue of \theoremref{thm:extension-Kahler} for forms with logarithmic poles
needs some additional results about mixed Hodge modules.  Recall that a
resolution of singularities $r \colon \wtilde{X} → X$ of a complex space is
called a \emph{(strong) log resolution} if the $r$-exceptional set is a divisor
with (simple) normal crossings on $\wtilde{X}$.

\begin{thm}[Logarithmic forms]\label{thm:extension-Kahler-log}
  Let $X$ be a reduced complex space of pure dimension $n$, and
  $r \colon \wtilde{X} → X$ a log resolution of singularities with exceptional
  divisor $E ⊆ X$.  A holomorphic $p$-form $α ∈ H⁰(X_{\reg}, Ω^p_X)$ extends to
  a holomorphic section of the bundle $Ω_{\wtilde{X}}^p(\log E)$ on $\wtilde{X}$
  if, and only if, for every open subset $U ⊆ X$, and for every pair of Kähler
  differentials $β ∈ H⁰(U, Ω^{n-p}_X)$ and $γ ∈ H⁰(U, Ω^{n-p-1}_X)$, the
  holomorphic $n$-forms $α Λ β$ and $d α Λ γ$ on $U_{\reg}$ extend to
  holomorphic sections of the bundle $Ω_{\wtilde{X}}^n(\log E)$ on $r^{-1}(U)$.
\end{thm}

\subsection{Consequences}
\approvals{Christian & yes \\Stefan & yes}

The extension problem for holomorphic (or logarithmic) forms on a complex space
$X$ is of course closely related to the singularities of $X$.  Since there might
not be any global $p$-forms on $X_{\reg}$, the effect of the singularities is
better captured by the following local version of the problem.  Given a
resolution of singularities $r \colon \wtilde{X} → X$ of a reduced complex space
$X$, and an arbitrary open subset $U ⊆ X$, which holomorphic $p$-forms on
$U_{\reg}$ extend to holomorphic $p$-forms on $r^{-1}(U)$?  If
$j \colon X_{\reg} ↪ X$ denotes the embedding of the regular locus, this amounts
to asking for a description of the subsheaf
$r_* Ω^p_{\wtilde{X}} ↪ j_* Ω^p_{X_{\reg}}$.  This subsheaf is $𝒪_X$-coherent
(by Grauert's theorem) and independent of the choice of resolution (because any
two resolutions are dominated by a common third).  If the exceptional locus of
$r$ is a normal crossing divisor $E$, one can also ask for a description of the
subsheaf $r_* Ω^p_{\wtilde{X}}(\log E) ↪ j_* Ω^p_{X_{\reg}}$, which has similar
properties.

\begin{example}
  When $X$ is reduced and irreducible, it is easy to see that
  $r_* 𝒪_{\wtilde{X}} ↪ j_* 𝒪_{X_{\reg}}$ is an isomorphism if and only if
  $\dim X_{\sing} ≤ \dim X-2$.  (Use the normalisation of $X$.)
\end{example}

One consequence of \theoremref{thm:extension-Kahler} is that the extension
problem for holomorphic forms of a given degree also controls what happens for
all forms of smaller degree.

\begin{thm}[Extension for $p$-forms]\label{thm:main-new}
  Let $X$ be a reduced and irreducible complex space.  Let
  $r \colon \wtilde{X} → X$ be any resolution of singularities, and
  $j \colon X_{\reg} ↪ X$ the inclusion of the regular locus.  If the morphism
  $r_* Ω^k_{\wtilde{X}} ↪ j_* Ω^k_{X_{\reg}}$ is an isomorphism for some
  $0 ≤ k ≤ \dim X$, then $\dim X_{\sing} ≤ \dim X-2$, and
  $r_* Ω^p_{\wtilde{X}} ↪ j_* Ω^p_{X_{\reg}}$ is an isomorphism for every
  $0 ≤ p ≤ k$.
\end{thm}

An outline of the proof can be found in \parref{par:outline} below.  The key
idea is to use the Decomposition Theorem \cite{BBD, Saito:HodgeModules}, in
order to relate the coherent $𝒪_X$-module $r_* Ω^p_{\wtilde{X}}$ to the
intersection complex of $X$, viewed as a polarisable Hodge module.  In
\appendixref{app:cones}, we look at the example of cones over smooth projective
varieties; it gives a hint that the extension problem for all $p$-forms should
be governed by what happens for $n$-forms.  When $X$ is normal, an equivalent
formulation of \theoremref{thm:main-new} is that, if the coherent $𝒪_X$-module
$r_* Ω^k_{\wtilde{X}}$ is reflexive for some $k ≤ \dim X$, then
$r_* Ω^p_{\wtilde{X}}$ is reflexive for every $p ≤ k$.

\begin{note}
  One can easily generalise \theoremref{thm:main-new} to arbitrary reduced
  complex spaces.  The precise (but somewhat cumbersome) statement is that if
  the morphism $r_* Ω^k_{\wtilde{X}} ↪ j_* Ω^k_{X_{\reg}}$ is an isomorphism for
  some $k ≥ 0$, and if $Z ⊆ X$ denotes the union of all the irreducible
  components of $X$ of dimension $≥ k$, then $\dim Z_{\sing} ≤ k-2$, and the
  restriction to $Z$ of the morphism $r_* Ω^p_{\wtilde{X}} ↪ j_* Ω^p_{X_{\reg}}$
  is an isomorphism for every $0 ≤ p ≤ k$.  The reason is that the irreducible
  components of $X$ are separated in any resolution of singularities, and so one
  can simply apply \theoremref{thm:main-new} one component at a time.
\end{note}

We also establish a version of \theoremref{thm:main-new} with log poles, by
adapting the techniques in the proof to a certain class of mixed Hodge modules.

\begin{thm}[Extension for log $p$-forms]\label{thm:main-log}
  Let $X$ be a reduced and irreducible complex space.  Let
  $r \colon \wtilde{X} → X$ be a log resolution with exceptional divisor
  $E ⊆ \wtilde{X}$, and $j \colon X_{\reg} ↪ X$ the inclusion of the regular
  locus.  If the morphism $r_* Ω^k_{\wtilde{X}}(\log E) ↪ j_* Ω^k_{X_{\reg}}$ is
  an isomorphism for some $0 ≤ k ≤ \dim X$, then $\dim X_{\sing} ≤ \dim X-2$,
  and $r_* Ω^p_{\wtilde{X}}(\log E) ↪ j_* Ω^p_{X_{\reg}}$ is an isomorphism for
  every $0 ≤ p ≤ k$.
\end{thm}

By a result of Kovács, Schwede, and Smith \cite[Thm.~1]{KSS10}, a complex
algebraic variety $X$ that is normal and Cohen-Macaulay has Du Bois
singularities if and only if $r_* ω_{\wtilde{X}}(E)$ is a reflexive $𝒪_X$-module
for some log resolution $r \colon \wtilde{X} → X$.  We think that it would be
interesting to know the precise relationship between Du Bois singularities and
the extension problem for logarithmic $n$-forms.  The tools we develop for the
proof of \theoremref{thm:main-log} also lead to a slightly better answer in the
case of holomorphic forms of degree $\dim X - 1$.

\begin{thm}[Extension for $(n-1)$-forms]\label{thm:extension-n-1}
  Let $X$ be a reduced and irreducible complex space.  Let
  $r \colon \wtilde{X} → X$ be a log resolution with exceptional divisor
  $E ⊆ \wtilde{X}$, and $j \colon X_{\reg} ↪ X$ the inclusion of the regular
  locus.  If the natural morphism $r_* Ω^n_{\wtilde{X}} ↪ j_* Ω^n_{X_{\reg}}$ is
  an isomorphism, where $n = \dim X$, then the two morphisms
  $$
  r_* \bigl( Ω^{n-1}_{\wtilde{X}}(\log E)(-E) \bigr) %
  ↪ r_* Ω^{n-1}_{\wtilde{X}} %
  ↪ j_* Ω^{n-1}_{X_{\reg}}
  $$
  are also isomorphisms.
\end{thm}

\subsection{Rational and weakly rational singularities}
\label{sssec:wratl}
\approvals{Christian & yes \\Stefan & yes}

An important class of singular spaces where \theoremref{thm:main-new} applies is
normal complex spaces with rational singularities.  Recall that $X$ has
\define{rational singularities} if the following equivalent conditions hold.  We
refer to \cite[§5.1]{KM98} for details.\CounterStep

\begin{enumerate}
\item\label{il:A1} $X$ is normal, and if $r \colon \wtilde{X} → X$ is any
  resolution of singularities, then $Rⁱ r_* 𝒪_{\wtilde{X}} = 0$ for every
  $i ≥ 1$.
  
\item\label{il:B1} $X$ is Cohen-Macaulay and $ω_X^{\GR} = ω_X$.
  
\item\label{il:C1} $X$ is Cohen-Macaulay and $ω_X^{\GR}$ is reflexive.
\end{enumerate}

Here $ω_X^{\GR} = r_* ω_{\wtilde{X}}$ is sometimes called the
\emph{Grauert-Riemenschneider sheaf}, because it appears in the
Grauert-Riemenschneider vanishing theorem.  In view of Condition~\ref{il:C1}, we
say that a normal space $X$ has \define{weakly rational singularities} if the
Grauert-Riemenschneider sheaf $ω_X^{\GR}$ is reflexive.  With this notation,
\theoremref{thm:main-new} has the following immediate corollary.

\begin{cor}[Extension in the case of rational singularities]\label{cor:rational}
  Let $X$ be a normal complex space with weakly rational singularities, and let
  $r \colon \wtilde{X} → X$ be a resolution of singularities.  Then every
  holomorphic form defined on $X_{\reg}$ extends uniquely to a holomorphic form
  on $\wtilde{X}$.  \qed
\end{cor}

\begin{rem}
  Rational singularities are weakly rational by definition.  In particular,
  recall from \cite[Thm.~5.22 and references there]{KM98} that klt spaces have
  rational (and hence weakly rational) singularities.  For algebraic klt
  varieties, the extension result was shown previously in
  \cite[Thm.~1.4]{GKKP11}.
\end{rem}

As we will see in \parref{sec:extshf-ng}, having weakly rational singularities
turns out to be equivalent to the collection of inequalities
\[
  \dim \Supp Rⁱ r_* 𝒪_{\wtilde{X}} ≤ \dim X - 2 - i
  \quad \text{for every $i ≥ 1$.}
\]
One can also describe the class of weakly rational singularities in more
analytic terms: a normal complex space $X$ of dimension $n$ has weakly rational
singularities if and only if, for every open subset $U ⊆ X$ and every
holomorphic $n$-form $ω ∈ H⁰(U_{\reg}, Ω^n_{U_{\reg}})$, the $(n,n)$-form
$ω Λ \overline{ω}$ on $U_{\reg}$ is locally integrable on all of $U$.
\appendixref{sec:wratlSings} discusses examples and establishes elementary
properties of this class of singularities.

\subsection{Local vanishing conjecture}
\label{ssec:localVanishing}
\approvals{Christian & yes \\Stefan & yes}

The methods developed in this paper also settle the ``local vanishing
conjecture'' proposed by Mustaţă, Olano, and Popa
\cite[Conj.~A]{Mustata+Olano+Popa:LocalVanishing}.  The original conjecture
contained the assumption that $X$ is a normal algebraic variety with rational
singularities.  In fact, the weaker assumption
$R^{\dim X-1} r_* 𝒪_{\wtilde{X}} = 0$ is sufficient.

\begin{thm}[Local vanishing]\label{thm:MOP}
  Let $X$ be a reduced and irreducible complex space of dimension $n$.  Let
  $r \colon \wtilde{X} → X$ be a log resolution, with exceptional divisor
  $E ⊆ \wtilde{X}$.  If $R^{n-1} r_* 𝒪_{\wtilde{X}} = 0$, then
  $R^{n-1} r_* Ω¹_{\wtilde{X}}(\log E) = 0$.
\end{thm}

As shown in \cite{Mustata+Olano+Popa:LocalVanishing}, this result has
interesting consequences for the Hodge filtration on the complement of a
hypersurface with at worst rational singularities.

\subsection{Functorial pull-back}
\label{sec:fpb}
\approvals{Christian & yes \\Stefan & yes}

One can interpret \theoremref{thm:main-new} as saying that any differential form
$σ ∈ H⁰ \bigl( X_{\reg},\, Ω¹_{X_{\reg}} \bigr) = H⁰ \bigl( X,\, Ω^{[1]}_X
\bigr)$ induces a \emph{pull-back form}
$\wtilde{σ} ∈ H⁰ \bigl( \wtilde{X},\, Ω¹_{\wtilde{X}} \bigr)$.  More generally,
we show that pull-back exists for reflexive differentials and arbitrary
morphisms between varieties with rational singularities.  The paper
\cite{MR3084424} discusses these matters in detail.

\begin{thm}[Functorial pull-back for reflexive differentials]\label{thm:pullBack}
  Let $f \colon X → Y$ be any morphism between normal complex spaces with
  rational singularities.  Write $Ω^{[p]}_X := \bigl(Ω^p_X\bigr)^{**}$, ditto
  for $Ω^{[p]}_Y$.  Then there exists a pull-back morphism
  $$
  \drefl f \colon f^* Ω^{[p]}_Y → Ω^{[p]}_X,
  $$
  uniquely determined by natural universal properties.
\end{thm}

We refer to \theoremref{thm:PB-thmA} and \parref{ssec:pbform} for a precise
formulation of the ``natural universal properties'' mentioned in
\theoremref{thm:pullBack}.  In essence, it is required that the pull-back
morphisms agree with the pull-back of Kähler differentials wherever this makes
sense, and that they satisfy the composition law.

\begin{note}
  \theoremref{thm:pullBack} applies to morphisms $X → Y$ whose image is entirely
  contained in the singular locus of $Y$.  Taking the inclusion of the singular
  set for a morphism, \theoremref{thm:pullBack} implies that every differential
  form on $Y_{\reg}$ induces a differential form on every stratum on the
  singularity stratification.
\end{note}

\subsubsection{$\h$-differentials}
\approvals{Christian & yes \\Stefan & yes}

One can also reformulate \theoremref{thm:pullBack} in terms of
$\h$-differentials; these are obtained as the sheafification of Kähler
differentials with respect to the $\h$-topology on the category of complex
spaces, as introduced by Voevodsky.  We refer the reader to \cite{MR3272910} and
to the survey \cite{MR3580792} for a gentle introduction to these matters.
Using the description of $\h$-differentials found in \cite[Thm.~1]{MR3272910},
the following is an immediate consequence of \theoremref{thm:pullBack}.

\begin{cor}[$\h$-differentials on spaces with rational singularities]
  Let $X$ be a normal complex space with rational singularities.  Write
  $Ω^{[p]}_X := \bigl(Ω^p_X\bigr)^{**}$.  Then, $\h$-differentials and reflexive
  differentials agree: $Ω^p_{\h}(X) = Ω^{[p]}_X(X)$.  \qed
\end{cor}

The sheaf $Ω^p_{\h}$ of $\h$-differentials appears under a different name in the
work of Barlet, \cite{MR3899534}, who describes it in analytic terms (``integral
dependence equations for differential forms'') as a subsheaf of $Ω^{[p]}_X$ and
relates it to the normalised Nash transform.

\subsection{Applications}
\label{ssec:app}
\approvals{Christian & yes \\Stefan & yes}

The main results of this paper allow to study (sheaves of) reflexive
differential forms on singular spaces, by pulling them back to a resolution of
singularities.  At times, this allows to prove results of Hodge-theoretic
flavour in settings where classic Hodge-theory is not readily available.  We
give two immediate application of \theoremref{thm:main-new}, which can be proven
in just a few lines, following \cite[Sect.~6 and 7]{GKKP11} verbatim.

\begin{thm}[Closedness of forms and Bogomolov-Sommese vanishing]\label{thm:BSV}
  Let $X$ be a normal complex projective variety.  If $ω_X^{\GR}$ is reflexive,
  then any holomorphic differential form on $X_{\reg}$ is closed.  If
  $𝒜 ⊆ Ω^{[p]}_X$ is an invertible subsheaf, then $κ(𝒜) ≤ p$.  \qed
\end{thm}

\begin{thm}[Lipman-Zariski conjecture for weakly rational complex spaces]\label{thm:LZ}
  Let $X$ be a normal complex space where $ω_X^{\GR}$ is reflexive.  If the
  tangent sheaf $𝒯_X$ is locally free, then $X$ is smooth.  \qed
\end{thm}

To illustrate the range of applicability, we mention some recent (and much more
substantial) results that rely on the previously known extension theorem for klt
spaces by Greb, Kebekus, Kovács, and Peternell.

\begin{itemize}
\item The standard conjectures of minimal model theory predict that the minimal
  model of any projective manifold $X$ of Kodaira dimension $κ(X)=0$ is a
  singular space with vanishing first Chern class.  A series of papers
  \cite{GKP16, Dru16, MR3988092, MR3953506} extended the classic
  Beauville-Bogomolov Decomposition Theorem to the singular setting.  A recent
  paper \cite{MR3959097} studies degenerations of hyper-Kähler manifolds, using
  the minimal model program to reduce any degeneration to (singular) ``Kulikov
  type form''.

\item A series of papers \cite{GKPT19b, GKPT17, GKPT19} extends the classic
  non-abelian Hodge correspondence from Kähler manifolds to singular spaces.
  For klt spaces, this relates representations of the fundamental group with
  (singular) Higgs sheaves and yields new results on quasi-étale uniformisation,
  \cite{GKP13, LT18, GKT16}.
  
\item The extension results are used analysis in the study of (singular)
  Kähler-Einstein metrics, \cite{MR3283927, MR3996488}
  
\item The extension result is used in holomorphic dynamics and foliations, for
  the classification of foliations \cite{MR3273645, MR3033631}, but also in the
  study of compactifications of Drinfeld half-spaces over a finite field,
  \cite{MR3902334}.
\end{itemize}

\subsubsection*{Very recent work}

Our generalisation of the extension theorem to (possibly non-algebraic) complex
spaces with rational singularities (in \corollaryref{cor:rational} above) is
used in a crucial way in the recent work of Bakker and Lehn \cite{BL18} on
global moduli for symplectic varieties, where a ``symplectic variety'' is a
normal Kähler space $X$ with a nondegenerate holomorphic $2$-form on $X_{\reg}$
that extends holomorphically to any resolution of singularities.

\subsection{Earlier results}
\approvals{Christian & yes \\Stefan & yes}

As mentioned above, \theoremref{thm:main-new} was already known for spaces with
Kawamata log terminal (=klt) singularities, where $r_* ω_{\wtilde{X}}$ is
reflexive by definition \cite{GKK08, GKKP11}.  If one is only interested in
$p$-forms of small degree (compared to $\dim X$), there are earlier results of
Steenbrink-van Straten \cite{SS85} and Flenner \cite{Flenner88}.  In the special
case where $p=1$, Graf-Kovács relate the extension problem to the notion of Du
Bois singularities \cite{MR3247804}.  For morphisms between varieties with klt
singularities, the existence of a pull-back functor was shown in
\cite{MR3084424}.

We refer the reader to the paper \cite{GKKP11} or to the survey
\cite[§4]{Keb13a} for a more detailed introduction, and for remarks on the
history of the problem.  The book \cite[§8.5]{MR3057950} puts the results into
perspective.

\subsection{Acknowledgements}
\approvals{Christian & yes \\Stefan & yes}
  
Both authors would like to thank Mark de Cataldo, Bradley Drew, Philippe
Eyssidieux, Annette Huber-Klawitter, Florian Ivorra, Mihai Păun, and Mihnea Popa
for helpful discussions, and the Freiburg Institute for Advanced Studies for
support and for its stimulating atmosphere.  We thank Fabio Bernasconi for
showing us a counterexample to the extension theorem in positive characteristic,
and two anonymous referees for careful reading and for suggestions that help to
improve the exposition of this paper.

\subsubsection{Funding Information}
\approvals{Christian & yes \\Stefan & yes}
  
Stefan Kebekus gratefully acknowledges the support through a senior fellowship
of the Freiburg Institute of Advanced Studies (FRIAS).

During the preparation of this paper, Christian Schnell was supported by a
Mercator Fellowship from the Deutsche Forschungsgemeinschaft (DFG), by research
grant DMS-1404947 from the National Science Foundation (NSF), and by the Kavli
Institute for the Physics and Mathematics of the Universe (IPMU) through the
World Premier International Research Center Initiative (WPI initiative), MEXT,
Japan.

%
%
\svnid{$Id: S02-outline.tex 269 2020-01-20 11:28:53Z kebekus $}

\section{Techniques and main ideas}
\label{par:outline}
\subversionInfo
\approvals{Christian & yes \\Stefan & yes}

In this section, we sketch some of the ideas that go into the proof of
\theoremref{thm:main-new}.  The one-line summary is that it is a consequence of
the Decomposition Theorem \cite{BBD,Saito:MixedHodgeModules}.
\appendixref{app:cones} contains a short section on cones over projective
manifolds that illustrates the extension problem in a particularly transparent
case and explains why one might even expect a result such as
\theoremref{thm:main-new} to hold true.

\subsection{First proof of Theorem~\ref*{thm:main-new}}
\approvals{Christian & yes \\Stefan & yes}

We actually give two proofs for \theoremref{thm:main-new}.  The first proof (in
\parref{sec:pf-1-x}) relies on \theoremref{thm:extension-Kahler}, which
characterises those holomorphic forms on the regular locus of a complex space
that extend holomorphically to any resolution of singularities.  This proof is
very short and, shows clearly why the extension problem for $k$-forms also
controls the extension problem for $(k-1)$-forms (and hence for all forms of
smaller degrees).

\subsection{Second proof of Theorem~\ref*{thm:main-new}}
\label{sec:second-proof}
\approvals{Christian & yes \\Stefan & yes}

To illustrate the main ideas and techniques used in this paper, we are now going
to describe a second, more systematic proof for \theoremref{thm:main-new}.  It
is longer, and covers only the case where $k = n$, but it has the advantage of
producing a stronger result that has other applications (such as the proof of
the local vanishing conjecture).  We hope that the description below will make
it clear why the Decomposition Theorem is useful in studying the extension
problem for holomorphic forms.

\subsection*{Setup}
\approvals{Christian & yes \\Stefan & yes}

We fix a reduced and irreducible complex space $X$ of dimension $n$, and a
resolution of singularities $r \colon \wtilde{X} → X$.  We denote by
$j \colon X_{\reg} ↪ X$ the embedding of the set of regular points, and assume
that the natural morphism $r_* Ω_{\wtilde{X}}^n ↪ j_* Ω_{X_{\reg}}^n$ is an
isomorphism.  This means concretely that, locally on $X$, holomorphic $n$-forms
extend from the regular locus to the resolution.  Rather than using the given
resolution $\wtilde{X}$ to show that $p$-forms extend, we are going to prove
directly that the natural morphism $r_* Ω^p_{\wtilde{X}} ↪ j_* Ω^p_{X_{\reg}}$
is an isomorphism for every $p ∈ \{0, 1, …, n\}$.  This is a statement about $X$
itself, because the subsheaf $r_* Ω^p_{\wtilde{X}}$ does not depend on the
choice of resolution, as we have seen in the introduction.

\begin{note}
  Using independence of the resolution, we may assume without loss of generality
  that the resolution $r \colon \wtilde{X} → X$ is projective, and an
  isomorphism over $X_{\reg}$.  Such resolutions exist for every reduced complex
  space by \cite[Thm.~10.7]{BM96}.
\end{note}

\subsection*{Criteria for extension}
\approvals{Christian & yes \\Stefan & yes}

The first idea in the proof of \theoremref{thm:main-new} is to use
duality.\footnote{For the sake of exposition, we work directly on $X$ in this
  section.  In the actual proof, we only use duality for coherent sheaves on
  complex manifolds, after locally embedding $X$ into a complex manifold.} Let
$ω_X^• ∈ \Dbcoh(𝒪_X)$ denote the dualizing complex of $X$; on the
$n$-dimensional complex manifold $\wtilde{X}$, one has
$ω_{\wtilde{X}}^• ≅ ω_{\wtilde{X}} \decal{n}$.  The dualizing complex gives rise
to a simple numerical criterion for whether sections of a coherent $𝒪_X$-module
extend uniquely over $X_{\sing}$.  Indeed, \propositionref{prop:ccor-ng} -- or
rather its generalisation to singular spaces -- says that sections of a coherent
$𝒪_X$-module $ℱ$ extend uniquely over $X_{\sing}$ if and only if
\[
  \dim \bigl( X_{\sing} ∩ \Supp R^k \sHom_{𝒪_X}(ℱ, ω_X^•) \bigr) ≤ -(k+2) \quad
  \text{for every $k ∈ ℤ$.}
\]
When the support of $ℱ$ has pure dimension $n$, as is the case for the
$𝒪_X$-module $r_* Ω^p_{\wtilde{X}}$ that we are interested in, this amounts to
the following two conditions:
\begin{enumerate}
\item $\dim X_{\sing} ≤ n-2$
  
\item $\dim \Supp R^k \sHom_{𝒪_X}(ℱ, ω_X^{•}) ≤ -(k+2)$ for every $k ≥ -n+1$
\end{enumerate}

Unfortunately, there is no good way to compute the dual complex of
$r_* Ω^p_{\wtilde{X}}$.  But if we work instead with the entire complex
$\derR r_* Ω^p_{\wtilde{X}}$, things get better: Grothendieck duality
\cite{MR0308439}, applied to the proper holomorphic mapping
$r \colon \wtilde{X} → X$, yields
\begin{equation}\label{eq:outline-duality}
  \derR \sHom_{𝒪_X} \Bigl( \derR r_* Ω^p_{\wtilde{X}}, ω_X^• \Bigr) %
  \overset{\text{duality}}{≅} \derR r_* \derR \sHom \Bigl( Ω^p_{\wtilde{X}}, ω_{\wtilde{X}} \decal{n} \Bigr) %
  ≅ \derR r_* Ω^{n-p}_{\wtilde{X}} \decal{n}.
\end{equation}
In \propositionref{prop:reflexive-complex-ng}, we prove the following variant of
the criterion for section extension: if $K ∈ \Dbcoh(𝒪_X)$ is a complex with
$ℋ^j K = 0$ for $j < 0$, and if
\begin{equation}\label{eq:outline-inequality}
  \dim \bigl( X_{\sing} ∩ \Supp R^k \sHom_{𝒪_X}(K, ω_X^•) \bigr)
  ≤ -(k+2) \quad \text{for every $k ∈ ℤ$},
\end{equation}
then sections of the coherent $𝒪_X$-module $ℋ⁰ K$ extend uniquely over
$X_{\sing}$.  This observation transforms the problem of showing that sections
of $r_* Ω^p_{\wtilde{X}}$ extend uniquely over $X_{\sing}$ into the problem of
showing that
\[
  \dim \Supp R^k r_* Ω^{n-p}_{\wtilde{X}} ≤ n-2-k \quad \text{for every
    $k ≥ 1$.}
\]
In summary, we see that a good upper bound for the dimension of the support of
$R^k r_* Ω^{n-p}_{\wtilde{X}}$ would be enough to conclude that $p$-forms
extend.  Or, to put it more simply, ``vanishing implies extension''.

\subsection*{Hodge modules and the Decomposition Theorem}
\approvals{Christian & yes \\Stefan & yes}

The problem with the approach outlined above is that the complex
$\derR r_* Ω^{n-p}_{\wtilde{X}}$ has too many potentially nonzero cohomology
sheaves, which makes it hard to prove the required vanishing.  For example, if
the preimage of a singular point $x ∈ X_{\sing}$ is a divisor in the resolution
$\wtilde{X}$, then $R^{n-1} r_* Ω^{n-p}_{\wtilde{X}}$ might be supported at $x$,
violating the inequality in \eqref{eq:outline-inequality}.  Since we are not
assuming that the singularities of $X$ are klt, we also do not have enough
information about the fibres of $r \colon \wtilde{X} → X$ to prove vanishing by
restricting to fibres as in \cite[§18]{GKKP11}.
 
The second idea in the proof, which completely circumvents this problem, is to
relate the $𝒪_X$-module $r_* Ω^p_{\wtilde{X}}$ to the intersection complex of
$X$, viewed as a polarisable Hodge module\footnote{Since the intersection
  complex is intrinsic to $X$, this also serves to explain once again why the
  $𝒪_X$-module $r_* Ω^p_{\wtilde{X}}$ does not depend on the choice of
  resolution.}.  In the process, we make use of the Decomposition Theorem.
Roughly speaking, the Decomposition Theorem decomposes the push-forward of the
constant sheaf into a ``generic part'' (that only depends on $X$) and a
``special part'' (that is affected by the positive-dimensional fibres of $r$).
The upshot is that the generic part carries all the relevant information, and
that the positive-dimensional fibres of $r$ are completely irrelevant for the
extension problem.  To be more precise, the Decomposition Theorem for the
projective morphism $r$, together with Saito's formalism of Hodge modules, leads
to a (non-canonical) decomposition
\begin{equation}\label{eq:decomposition}
  \derR r_* Ω^p_{\wtilde{X}} ≅ K_p ⊕ R_p
\end{equation}
into two complexes $K_p, R_p ∈ \Dbcoh(𝒪_X)$ with the following properties:
\begin{enumerate}
\item\label{il:emerson} The support of $R_p$ is contained in the singular locus
  $X_{\sing}$.
  
\item\label{il:lake} One has $ℋ^k K_p = 0$ for $k ≥ n-p+1$.
  
\item\label{il:palmer} The complexes $K_p$ and $K_{n-p}$ are related by
  Grothendieck duality in the same way that the complexes
  $\derR r_* Ω^p_{\wtilde{X}}$ and $\derR r_* Ω^{n-p}_{\wtilde{X}}$ are related
  in \eqref{eq:outline-duality}.  More precisely, one has
  $\derR \sHom_{𝒪_X}(K_p, ω_X^•) ≅ K_{n-p} \decal{n}$.
\end{enumerate}

\subsection*{An improved criterion}
\approvals{Christian & yes \\Stefan & yes}

As an immediate consequence of the decomposition in \eqref{eq:decomposition}, we
obtain a decomposition of the $0$-th cohomology sheaves
\[
  r_* Ω^p_{\wtilde{X}} ≅ ℋ⁰ K_p ⊕ ℋ⁰ R_p.
\]
Because $ℋ⁰ R_p$ is supported inside $X_{\sing}$, whereas $Ω^p_{\wtilde{X}}$ is
torsion free, we deduce that $ℋ⁰ R_p = 0$, and hence that
$r_* Ω^p_{\wtilde{X}} ≅ ℋ⁰ K_p$.  According to the criterion for section
extension in \propositionref{prop:reflexive-complex-ng}, now applied to the
complex $K_p$, all we therefore need for sections of $r_* Ω^p_{\wtilde{X}}$ to
extend uniquely over $X_{\sing}$ is to establish the collection of inequalities
\begin{equation}\label{eq:outline-inequalities-Kp}
  \dim \bigl( X_{\sing} ∩ \Supp ℋ^k K_{n-p} \bigr) ≤ n-2-k
  \quad \text{for all $k ∈ ℤ$.}
\end{equation}
Property~\ref{il:lake} makes this a much more manageable task, compared to the
analogous problem for the original complex $\derR r_* Ω^{n-p}_{\wtilde{X}}$.  We
stress that, except in the case $p = n$, these inequalities are stronger than
asking that sections of $r_* Ω^p_{\wtilde{X}}$ extend uniquely over $X_{\sing}$.

\subsection*{The case of isolated singularities}
\approvals{Christian & yes \\Stefan & yes}

We conclude this outline with a brief sketch how
\eqref{eq:outline-inequalities-Kp} is proved in the case of isolated
singularities.  In \parref{par:pure}, we more or less reduce the general case to
this special case by locally cutting with hypersurfaces; note that this works
because we are proving a stronger statement than just extension of $p$-forms.

Because of Property~\ref{il:lake}, we have $ℋ^k K_{n-p} = 0$ for $k ≥ p+1$.
Since $\dim X_{\sing} = 0$, the inequality in \eqref{eq:outline-inequalities-Kp}
is therefore true by default as long as $p ≤ n-2$.  In this way, we recover the
result of Steenbrink and van Straten \cite[Thm.~1.3]{SS85} mentioned in the
introduction: on an $n$-dimensional complex space with isolated singularities,
$p$-forms extend for every $p ≤ n-2$.  This only leaves two cases, namely
$p = n-1$ and $p = n$.

The case $p = n$ is covered by the assumption that $n$-forms extend.  We have
$K_n ≅ ℋ⁰ K_n ≅ r_* Ω_{\wtilde{X}}^n$, and sections of $r_* Ω_{\wtilde{X}}^n$
extend uniquely over $X_{\sing}$.  Because of the isomorphism
$\derR \sHom_{𝒪_X} \bigl( K_n, ω_X^• \bigr) ≅ K_0 \decal{n}$,
\propositionref{prop:ccor-ng} gives us the desired inequalities
\[
  \dim \bigl( X_{\sing} ∩ \Supp ℋ^k K_0 \bigr) %
  = \dim \Bigl( X_{\sing} ∩ \Supp R^{k-n} \sHom_{𝒪_X}(K_n, ω_X^•) \Bigr) %
  ≤ n-k-2
\]
for every $k ∈ ℤ$.

In the other case $p = n-1$, the inequalities in
\eqref{eq:outline-inequalities-Kp} are easily seen to be equivalent to the
single vanishing $ℋ^{n-1} K_1 = 0$.  Using the fact that $ℋ^k K_0 = 0$ for
$k ≥ n-1$, one shows that the $𝒪_X$-module $ℋ^{n-1} K_1$ is a quotient of the
(constructible) $0$-th cohomology sheaf of the intersection complex of $X$.  But
the intersection complex is known to be concentrated in strictly negative
degrees, and therefore $ℋ^{n-1} K_1 = 0$.

%
%
\svnid{$Id: S03-notation.tex 269 2020-01-20 11:28:53Z kebekus $}

\section{Conventions}
\subversionInfo
\label{sect:notation}

\subsection{Global conventions}
\approvals{Christian & yes \\Stefan & yes}

Throughout this paper, all complex spaces are assumed to be countable at
infinity.  All schemes and algebraic varieties are assumed to be defined over
the field of complex numbers.  We follow the notation used in the standard
reference books \cite{Ha77, CAS}.  In particular, varieties are assumed to be
irreducible, and the support of a coherent sheaf $ℱ$ on $X$ is a closed subset
of $X$, with the induced reduced structure.  For clarity, we will always say
explicitly when a complex space needs to be reduced, irreducible, or of pure
dimension.

\subsection{$𝒟$-modules}
\approvals{Christian & yes \\Stefan & yes}

Unless otherwise noted, we use left $𝒟$-modules throughout this paper.  This
choice agrees with the notation of the paper \cite{saito-vanishing}, which we
will frequently cite.  It is, however, incompatible with the conventions of the
reference papers \cite{Saito:HodgeModules, Saito:MixedHodgeModules} and of the
survey \cite{Schnell14a} that use right $𝒟$-modules throughout.  We refer the
reader to \cite[§A.5]{saito-vanishing}, where the conversion rules for left and
right $𝒟$-modules are recalled.

\subsection{Complexes}
\approvals{Christian & yes \\Stefan & yes}

Let $K$ be a complex of sheaves of Abelian groups on a topological space, for
example a complex of sheaves of $𝒪_X$-modules (or $𝒟_X$-modules) on a complex
manifold $X$.  We use the notation $ℋ^j K$ for the $j$-th cohomology sheaf of
the complex.  We use the notation $K \decal{n}$ for the shift of $K$.  We have
$ℋ^j K \decal{n} = ℋ^{j+n} K$, and all differentials in the shifted complex are
multiplied by $(-1)^n$.

\subsection{The dualizing complex}
\approvals{Christian & yes \\Stefan & yes}

If $X$ is any complex space, we write $ω_X^• ∈ \Dbcoh(𝒪_X)$ for the dualizing
complex as introduced by Ramis and Ruget, see
\cite[VII~Thm.~2.6]{BS76}\Preprint{\ and the original reference
  \cite{MR0279338}}.  Given a complex of $𝒪_X$-modules $K ∈ \Dbcoh(𝒪_X)$ with
bounded coherent cohomology, we call the complex
$\derR \sHom_{𝒪_X}(K, ω_X^•) ∈ \Dbcoh(𝒪_X)$ the \define{dual complex} of $K$.

\begin{note}
  When $X$ is a complex manifold of pure dimension, one has
  $ω_X^• ≅ ω_X \decal{\dim X}$.
\end{note}

\subsection{Reflexive sheaves on normal spaces}
\label{ssec:ccorS}
\approvals{Christian & yes \\Stefan & yes}

Let $X$ be a normal complex space, and $ℱ$ a coherent $𝒪_X$-module.  Recall that
$ℱ$ is called \define{reflexive} if the natural morphism from $ℱ$ to its double
dual $ℱ^{**} := \sHom_{𝒪_X} \bigl( \sHom_{𝒪_X}(ℱ, 𝒪_X), 𝒪_X \bigr)$ is an
isomorphism.  The following notation will be used.

\begin{notation}[Reflexive hull]
  Given a normal complex space $X$ and a coherent sheaf $ℱ$ on $X$, write
  $Ω^{[p]}_X := \bigl(Ω^p_X \bigr)^{**}$, $ℱ^{[m]} := \bigl(ℰ^{⊗ m} \bigr)^{**}$
  and $\det ℱ := \bigl( Λ^{\rank ℰ} ℱ \bigr)^{**}$.  Given any morphism
  $f \colon Y → X$ of normal complex spaces, write $f^{[*]} ℱ := (f^* ℱ)^{**}$,
  etc.  Ditto for quasi-projective varieties.
\end{notation}

\phantomsection\addcontentsline{toc}{part}{Mixed Hodge modules}
%
%
\svnid{$Id: S04-MHM.tex 270 2020-01-20 11:33:46Z kebekus $}

\section{Mixed Hodge modules}
\subversionInfo

\subsection{Mixed Hodge modules}
\approvals{Christian & yes \\Stefan & yes}

For the convenience of the reader, we briefly recall a number of facts
concerning mixed Hodge modules, and lay down the notation that will be used
throughout.  In a nutshell, a (mixed) Hodge module is something like a variation
of (mixed) Hodge structure with singularities, in the sense that the vector
bundles with connection (respectively locally constant sheaves) in the
definition of a variation of Hodge structure are replaced by regular holonomic
$𝒟$-modules (respectively perverse sheaves).  The standard references for mixed
Hodge modules are the original papers by Saito \cite{Saito:HodgeModules,
  Saito:MixedHodgeModules}.  The survey articles \cite{Saito:Introduction,
  Saito:Theory, Schnell14a} review some aspects of the theory in a smaller
number of pages.  A good reference for $ 𝒟$-modules is the book \cite{HTT}.  We
consider the following setting throughout the present section.

\begin{setting}\label{set:MHM}
  Assume that a complex manifold $Y$ of pure dimension $d$ and a
  graded-polarisable mixed Hodge module $M$ on $Y$ are given.
\end{setting}

More precisely, a \emph{polarisable Hodge module} $M$ of weight $w$ on $Y$ has
three components: a regular holonomic left $𝒟_Y$-module $\Mmod$, called the
\emph{underlying $𝒟$-module}; an increasing good filtration $F_{•} \Mmod$
by coherent $𝒪_Y$-modules, called the \emph{Hodge filtration}; and a perverse
sheaf of $ℚ$-vector spaces $\rat M$, called the \emph{underlying perverse
  sheaf}.  These are subject to a number of conditions, including the existence
of a polarisation, which together ensure that $M$ is determined by finitely many
polarisable variations of Hodge structure of weight $w - \dim Z_j$ on locally
closed submanifolds $Z_j ⊆ Y$.  A \emph{graded-polarisable mixed Hodge
  module} $M$ on $Y$ is an object of the same kind, but with an additional
increasing filtration $W_{•} M$, called the \emph{weight filtration}, such
that each subquotient
\[
  \gr_ℓ^W \!  M := \factor{W_ℓ M}{W_{ℓ-1}} M
\]
is a polarisable Hodge module of weight $ℓ$.  The \define{support} of $M$,
denoted by $\Supp M$, is by definition the support of the $𝒟_Y$-module $\Mmod$
(or, equivalently, of the perverse sheaf $\rat M$).

\begin{notation}[Category of mixed Hodge modules]
  In \settingref{set:MHM}, we denote by $\MHM(Y)$ the Abelian category of
  \define{graded-polarisable mixed Hodge modules} on $Y$, and by $\HM(Y, w)$ the
  Abelian category of \define{polarisable Hodge modules} of weight $w$.
\end{notation}

\begin{note}
  In \settingref{set:MHM}, we have
  \[
    \gr_ℓ^W \!  M := \factor{W_ℓ M}{W_{ℓ-1}} M ∈ \HM(Y, ℓ), \quad \text{for
      every } ℓ ∈ ℤ.
  \]
  Conversely, every polarisable Hodge module $N ∈ \HM(Y, w)$ may be viewed as a
  graded-polarisable mixed Hodge module $N$ with $W_{w-1} N = 0$ and
  $W_w N = N$.
\end{note}

\subsubsection{Tate twist}
\approvals{Christian & yes \\Stefan & yes}

Maintain \settingref{set:MHM}.  Given any integer $k ∈ ℤ$, define
$ℚ(k) = (2 π i)^k ℚ ⊆ ℂ$.  The \define{Tate twist} $M(k)$ is the mixed Hodge
module whose underlying perverse sheaf is $ℚ(k) ⊗ \rat M$, whose underlying
filtered $𝒟_Y$-module is $(\Mmod, F_{• - k} \Mmod)$, and whose weight filtration
is given by $W_ℓ \, M(k) = W_{ℓ+2k} M$.  When $M$ is pure of weight $w$, it
follows that $M(k)$ is again pure of weight $w-2k$.

\subsubsection{Decomposition by strict support}
\approvals{Christian & yes \\Stefan & yes}

In \settingref{set:MHM}, one says that the mixed Hodge module $M$ has
\define{strict support} if the support of every nontrivial subquotient of $M$ is
equal to $\Supp M$.  Ditto for perverse sheaves and regular holonomic
$𝒟_Y$-modules.  Note that the strict support property is generally \emph{not}
preserved by restriction to open subsets; for example, $\Supp M$ may be globally
irreducible, but locally reducible.  We use the symbol $\HM_X(Y, w)$ to denote
the Abelian category of polarisable Hodge modules on $Y$ of weight $w$ with
strict support $X$; this is a full subcategory of $\HM(Y,w)$.

If $M$ is a polarisable Hodge module, then $M$ has strict support if and only if
the support of every nontrivial subobject (or quotient object) is equal to
$\Supp M$; the reason is that polarisable Hodge modules are semisimple
\cite[Cor.~5.2.13]{Saito:HodgeModules}.  By definition, every polarisable Hodge
module admits, on every open subset of $X$, a \define{decomposition by strict
  support} as a (locally finite) direct sum of polarisable Hodge modules with
strict support \cite[§5.1.6]{Saito:HodgeModules}.

\subsubsection{Weight filtration and dual module}
\approvals{Christian & yes \\Stefan & yes}

In \settingref{set:MHM}, we write $M' = 𝔻 M ∈ \MHM(Y)$ for the \define{dual
  mixed Hodge module}.  This is again a graded-polarisable mixed Hodge module
\cite[Prop.~2.6]{Saito:MixedHodgeModules}, with the property that
\[
  𝔻 \bigl( \gr_ℓ^W \!  M \bigr) ≅ \gr_{-ℓ}^W 𝔻 M.
\]
In particular, if $M$ is pure of weight $ℓ$, then $𝔻 M$ is again pure of weight
$-ℓ$.  The underlying perverse sheaf $\rat M'$ is isomorphic to the Verdier dual
\cite[Def.~4.5.2]{HTT} of $\rat M$.  The regular holonomic left $𝒟_Y$-module
$(\Mmod', F_• \Mmod')$ underlying $M' = 𝔻 M$ is isomorphic to the holonomic dual
\cite[Def.~2.6.1]{HTT}
\[
  R^d \sHom_{𝒟_Y} \bigl( ω_Y ⊗_{𝒪_Y} \Mmod, 𝒟_Y \bigr)
\]
of the regular holonomic left $𝒟_Y$-module $\Mmod$.

\subsubsection{The de Rham complex}
\approvals{Christian & yes \\Stefan & yes}

In \settingref{set:MHM}, the complex of sheaves of $ℂ$-vector spaces
\[
  \DR(\Mmod) = \Bigl\lbrack \Mmod → Ω¹_Y ⊗_{𝒪_Y} \Mmod → \dotsb → Ω^d_Y ⊗_{𝒪_Y}
  \Mmod \Bigr\rbrack \decal{d},
\]
concentrated in degrees $-d, …, 0$ is called the \define{de Rham complex} of
$\Mmod$.  Since $\Mmod$ is a regular holonomic $𝒟_Y$-module, the de Rham complex
$\DR(\Mmod)$ has constructible cohomology sheaves, and is in fact a perverse
sheaf on $Y$ by a theorem of Kashiwara \cite[Thm.~4.6.6]{HTT}.  In particular,
it is always \define{semiperverse}, which means concretely that
\[
  \dim \Supp ℋ^j \DR(\Mmod) ≤ -j \quad \text{for every $j ∈ ℤ$.}
\]
The perverse sheaf $\rat M$ and the de Rham complex of $\Mmod$ are related
through an isomorphism $ℂ ⊗_ℚ \rat M ≅ \DR(\Mmod)$ that is part of the data of a
mixed Hodge module.

\subsubsection{Subquotients of the de Rham complex}
\label{ssec:sqdR}
\approvals{Christian & yes \\Stefan & yes}

Assume \settingref{set:MHM}.  The filtration $F_• \Mmod$ induces an increasing
filtration on the de Rham complex by\CounterStep
\begin{equation}\label{eq:Aw}
  F_p \DR(\Mmod) = \Bigl\lbrack F_p \Mmod → Ω¹_Y ⊗_{𝒪_Y} F_{p+1} \Mmod → \dotsb
  → Ω^d_Y ⊗_{𝒪_Y} F_{p+d}\Mmod \Bigr\rbrack \decal{d}.
\end{equation}
The $p$-th subquotient of this filtration is the complex of $𝒪_Y$-modules
\begin{equation}\label{eq:Ax}
  \gr_p^F \DR(\Mmod) = \Bigl\lbrack \gr_p^F \!  \Mmod → Ω¹_Y ⊗_{𝒪_Y} \gr_{p+1}^F
  \!  \Mmod → \dotsb → Ω^d_Y ⊗_{𝒪_Y} \gr_{p+d}^F \!  \Mmod \Bigr\rbrack
  \decal{d}.
\end{equation}
For a more detailed discussion of these complexes, see for example
\cite[§7]{saito-vanishing}.  The following simple lemma will be useful later.

\begin{lem}\label{lem:5-2}
  In \settingref{set:MHM}, if $\gr_p^F \DR(\Mmod)$ is acyclic for every $p ≤ m$,
  then $F_{m + \dim Y} \Mmod = 0$.
\end{lem}
\begin{proof}
  Since $F_• \Mmod$ is a good filtration, there is, at least locally on $Y$, an
  integer $p_0$ such that $F_{p_0} \Mmod = 0$ and, hence, $\gr_p^F \Mmod = 0$
  for every $p ≤ p_0$.  To show that $F_{m + d} \Mmod = 0$, it is therefore
  enough to prove that $\gr_p^F \Mmod = 0$ for every $p ≤ m+d$.  Because
  $\gr_p^F \DR(\Mmod)$ is acyclic for $p ≤ m$, this follows from \eqref{eq:Ax}
  by induction on $p ≥ p_0$.
\end{proof}

\subsection{Duality}
\approvals{Christian & yes \\Stefan & yes}

Next, we review how the duality functor for mixed Hodge modules affects the
subquotients of the de Rham complex.  The following nontrivial result by Saito
shows that the dual complex of $\gr_p^F \DR(\Mmod)$ is nothing but
$\gr_{-p}^F \DR(\Mmod')$.

\begin{prop}[Duality, mixed case]\label{prop:duality-mixed}
  Assume \settingref{set:MHM}.  Then,
  \[
    \derR \sHom_{𝒪_Y} \Bigl( \gr_p^F \DR(\Mmod), ω^•_Y \Bigr) ≅ \gr_{-p}^F
    \DR(\Mmod') \quad \text{for every $p ∈ ℤ$},
  \]
  where $(\Mmod', F_• \Mmod')$ is the filtered $𝒟_Y$-module underlying
  $M' = 𝔻 M$.
\end{prop}
\begin{proof}
  This is proved in \cite[§2.4.3]{Saito:HodgeModules}; see also
  \cite[Lem.~7.4]{saito-vanishing}.  The crucial point in the proof is that
  $\gr_•^F \Mmod$ is a Cohen-Macaulay module over $\gr_•^F 𝒟_Y$, due to the fact
  that $(\Mmod, F_• \Mmod)$ underlies a mixed Hodge module.
\end{proof}

In the special case where $M$ is a polarisable Hodge module, the de Rham complex
is self-dual, up to a shift in the filtration.  Duality therefore relates
different subquotients of $\DR(\Mmod)$, in a way that will be very useful for
the proof of \theoremref{thm:main-new}.

\begin{cor}[Duality, pure case]\label{cor:duality-pure}
  Let $M ∈ \HM(Y, w)$ be a polarisable Hodge module of weight $w$ on a complex
  manifold $Y$.  Any polarisation on $M$ induces an isomorphism
  \[
    \derR \sHom_{𝒪_Y} \Bigl( \gr_p^F \DR(\Mmod), ω^•_Y \Bigr) ≅ \gr_{-p-w}^F
    \DR(\Mmod) \quad\text{for every $p ∈ ℤ$.}
  \]
\end{cor}
\begin{proof}
  A polarisation on $M$ induces an isomorphism $𝔻 M ≅ M(w)$
  \cite[§5.2.10]{Saito:HodgeModules}, and therefore an isomorphism
  $(\Mmod', F_• \Mmod') ≅ (\Mmod, F_{•-w} \Mmod)$.  Now apply
  \propositionref{prop:duality-mixed}.
\end{proof}

The following proposition contains an acyclicity criterion for subquotients of
the de Rham complex, involving both the weight filtration $W_• M$ and the Hodge
filtration $F_• \Mmod$.

\begin{prop}[Acyclic subquotients]\label{prop:DR-acyclic}
  Assume \settingref{set:MHM}.  If $w$, $c ∈ ℤ$ are such that $W_{w-1} M = 0$
  and $F_{c-1} \Mmod = 0$, then $\gr_p^F \DR(\Mmod)$ is acyclic unless
  $c-d ≤ p ≤ d-w-c$.
\end{prop}
\begin{proof}
  Since $F_{c-1} \Mmod = 0$ and $d = \dim Y$, a look at the formula
  \eqref{eq:Ax} for the $p$-th subquotient of $\DR(\Mmod)$ reveals that
  $\gr_p^F \DR(\Mmod) = 0$ for $p ≤ c-1-d$.  The other inequality is going to
  follow by duality.  Let us first consider the pure case, meaning that
  $M ∈ \HM(Y, w')$ is a polarisable Hodge module of weight $w'$.  By
  \corollaryref{cor:duality-pure}, we have
  \[
    \gr_p^F \DR(\Mmod) ≅ \derR \sHom_{𝒪_Y} \Bigl( \gr_{-p-w'}^F \DR(\Mmod),
    ω^•_Y \Bigr)
  \]
  and since the complex on the right-hand side is acyclic for $-p-w' ≤ c-1-d$,
  we get the result when $M$ is pure.  The general case follows from this by
  considering the subquotients of the weight filtration $W_• M$.
\end{proof}

\propositionref{prop:DR-acyclic} is especially useful when combined with the
following general fact, which an easy consequence of the filtration $F_• \Mmod$
being exhaustive.

\begin{prop}\label{prop:DR-qiso}
  Let $Y$ be a complex manifold.  Let $(\Mmod, F_• \Mmod)$ be a coherent
  $𝒟_Y$-module with a good filtration.  If $\gr_p^F \DR(\Mmod)$ is acyclic for
  all $p ≥ p_0+1$, then the inclusion $F_{p_0} \DR(\Mmod) ↪ \DR(\Mmod)$ is a
  quasi-isomorphism.  \qed
\end{prop}

\subsection{Direct images and the Decomposition Theorem}
\label{sec:push-forward}
\approvals{Christian & yes \\Stefan & yes}

Let $f \colon X → Y$ be a projective holomorphic mapping between two complex
manifolds, and let $M ∈ \MHM(X)$ be a graded-polarisable mixed Hodge module on
$X$.  One of the most important results in Saito's theory is that, in this
setting, one can define a direct image functor, compatible with the direct image
functor for perverse sheaves and filtered $𝒟$-modules, and that the $i$-th
higher direct image $Hⁱ f_* M$ is again a graded-polarisable mixed Hodge module
on $Y$.  In this section, we briefly review this result and its implications for
the underlying filtered $𝒟_X$-module $(\Mmod, F_• \Mmod)$ and the de Rham
complex $\DR(\Mmod)$.

\subsubsection{Filtered $𝒟$-modules and strictness}
\approvals{Christian & yes \\Stefan & yes}

Let $X$ be a complex manifold.  Following Saito, we denote by $\Dbcoh F(𝒟_X)$
the derived category of (certain cohomologically bounded and coherent complexes
of) filtered $𝒟_X$-modules, as defined in \cite[§2.1.15]{Saito:HodgeModules}.
The category of filtered $𝒟_X$-modules is only an exact category, but it embeds
into the larger Abelian category of graded $R_F 𝒟_X$-modules, where
\[
  R_F 𝒟_X = \bigoplus_{p ∈ ℕ} F_p 𝒟_X
\]
is the Rees algebra of $𝒟_X$ with respect to the order filtration.  The
embedding takes a coherent filtered $𝒟_X$-module $(\Mmod, F_• \Mmod)$ to the
associated Rees module
\[
  R_F \Mmod = \bigoplus_{p ∈ ℤ} F_p \Mmod,
\]
which is coherent over $R_F 𝒟_X$.  Let $\Dbcoh G(R_F 𝒟_X)$ be the derived
category of (cohomologically bounded and coherent complexes of) graded
$R_F 𝒟_X$-modules.  Then the Rees module construction gives an equivalence of
categories
\[
  \Dbcoh F(𝒟_X) ≅ \Dbcoh G(R_F 𝒟_X),
\]
according to \cite[Prop.~2.1.16]{Saito:HodgeModules}.  The cohomology modules of
an object in $\Dbcoh F(𝒟_X)$ are therefore in general not filtered
$𝒟_X$-modules, but graded $R_F 𝒟_X$-modules.

\begin{defn}[Strictness]
  A graded $R_F 𝒟_X$-module is called \define{strict} if it is isomorphic to the
  Rees module of a coherent filtered $𝒟_X$-module.  A complex
  $K ∈ \Dbcoh G(R_F 𝒟_X)$ is called \define{strict} if all of its cohomology
  modules $ℋ^j K$ are strict.
\end{defn}

The functor that takes a coherent filtered $𝒟_X$-module $(\Mmod, F_• \Mmod)$ to
the underlying $𝒟_X$-module $\Mmod$ extends uniquely to an exact functor
\[
  \Dbcoh G(R_F 𝒟_X) → \Dbcoh(𝒟_X).
\]
Indeed, if we denote by $z ∈ R_F 𝒟_X$ the degree-one element obtained from
$1 ∈ F_1 𝒟_X$, then the functor is simply the derived tensor product with
$R_F 𝒟_X / (1-z) R_F 𝒟_X$.  Similarly, the functor that takes a coherent
filtered $𝒟_X$-module $(\Mmod, F_• \Mmod)$ to the coherent graded
$\Sym 𝒯_X$-module $\gr_•^F \Mmod$ extends uniquely to an exact functor
\[
  \gr_•^F \colon \Dbcoh G(R_F 𝒟_X) → \Dbcoh G(\Sym 𝒯_X).
\]
This time, the functor is given by the derived tensor product with
$R_F 𝒟_X / z R_F 𝒟_X$.  Lastly, for every $p ∈ ℤ$, the functor that takes a
coherent filtered $𝒟_X$-module $(\Mmod, F_• \Mmod)$ to the complex of coherent
$𝒪_X$-modules $\gr_p^F \DR(\Mmod)$ extends uniquely to an exact functor
\[
  \gr_p^F \DR \colon \Dbcoh G(R_F 𝒟_X) → \Dbcoh(𝒪_X).
\]
Indeed, by \cite[Prop.~2.2.10]{Saito:HodgeModules}, the de Rham functor (which
Saito denotes by the symbol $\wtilde{DR}$) defines an equivalence of categories
between $\Dbcoh F(𝒟_X)$ and the derived category of filtered differential
complexes $\Dbcoh F^f(𝒪_X, \Diff)$, and $\gr_p^F$ of a filtered differential
complex is by construction a (cohomologically bounded and coherent) complex of
$𝒪_X$-modules \cite[§2.2.4]{Saito:HodgeModules}.

\subsubsection{Direct image functor for filtered $𝒟$-modules}
\approvals{Christian & yes \\Stefan & yes}

Now suppose that $f \colon X → Y$ is a proper holomorphic mapping between
complex manifolds.  In this setting, one can construct a direct image functor
\[
  f_+ \colon \Dbcoh G(R_F 𝒟_X) → \Dbcoh G(R_F 𝒟_Y);
\]
see \cite[§2.3.5]{Saito:HodgeModules} for the precise definition.  This functor
is compatible with the functor $\gr_p^F \DR$ in the following manner
\cite[§2.3.7]{Saito:HodgeModules}.

\begin{prop}\label{prop:fl-grDR}
  Let $f \colon X → Y$ be a proper holomorphic mapping between complex
  manifolds.  For every $p ∈ ℤ$, one has a natural isomorphism of functors
  \[
    \derR f_* ◦ \gr_p^F \DR ≅ \gr_p^F \DR ◦ f_+
  \]
  as functors from $\Dbcoh G(R_F 𝒟_X)$ to $\Dbcoh(𝒪_Y)$.
\end{prop}
\begin{proof}
  By \cite[Lem.~2.3.6]{Saito:HodgeModules}, the de Rham functor exchanges the
  direct image functor $f_+ \colon \Dbcoh G(R_F 𝒟_X) → \Dbcoh G(R_F 𝒟_Y)$ and
  the direct image functor
  \[
    f_* \colon \Dbcoh F^f(𝒪_X, \Diff) → \Dbcoh F^f(𝒪_Y, \Diff)
  \]
  for filtered differential complexes.  But the latter commutes with taking
  $\gr_p^F$, as is clear from the construction in
  \cite[§2.3.7]{Saito:HodgeModules}.
\end{proof}

\begin{note}
  In the case of a single coherent filtered $𝒟_X$-module, this says that
  \[
    \derR f_* \gr_p^F \DR(\Mmod) ≅ \gr_p^F \DR \bigl( f_+(R_F \Mmod) \bigr),
  \]
  as objects of the derived category $\Dbcoh(𝒪_Y)$.
\end{note}

\subsubsection{Direct image theorem, pure case}
\approvals{Christian & yes \\Stefan & yes}

We now assume that the proper holomorphic mapping $f \colon X → Y$ is actually
projective.  Then we have the following important ``direct image theorem'' due
to Saito \cite[§5.3]{Saito:HodgeModules}.

\begin{thm}[Direct image theorem, pure case]\label{thm:direct-image}
  Let $f \colon X → Y$ be a projective morphism between complex manifolds, and
  let $ℓ ∈ H²(X, ℤ(1))$ be the first Chern class of a relatively ample line
  bundle.  If $M ∈ \HM(X, w)$ is a polarisable Hodge module $X$, then:
  \begin{enumerate}
  \item The complex $f_+(R_F \Mmod)$ is strict, and each $ℋⁱ f_+(R_F \Mmod)$ is
    the filtered $𝒟_Y$-module underlying a polarisable Hodge module
    $Hⁱ f_* M ∈ \HM(Y,w+i)$.
    
  \item\label{en:direct-image-2} For every $i ≥ 0$, the Lefschetz morphism
    \[
      ℓⁱ \colon H^{-i} f_* M → Hⁱ f_* M(i)
    \]
    is an isomorphism between Hodge modules of weight $w-i$.
    
  \item Any polarisation on $M$ induces a polarisation on $\bigoplus_i Hⁱ f_* M$
    in the Hodge-Lefschetz sense (= on primitive parts with respect to the
    action of $ℓ$).  \qed
  \end{enumerate}
\end{thm}

One consequence of \theoremref{thm:direct-image} is a version of the
Decomposition Theorem for those filtered $𝒟$-modules that underlie polarisable
Hodge modules.

\begin{cor}[Decomposition Theorem]\label{cor:decomposition-MF}
  Let $f \colon X → Y$ be a projective morphism between complex manifolds.  Let
  $M ∈ \HM(X,w)$ be a polarisable Hodge module on $X$, and let
  $M_i = Hⁱ f_* M ∈ \HM(Y,w+i)$.  Write $(\Mmod, F_• \Mmod)$ respectively
  $(\Mmod_i, F_• \Mmod_i)$ for the underlying filtered $𝒟$-modules.  Then
  \[
    f_+(R_F \Mmod) \simeq \bigoplus_{i ∈ ℤ} ℋⁱ f_+(R_F \Mmod) \, \decal{-i} ≅
    \bigoplus_{i ∈ ℤ} R_F \Mmod_i \, \decal{-i},
  \]
  in the derived category $\Dbcoh G(R_F 𝒟_Y)$.
\end{cor}
\begin{proof}
  The first isomorphism is a formal consequence of \ref{en:direct-image-2}.  The
  second isomorphism follows because the complex $f_+(R_F \Mmod)$ is strict.
\end{proof}

\subsubsection{Direct image theorem, mixed case}
\approvals{Christian & yes \\Stefan & yes}

In the case of mixed Hodge modules, there are some additional results, having to
do with the weight filtration.  We summarise them in the following theorem
\cite[Thm.~2.14 and Prop.~2.15]{Saito:MixedHodgeModules}.

\begin{thm}[Direct image theorem, mixed case]\label{thm:fl-mixed}
  Let $f \colon X → Y$ be a projective morphism between complex manifolds, and
  let $M ∈ \MHM(X)$ be a graded-polarisable mixed Hodge module on $X$.
  \begin{enumerate}
  \item\label{en:fl-mixed-SS} The complex $f_+(R_F \Mmod)$ is strict, and each
    $ℋⁱ f_+(R_F \Mmod)$ is the filtered $𝒟_Y$-module underlying a
    graded-polarisable mixed Hodge module $Hⁱ f_* M ∈ \MHM(Y)$.
    
  \item\label{en:fl-mixed-SS2} One has a convergent \emph{weight spectral
      sequence}
    \[
      E_1^{p,q} = H^{p+q} f_* \gr_{-p}^W M \Longrightarrow H^{p+q} f_* M,
    \]
    and each differential $d_1 \colon E_1^{p,q} → E_1^{p+1,q}$ is a morphism in
    $\HM(Y, q)$.
    
  \item The weight spectral sequence degenerates at $E_2$, and one has
    \[
      \gr_q^W H^{p+q} f_* M ≅ E_2^{p,q}\quad \text{ for every $p,q ∈ ℤ$.} \eqno\qed
    \]
  \end{enumerate}
\end{thm}

One can use this result to control the range in which the Hodge filtration on
the direct image of a graded-polarisable mixed Hodge module is nontrivial.

\begin{prop}\label{prop:Fc}
  Let $f \colon X → Y$ be a projective morphism between complex manifolds, and
  let $M ∈ \MHM(X)$ be a graded-polarisable mixed Hodge module on $X$.  Suppose
  that the underlying filtered $𝒟_X$-module $(\Mmod, F_• \Mmod)$ satisfies
  $F_{m-1} \Mmod = 0$.  Then one has
  \[
    F_{m+c-1} ℋⁱ f_+(R_F \Mmod) = 0
  \]
  for every $i ∈ ℤ$, where $c = \dim Y - \dim X$.
\end{prop}
\begin{proof}
  One can deduce this from the construction of the direct image functor in
  \cite[§2.3]{Saito:HodgeModules}.  Here we outline another proof based on
  \theoremref{thm:direct-image} and \theoremref{thm:fl-mixed}.

  We first deal with the case where $M ∈ \HM(X, W)$ is a polarisable pure Hodge
  module.  By \propositionref{prop:fl-grDR} and
  \corollaryref{cor:decomposition-MF}, we have for every $p ∈ ℤ$ an isomorphism
  \[
    \derR f_* \gr_p^F \DR(\Mmod) ≅ \gr_p^F \DR \bigl( f_+(R_F \Mmod) \bigr) %
    ≅ \bigoplus_{i ∈ ℤ} \gr_p^F \DR(\Mmod_i) \decal{-i},
  \]
  where $(\Mmod_i, F_• \Mmod_i)$ is the filtered $𝒟_Y$-module underlying
  $Hⁱ f_* M ∈ \HM(Y, w+i)$.  Since $F_{m-1} \Mmod = 0$, we get
  $\gr_p^F \DR(\Mmod) = 0$ for all $p ≤ m-1-\dim X$, and $\gr_p^F \DR(\Mmod_i)$
  is therefore acyclic as long as $p ≤ m-1-\dim X$.  According to
  \lemmaref{lem:5-2}, this is enough to conclude that
  $F_{m+c-1} \Mmod_i = F_{m-1-\dim X + \dim Y} \Mmod_i = 0$ for every $i ∈ ℤ$.
  
  Now suppose that $M ∈ \MHM(X)$ is a graded-polarisable mixed Hodge module.
  The underlying $𝒟_X$-module of the Hodge module $\gr_w^F M ∈ \HM(X, w)$ is
  $\gr_w^W \Mmod$, with the induced Hodge filtration; because
  $F_{m-1} \Mmod = 0$, we have $F_{m-1} \gr_w^F \Mmod = 0$.  Since we already
  have the result in the pure case, the assertion now follows by looking at the
  spectral sequence in \ref{en:fl-mixed-SS2}.
\end{proof}

\subsection{Non-characteristic restriction to hypersurfaces}
\label{sec:non-characteristic}
\approvals{Christian & yes \\Stefan & yes}

We briefly review the non-characteristic restriction of a mixed Hodge module to
a hypersurface.  For a more general discussion of non-characteristic
restriction, see \cite[§3.5]{Saito:HodgeModules} or \cite[§8]{saito-vanishing}.

\begin{defn}[Non-characteristic hypersurfaces]\label{def:nonchar}
  Let $X$ be a complex manifold, and let $D ⊆ X$ be a smooth hypersurface.  The
  inclusion $i_D \colon D ↪ X$ gives rise to the following morphisms between
  cotangent bundles:
  \begin{equation}\label{eq:diagram-D}
    \begin{tikzcd}
      (T^{\ast} X) ⨯_X D \rar \dar{p_1} & T^{\ast} D \\
      T^{\ast} X
    \end{tikzcd}
  \end{equation}
  Given a regular holonomic left $𝒟_X$-module $\Mmod$ on $X$, let
  $\Ch(\Mmod) ⊆ T^{\ast} X$ denote its characteristic variety.  We say that
  $D ⊆ X$ is \define{non-characteristic} for $\Mmod$ if $p_1^{-1} \Ch(\Mmod)$ is
  finite over its image in $T^{\ast} D$.
\end{defn}

\begin{note}
  As explained for example in \cite[§8]{saito-vanishing}, $D ⊆ X$ is
  non-characteristic for $\Mmod$ if and only if $D$ is transverse to every
  stratum in a Whitney stratification of $X$ that is adapted to the perverse
  sheaf $\DR(\Mmod)$.  In particular, generic hyperplane sections (in $ℙ^n$ or
  $ℂ^n$) are always non-characteristic.
\end{note}

The following result of Saito \cite[Lem.~2.25]{Saito:MixedHodgeModules} describes
what happens to mixed Hodge modules under non-characteristic restriction to smooth
hypersurfaces.

\begin{thm}[Restriction to non-characteristic hypersurfaces]\label{thm:restr}
  Let $X$ be a complex manifold, and let $M ∈ \MHM(X)$ be a graded-polarisable
  mixed Hodge module on $X$, with underlying filtered $𝒟_X$-module
  $(\Mmod, F_• \Mmod)$.  Suppose that $i_D \colon D ↪ X$ is a smooth
  hypersurface that is non-characteristic for $\Mmod$.  Then there is a
  graded-polarisable mixed Hodge module $H^{-1} i_D^* M ∈ \MHM(D)$, whose
  underlying filtered $𝒟_D$-module is isomorphic to
  \[
    \bigl( 𝒪_D ⊗_{i_D^{-1}𝒪_X} i_D^{-1} \Mmod, 𝒪_D ⊗_{i_D^{-1}𝒪_X} i_D^{-1} F_•
    \Mmod \bigr),
  \]
  and whose de Rham complex is quasi-isomorphic to
  \[
    i_D^{-1} \DR(\Mmod) \decal{-1}.
  \]
  Moreover, if $M$ is pure of weight $w$, then $H^{-1} i_D^* M$ is again pure of
  weight $w-1$.
\end{thm}

As the discussion in Saito's paper is rather brief, we include a sketch of the
proof of \theoremref{thm:restr} for the convenience of the reader.  It relies on
the following result of Saito \cite[Lem.~3.5.6]{Saito:HodgeModules} whose proof we
reproduce here.

\begin{lem}[Existence of $V$-filtration]\label{lem:saito-nonchar}
  In the setting of \definitionref{def:nonchar}, suppose that the smooth
  hypersurface $D ⊆ X$ is non-characteristic for $\Mmod$.  Then the rational
  V-filtration of $\Mmod$ relative to $D$ exists and is given by
  \[
    V^{α} \Mmod =
    \begin{cases}
      \Mmod & \text{for $α ≤ 0$,} \\
      𝒥_D^{⌈ α ⌉} \Mmod & \text{for $α ≥ 0$,}
    \end{cases}
  \]
  where $𝒥_D ⊆ 𝒪_X$ denotes the coherent ideal sheaf of $D$.
\end{lem}
\begin{proof}
  The problem is local, and after shrinking $X$, we may assume that
  $D = t^{-1}(0)$, where $t \colon X → ℂ$ is holomorphic and submersive.  We may
  also assume that we have a global holomorphic vector field $∂_t$ with the
  property that $\lbrack ∂_t, t \rbrack = 1$.  In this situation, the rational
  V-filtration is the unique exhaustive decreasing filtration $V^• \Mmod$,
  indexed discretely and left-continuously by the set of rational numbers, with
  the following properties:
  \begin{enumerate}
  \item Each $V^{α} \Mmod$ is coherent over $V⁰ 𝒟_X$, the $𝒪_X$-subalgebra of
    $𝒟_X$ preserving $𝒥_D$.
    
  \item One has $t · V^{α} \Mmod ⊆ V^{α+1} \Mmod$ and
    $∂_t · V^{α} \Mmod ⊆ V^{α-1} \Mmod$ for every $α ∈ ℚ$.
    
  \item For $α > -1$, multiplication by $t$ induces an isomorphism
    $V^{α} \Mmod ≅ V^{α+1} \Mmod$.
    
  \item The operator $t ∂_t - α$ acts nilpotently on
    $\gr_V^{α} \Mmod = V^{α} \Mmod / V^{>α} \Mmod$.
  \end{enumerate}
  If we define the filtration $V^• \Mmod$ as in the statement of the lemma, then
  the last three properties are immediate; the only thing we need to check is
  that $\Mmod$ itself is coherent over $V⁰ 𝒟_X$.  After choosing a good
  filtration $F_• \Mmod$, it is enough to show that $\gr_•^F \Mmod$ is coherent
  over $\gr_•^F V⁰ 𝒟_X$.  Note that $\gr_•^F \Mmod$ is always coherent over
  $\gr_•^F 𝒟_X ≅ \Sym 𝒯_X$.

  To prove the required coherence, we denote by $𝒯_{X/ℂ}$ the relative tangent
  sheaf, and by $T^{\ast}(X/ℂ)$ the relative cotangent bundle.  The fact that
  $t$ is submersive means that we have a surjective bundle morphism
  $T^{\ast} X → T^{\ast}(X/ℂ)$ on $X$; its restriction to $D$ is the horizontal
  arrow in \eqref{eq:diagram-D}.  By assumption, $p_1^{-1} \Ch(\Mmod)$ is finite
  over its image in $T^{\ast} D$, and because finiteness is an open condition,
  we can replace $X$ by a suitable open neighbourhood of $D$ and arrange that
  $\Ch(\Mmod)$ is actually finite over its image in $T^{\ast}(X/ℂ)$.  By
  definition of the characteristic variety, the support of the coherent sheaf on
  $T^{\ast} X$ corresponding to $\gr_•^F \Mmod$ is precisely $\Ch(\Mmod)$.
  Because push forward by finite holomorphic mappings preserves coherence, it
  follows that $\gr_•^F \Mmod$ is coherent over the subalgebra
  $\Sym 𝒯_{X/ℂ} ⊆ \Sym 𝒯_X$.  Now it is easy to see that
  \[
    \gr_•^F V⁰ 𝒟_X ≅ \Sym 𝒯_{X/ℂ} \, \lbrack t ∂_t \rbrack,
  \]
  and so $\gr_•^F \Mmod$ is coherent over this larger $𝒪_X$-algebra as well.
\end{proof}

We use the above description of the rational V-filtration to prove
\theoremref{thm:restr}.

\begin{proof}[Proof of Theorem~\ref*{thm:restr}]
  Since all the assertions are local on $X$, we may assume that $D = t^{-1}(0)$,
  where $t \colon X → ℂ$ is submersive.  We keep the notation introduced during
  the proof of \lemmaref{lem:saito-nonchar}.  Since $(\Mmod, F_• \Mmod)$
  underlies a mixed Hodge module, multiplication by $t$ induces an isomorphism
  between $F_p V^{α} \Mmod$ and $F_p V^{α+1} \Mmod$ for every $α > -1$; see
  \cite[§3.2.1]{Saito:HodgeModules}, but keep in mind that we are talking about
  left $𝒟$-modules.  Specialising to $α = 0$, we conclude that
  \[
    F_p \Mmod ∩ t \Mmod = t F_p \Mmod.
  \]
  It follows that $t \colon \gr_•^F \Mmod → \gr_•^F \Mmod$ is injective, and
  hence that $𝒪_D ⊗_{i_D^{-1} 𝒪_X} i_D^{-1} F_• \Mmod$ defines a good filtration
  of $𝒪_D ⊗_{i_D^{-1} 𝒪_X} i_D^{-1} \Mmod$ by coherent $𝒪_D$-submodules.  In
  particular, $i_D \colon D ↪ X$ is strictly non-characteristic for
  $(\Mmod, F_• \Mmod)$, in the terminology of \cite[§8]{saito-vanishing}.

  According to \lemmaref{lem:saito-nonchar}, we have
  \[
  \gr_V⁰ \Mmod ≅ \Mmod / t \Mmod ≅ 𝒪_D ⊗_{i_D^{-1} 𝒪_X} i_D^{-1} \Mmod,
  \]
  and the action of the (nilpotent) operator $N = t ∂_t$ is trivial.
  Consequently, the relative weight filtration of $N$ is equal to the filtration
  $W_• \Mmod / t W_• \Mmod ≅ 𝒪_D ⊗_{i_D^{-1} 𝒪_X} i_D^{-1} W_• \Mmod$ induced by
  the weight filtration of $M$ itself \cite[§2.3]{Saito:MixedHodgeModules}.  Now
  Saito's inductive definition of the category of (mixed) Hodge modules in
  \cite[§5.1]{Saito:HodgeModules} and \cite[(2.d)]{Saito:MixedHodgeModules}
  implies the first and third assertion.  The second assertion is a special case
  of Kashiwara's version of the Cauchy-Kovalevskaya theorem
  \cite[Cor.~4.3.4]{HTT}, which says that non-characteristic restriction is
  compatible with passage to the de Rham complex.
\end{proof}

We end this section by describing the relation between the de Rham complexes of the
two mixed Hodge modules $M$ and $H^{-1} i_D^{\ast} M$; see
\cite[(13.3)]{saito-vanishing} for the proof.

\begin{prop}[Comparison of de Rham complexes]\label{prop:restrComparison}
  In the setting of \theoremref{thm:restr}, denote by $(\Mmod_D, F_• \Mmod_D)$
  the filtered $𝒟_D$-module underlying the mixed Hodge module
  $M_D = H^{-1} i_D^{\ast} M$.  Given any $p ∈ ℤ$, one has a short exact
  sequence of complexes
  \[
    0 → N_{D \mid X}^{\ast} ⊗_{𝒪_D} \gr_{p+1}^F \DR(\Mmod_D) → 𝒪_D ⊗_{𝒪_Y}
    \gr_p^F \DR(\Mmod) → \gr_p^F \DR(\Mmod_D) \decal{1} → 0,
  \]
  where $N_{D \mid X}^{\ast}$ means the conormal bundle for the inclusion
  $D ⊆ X$.  \qed
\end{prop}

%
%
\svnid{$Id: S05-perverse.tex 269 2020-01-20 11:28:53Z kebekus $}

\section{A vanishing theorem for intersection complexes}
\subversionInfo
\approvals{Christian & yes \\Stefan & yes}

We briefly discuss a vanishing theorem for certain perverse sheaves that applies
in particular to intersection complexes.  Recall that a perverse sheaf $K$ on a
complex manifold $Y$ is, by definition, always \define{semiperverse}, meaning
that\CounterStep
\begin{equation}\label{eq:hgjh}
  \dim \Supp ℋ^j K ≤ -j, \quad\text{for every }j ∈ ℤ.
\end{equation}
These inequalities can be improved, provided that $K$ does not admit any
nontrivial morphisms to perverse sheaves whose support is properly contained in
$\Supp K$.  This applies for example to the intersection complex on any
irreducible complex space, and more generally to the de Rham complex of any
polarisable Hodge module with strict support.

\begin{prop}\label{prop:perverse-quotients}
  Let $K$ be a perverse sheaf on a complex manifold $Y$, and assume that
  $\Supp K$ has pure dimension $n$.  Then the following two conditions are
  equivalent:
  \begin{enumerate}
  \item\label{prop:quot-a} If $L$ is a perverse sheaf on $Y$ with
    $\dim \Supp L ≤ n-1$, then $\Hom(K, L) = 0$.
    
  \item \label{prop:quot-b} For every $j ≥ -n+1$, one has
    $\dim \Supp ℋ^j K ≤ -(j+1)$.
  \end{enumerate}
\end{prop}
\begin{proof}
  Let us show that \ref{prop:quot-a} implies \ref{prop:quot-b}.  Since $K$ is a
  perverse sheaf, one has $ℋ^j K = 0$ for $j ≤ -n-1$, and the inequalities in
  \eqref{eq:hgjh} imply that $ℋ^{-n} K$ is supported on all of $X$, whereas
  $\dim \Supp ℋ^j K ≤ -j$ for every $j ≥ -n+1$.  If we truncate $K$ with respect
  to the standard t-structure on $\Dbc(ℂ_X)$, the resulting constructible
  complex $K' := τ_{≥ -n+1} K$ is still semiperverse, and supported in a complex
  subspace that is properly contained in $X$.  By \ref{prop:quot-a}, the natural
  composed morphism
  \[
    K → K' → {}^p ℋ⁰ K'
  \]
  to the $0$-th cohomology sheaf for the perverse t-structure must therefore be
  trivial, which implies that the morphism $K → K'$ factors through
  $K'' := {}^p τ_{≤ -1} K'$, truncated with respect to the perverse t-structure
  on $\Dbc(ℂ_X)$.  For each $j ≥ -n+1$, this gives us a factorisation
  \[
    ℋ^j K → ℋ^j K'' → ℋ^j K'
  \]
  of the identity morphism.  By construction, $\dim \Supp ℋ^j K'' ≤ -(j+1)$, and
  therefore also $\dim \Supp ℋ^j K ≤ -(j+1)$ for every $j ≥ -n+1$, proving
  \ref{prop:quot-b}.

  It remains to show that, conversely, \ref{prop:quot-b} implies
  \ref{prop:quot-a}.  Suppose we are given a morphism of perverse sheaves
  $φ \colon K → L$ with $\dim \Supp L ≤ n-1$.  After replacing $L$ by $\img φ$,
  we can assume that $φ$ is surjective.  As before, we have $ℋ^j L = 0$ for
  $j ≤ -n$.  Now fix some $j ≥ -n+1$, and consider the short exact sequence
  \[
    ℋ^j K → ℋ^j L → ℋ^{j+1} (\ker φ).
  \]
  We have $\dim \Supp ℋ^j K ≤ -(j+1)$ by \ref{prop:quot-b}, and
  $\dim \Supp ℋ^{j+1}(\ker φ) ≤ -(j+1)$ by \eqref{eq:hgjh}.  Consequently,
  $\dim \Supp ℋ^j L ≤ -(j+1)$ for every $j ∈ ℤ$, and since $L$ is a perverse
  sheaf, the properties of the perverse t-structure imply that $L = 0$.
\end{proof}

The following vanishing theorem for the de Rham complex plays a crucial role in
the proof of our main theorem, and so we state it as a corollary.

\begin{cor}\label{cor:IC-vanishing}
  Let $Y$ be a complex manifold, and let $M ∈ \HM_X(Y, w)$ be a polarisable
  Hodge module of weight $w$ with strict support an irreducible complex subspace
  $X ⊆ Y$.  If $F_{c - 1} \Mmod = 0$ for some $c ∈ ℤ$, one has
  $ℋ⁰ F_{\dim Y - (w+c)} \DR(\Mmod) = 0$.
\end{cor}
\begin{proof}
  According to \propositionref{prop:DR-acyclic}, the complex
  $\gr_p^F \DR(\Mmod)$ is acyclic for $p ≥ \dim Y - (w+c) + 1$.  By
  \propositionref{prop:DR-qiso}, this implies that the inclusion of the
  subcomplex $F_{p_0} \DR(\Mmod)$ into $\DR(\Mmod)$ is a quasi-isomorphism for
  $p_0 = \dim Y - (w+c)$.  In particular, the inclusion induces an isomorphism
  $ℋ⁰ F_{p_0} \DR(\Mmod) ≅ ℋ⁰ \DR(\Mmod)$.  But now $M$ has strict support $X$,
  and so the perverse sheaf $\DR(\Mmod)$ does not have nontrivial quotient
  objects whose support is properly contained in $X$.  We conclude that
  $ℋ⁰ \DR(\Mmod) = 0$, by \propositionref{prop:perverse-quotients}.
\end{proof}

%
%
\svnid{$Id: S06-reflexivity.tex 269 2020-01-20 11:28:53Z kebekus $}

\section{Coherent sheaves and Mixed Hodge modules}
\subversionInfo
\approvals{Christian & yes \\Stefan & yes}

The present section forms the technical core of the present paper.  Its main
results, \theoremref{thm:reflexive-pure} and \theoremref{thm:reflexive-mixed},
as well as \corollaryref{cor:reflexive-pure-ng} and
\corollaryref{cor:reflexive-mixed-ng} are criteria to guarantee that sections of
certain coherent sheaves derived from the de Rham complex of certain (mixed)
Hodge modules on $X$ extend across the singular locus $X_{\sing}$.

\subsection{Extending sections of coherent sheaves}
\approvals{Christian & yes \\Stefan & yes}
\label{sec:extshf-ng}

In this paragraph, we give a homological formulation of the property that
sections of a coherent sheaf extend uniquely over a given complex subspace.  The
material covered here will be known to experts.

\begin{propdef}[Extension across subsets]\label{prop:ccor-ng}
  Let $Y$ be a complex manifold.  Let $A ⊆ Y$ be a complex subspace,
  and let $j \colon Y ∖ A ↪ Y$ be the open embedding.  If $ℱ$ is a coherent
  sheaf of $𝒪_Y$-modules, then the following conditions are equivalent:
  \begin{enumerate}
  \item\label{il:P1} The natural morphism $ℱ → j_* j^* ℱ$ is an isomorphism.
    
  \item\label{il:P2} For every $k ∈ ℤ$, one has
    $\dim \bigl( A ∩ \Supp R^k \sHom_{𝒪_Y}(ℱ, ω_Y^•) \bigr) ≤ -(k+2)$.
  \end{enumerate}
  If these conditions are satisfied, we say that \define{sections of $ℱ$ extend
    uniquely across $A$}.
\end{propdef}

We will often apply \propositionref{prop:ccor-ng} in the following form.

\begin{cor}\label{cor:ccor-ng}
  Let $Y$ be a complex manifold, and let $ℱ$ be a coherent sheaf of
  $𝒪_Y$-modules.  If $\Supp ℱ$ has pure dimension $n$, then the following
  conditions are equivalent:
  \begin{enumerate}
  \item Sections of $ℱ$ extend uniquely across any $A ⊆ Y$ with $\dim A ≤ n-2$.
    
  \item\label{il:P3} For every $k ≥ -n+1$, one has
    $\dim \Supp R^k \sHom_{𝒪_Y}(ℱ, ω_Y^•) ≤ -(k+2)$.
  \end{enumerate}
\end{cor}
\begin{proof}
  According to \cite[\href{http://stacks.math.columbia.edu/tag/0A7U}{Tag
    0A7U}]{stacks-project}, one has $R^k \sHom_{𝒪_Y}(ℱ, ω_Y^{•}) = 0$ for every
  $k ≤ -n$.  If $A ⊆ Y$ is a complex subspace with $\dim A ≤ n-2$, then of
  course
  \[
    \dim \bigl( A ∩ \Supp R^{-n} \sHom_{𝒪_Y}(ℱ, ω_Y^{•}) \bigr) ≤ n-2,
  \]
  and so the condition in \ref{il:P3} is equivalent to the condition in
  \ref{il:P2}.  The assertion now follows from \propositionref{prop:ccor-ng}.
\end{proof}

Before giving the proof of \propositionref{prop:ccor-ng}, we briefly review some
facts about singular sets of coherent sheaves.  Let $Y$ be a complex manifold,
and $ℱ$ a coherent sheaf of $𝒪_Y$-modules.  Recall that the \define{singular
  sets} of $ℱ$ are defined as
\[
  S_m(ℱ) := \bigl\{ \, y ∈ Y \,\big\vert \, \operatorname{depth}_y ℱ ≤ m \,
  \bigr\}.
\]
The singular sets $S_m(ℱ)$ are closed complex subspaces of $Y$; we refer the
reader to \cite[Chapt.~II.2]{BS76} for a detailed discussion.  The following
homological fact about regular local rings
\cite[\href{http://stacks.math.columbia.edu/tag/0A7U}{Tag 0A7U}]{stacks-project}
relates the singular sets to the dualizing complex.  In the smooth case at hand,
observe that the dualizing complex of
\cite[\href{http://stacks.math.columbia.edu/tag/0A7U}{Tag 0A7U}]{stacks-project}
agrees with the analytic dualizing complex, as both equal the canonical bundle
shifted by the dimension.

\begin{prop}[Singular sets and duality] \label{fact:stack} If $ℱ$ is a coherent
  sheaf of $𝒪_Y$-modules on a complex manifold $Y$, then the singular sets of
  $ℱ$ are described as
  \[
    S_m(ℱ) = \bigcup_{k ≥ 0} \Supp R^{k-m} \sHom_{𝒪_Y}(ℱ, ω_Y^•),
  \]
  where $ω_Y^•$ is the dualizing complex.  \qed
\end{prop}
\begin{proof}[Proof of Proposition~\ref*{prop:ccor-ng}]
  We consider the standard exact sequence for sheaves of local cohomology with
  supports, see for example \cite[II~Cor.~1.10]{BS76}.
  \[
    0 → ℋ⁰_A ℱ → ℱ → j_* j^* ℱ → ℋ¹_A ℱ → 0
  \]
  Because of this sequence, \ref{il:P1} is equivalent to the condition that
  $ℋ_A⁰ ℱ = ℋ_A¹ ℱ = 0$.  The vanishing theorem for local cohomology of
  Scheja-Trautmann \cite[II~Thm.~3.6]{BS76} relates this to the singular sets of
  $ℱ$: it asserts that $ℋ_A⁰ ℱ = ℋ_A¹ ℱ = 0$ is equivalent to the collection of
  inequalities
  \[
    \dim \bigl( A ∩ S_m(ℱ) \big) ≤ m-2 \quad \text{for all $m ∈ ℤ$.}
  \]
  But \propositionref{fact:stack} shows that this last line is in turn
  equivalent to \ref{il:P2}.
\end{proof}

We will later need the following variant of \propositionref{prop:ccor-ng} that
works for complexes of $𝒪_Y$-modules rather than single sheaves.  We stress
that, in the case of a complex with two or more nonzero cohomology sheaves, the
condition below is \emph{stronger} than asking that sections of $ℋ⁰ K$ extend
uniquely across $A$.

\begin{prop}\label{prop:reflexive-complex-ng}
  Let $Y$ be a complex manifold, let $A ⊆ Y$ be a complex subspace, and let
  $K ∈ \Dbcoh(𝒪_Y)$ be a complex with $ℋ^j K = 0$ for $j < 0$.  If
  \[
    \dim \bigl( A ∩ \Supp R^k \sHom_{𝒪_Y}(K, ω_Y^•) \bigr) ≤ -(k+2) \quad
    \text{for every $k ∈ ℤ$,}
  \]
  then sections in $ℋ⁰ K$ extend uniquely across $A$.
\end{prop}
\begin{proof}
  Let $τ_{≥ 1} K$ denote the truncation of the complex $K$ in cohomological
  degree $≥ 1$.  In the derived category $\Dbcoh(𝒪_Y)$, one has a distinguished
  triangle
  \[
    ℋ⁰ K → K → τ_{≥ 1} K → \bigl( ℋ⁰ K \bigr) \decal{1}.
  \]
  After applying the functor $\derR \sHom_{𝒪_Y}(\argbl, ω_Y^•)$ and taking
  cohomology, we obtain the following exact sequence:
  \[
    R^k \sHom_{𝒪_Y} \bigl( K, ω_Y^• \bigr) %
    → R^k \sHom_{𝒪_Y} \bigl( ℋ⁰ K, ω_Y^• \bigr) %
    → R^{k+1} \sHom_{𝒪_Y} \bigl( τ_{≥ 1} K, ω_Y^• \bigr)
  \]
  Thus $A ∩ \Supp R^k \sHom_{𝒪_Y} \bigl( ℋ⁰ K, ω_Y^• \bigr)$ is contained in the
  union of the two sets
  \[
    A ∩ \Supp R^k \sHom_{𝒪_Y} \bigl( K, ω_Y^• \bigr) \quad \text{and} \quad
    \Supp R^{k+1} \sHom_{𝒪_Y} \bigl( τ_{≥ 1} K, ω_Y^• \bigr)
  \]
  By assumption, the dimension of the first set is at most $-(k+2)$ for every
  $k ∈ ℤ$.  As $τ_{≥ 1} K ∈ \Dtcoh{≥ 1}(𝒪_Y)$, the same is true for the second
  set; this follows from \cite[\href{http://stacks.math.columbia.edu/tag/0A7U}%
  	{Tag 0A7U}]{stacks-project} by considering the spectral sequence
	\[
	E_2^{p,q} = R^p \sHom_{𝒪_Y} \bigl( ℋ^{-q} τ_{≥ 1} K, ω_Y^{•} \bigr)
	\Longrightarrow R^{p+q} \sHom_{𝒪_Y} \bigl( τ_{≥ 1} K, ω_Y^{•} \bigr).
	\]
  We conclude the proof by applying \propositionref{prop:ccor-ng} to the
  coherent $𝒪_Y$-module $ℋ⁰ K$.
\end{proof}

\subsection{The case of Hodge modules}
\approvals{Christian & yes \\Stefan & yes}
\label{par:pure}

In this section, we apply the criteria from \parref{sec:extshf-ng} to certain
coherent sheaves derived from the de Rham complex of certain Hodge modules.  We
specify the precise setting first.

\begin{setting}\label{set:5.8}
  Let $Y$ be a complex manifold, and let $X ⊆ Y$ be a reduced and irreducible complex
  subspace of dimension $n$.  Let $c$ be the codimension of the closed embedding
  $i_X \colon X ↪ Y$, so that $\dim Y = n + c$.  Suppose that $M ∈ \HM_X(Y, n)$
  is a polarisable Hodge module of weight $n$ with strict support equal to $X$.
  We denote the underlying filtered left $𝒟_Y$-module by $(\Mmod, F_• \Mmod)$,
  and make the following assumptions about $M$.
  \begin{enumerate}
  \item\label{en:pure-1} One has $F_{c-1} \Mmod = 0$.
    
  \item\label{en:pure-2} One has $\dim \Supp ℋ^j \gr_0^F \DR(\Mmod) ≤ -(j+2)$
    for every $j ≥ -n+1$.
  \end{enumerate}
\end{setting}

\begin{note}
  By \cite[Thm.~3.21]{Saito:MixedHodgeModules}, there is a dense Zariski-open
  subset of $X$ on which $M$ is a polarisable variation of Hodge structure of
  weight $0$.  The condition $F_{c-1} \Mmod = 0$ is equivalent to asking that
  the variation of Hodge structure is entirely of type $(0,0)$; being
  polarisable, it must therefore be a unitary flat bundle.  Now $F_c \Mmod$ is a
  certain extension of this unitary flat bundle to a coherent $𝒪_Y$-module, and
  \ref{en:pure-2} is equivalent to asking that sections of $F_c \Mmod$ extend
  uniquely over any complex subspace of $X$ of dimension at most $n - 2$.
\end{note}

\begin{thm}[Inequalities for Hodge modules]\label{thm:reflexive-pure}
  Assume \settingref{set:5.8} and let $p ∈ ℤ$ be any integer.  Then one has
  \begin{equation}\label{eq:reflexive-pure}
    \dim \Supp ℋ^j \gr_p^F \DR(\Mmod) ≤ -(p+j+2) \quad\text{for every $j$ with } p + j ≥ -n + 1.
  \end{equation}
\end{thm}

A proof of \theoremref{thm:reflexive-pure} is given in
\parref{pf:thm:reflexive-pure1} and \parref{pf:thm:reflexive-pure2} below.
First, however, we note that the dimension estimates in
\theoremref{thm:reflexive-pure} imply the promised extension property for
certain coherent sheaves derived from the de Rham complex.

\begin{cor}[Extending sections]\label{cor:reflexive-pure-ng}
  Assume \settingref{set:5.8}.  Then for any $p ∈ ℤ$, sections of
  $ℋ^{-(n-p)} \gr_{-p}^F \DR(\Mmod)$ extend uniquely across any complex subspace
  of dimension $≤ n-2$.
\end{cor}
\begin{proof}
  Recall from \propositionref{prop:DR-acyclic} that $\gr_{-p}^F \DR(\Mmod)$ is
  acyclic, unless $0 ≤ p ≤ n$.  Assuming that $p$ is in this range, we aim to
  apply \propositionref{prop:reflexive-complex-ng} to the complex
  \[
    K_p := \gr_{-p}^F \DR(\Mmod)[p-n],
  \]
  which requires first of all that $K_p$ is contained in $\Dtcoh{≥ 0}(𝒪_X)$.  To
  this end recall from Assumption~\ref{en:pure-1} that $F_{c-1} \Mmod = 0$.  An
  application of Formula~\eqref{eq:Ax} for the subquotients of the de Rham
  complex then shows that
  \[
    ℋ^j K_p = ℋ^{j+p-n} \gr_{-p}^F \DR(\Mmod) \overset{\text{\eqref{eq:Ax}}}{=}
    0, \quad\text{for every }j ≤ -1.
  \]
  So $K ∈ \Dtcoh{≥ 0}(𝒪_X)$, as desired.  Next, choose a polarisation on the
  Hodge module $M$, in order to obtain an isomorphism as follows,
  \[
    \derR \sHom_{𝒪_Y} \Bigl( K_p, ω_Y^• \Bigr)
    \overset{\text{\corollaryref{cor:duality-pure}}}{≅} \gr_{-(n-p)}^F
    \DR(\Mmod) \decal{n-p}.
   \]
   The Inequalities~\eqref{eq:reflexive-pure} of \theoremref{thm:reflexive-pure}
   therefore take the form
   \[
     \dim \Supp R^j \sHom_{𝒪_X} \bigl( K_p, ω_Y^• \bigr) = \dim \Supp ℋ^{j+n-p}
     \gr_{-(n-p)}^F \DR(\Mmod) ≤ -(j+2)
   \]
   for every $j ≥ -n + 1$.  We conclude from
   \propositionref{prop:reflexive-complex-ng} that sections of the coherent
   $𝒪_Y$-module $ℋ⁰ K_p = ℋ^{-(n-p)} \gr_{-p}^F \DR(\Mmod)$ extend uniquely
   across any complex subspace $A ⊆ Y$ with $\dim A ≤ n-2$.
\end{proof}

\subsubsection{Preparation for proof of Theorem~\ref*{thm:reflexive-pure}}
\approvals{Christian & yes \\Stefan & yes}
\label{pf:thm:reflexive-pure1}

In cases where $p + j ≥ \max(-n+1, -1)$, the inequality
\eqref{eq:reflexive-pure} in \theoremref{thm:reflexive-pure} is claiming that
$ℋ^j \gr_p^F \DR(\Mmod) = 0$.  As it turns out, the proof of this special case
is the core of the argument; the other cases follow quickly from the following
lemma by induction, taking repeated hyperplane sections.

\begin{lem}\label{lem:reflexive-pure}
  Assume \settingref{set:5.8}.  If $p+j ≥ \max(-n+1, -1)$, then
  $ℋ^j \gr_p^F \DR(\Mmod) = 0$.
\end{lem}
\begin{proof}
  The complex $\gr_p^F \DR(\Mmod)$ is concentrated in non-positive degrees, and
  acyclic for $p ≥ 1$ by \propositionref{prop:DR-acyclic} and by
  Assumption~\ref{en:pure-1}.  This means that $ℋ^j \gr_p^F \DR(\Mmod) = 0$
  whenever $j ≥ 1$ or $p ≥ 1$.  Assumption~\ref{en:pure-2} implies the claim
  when $p = 0$.  This leaves only one case to consider, namely $p = -1$ and
  $j = 0$.  We shall argue that $ℋ⁰ \gr_{-1}^F \DR(\Mmod) = 0$, too.
  
  Recall that $M$ has strict support $X$.  Assumption~\ref{en:pure-1} therefore
  allows us to apply \corollaryref{cor:IC-vanishing}.  We obtain
  $ℋ⁰ F_0 \DR(\Mmod) = 0$.  Now consider the short exact sequence of complexes
  (of sheaves of $ℂ$-vector spaces)
  \[
    0 → F_{-1} \DR(\Mmod) → F_0 \DR(\Mmod) → \gr_0^F \DR(\Mmod) → 0.
  \]
  Since $ℋ^j \gr_0^F \DR(\Mmod) = 0$ for $j ≥ -1$, we get
  \begin{equation}\label{eq:fgh}
    ℋ⁰ F_{-1} \DR(\Mmod) ≅ ℋ⁰ F_0 \DR(\Mmod)
    \overset{\text{Cor.~\ref{cor:IC-vanishing}}}{=} 0
  \end{equation}
  from the long exact sequence in cohomology.  By the same logic, the short
  exact sequence of complexes (of sheaves of $ℂ$-vector spaces)
  \[
    0 → F_{-2} \DR(\Mmod) → F_{-1} \DR(\Mmod) → \gr_{-1}^F \DR(\Mmod) → 0
  \]
  gives us an exact sequence
  \[
    ⋯ → \underbrace{ℋ⁰ F_{-1} \DR(\Mmod)}_{= 0\text{ by \eqref{eq:fgh}}} → ℋ⁰
    \gr_{-1}^F \DR(\Mmod) → \underbrace{ℋ¹ F_{-2} \DR(\Mmod)}_{\mathclap{= 0
        \text{, since concentr.~in non-pos.  degrees}}} → ⋯.
  \]
  As a consequence, we obtain the desired vanishing
  $ℋ⁰ \gr_{-1}^F \DR(\Mmod) = 0$.
\end{proof}

\subsubsection{Proof of Theorem~\ref*{thm:reflexive-pure}}
\label{pf:thm:reflexive-pure2}
\approvals{Christian & yes \\Stefan & yes}
\CounterStep

We prove \theoremref{thm:reflexive-pure} by induction on $n = \dim X$.  If
$n = 1$ or $n = 2$, then the desired statement follows from
\lemmaref{lem:reflexive-pure} above, and we are done.  We will therefore assume
for the remainder of the proof that $n ≥ 3$, and that
\theoremref{thm:reflexive-pure} is already known for all strictly smaller values
of $n$.

\subsubsection*{Cutting down}

The statement we are trying to prove is local on $Y$, and so we can assume for
the remainder of this proof that $Y$ is an open ball in $ℂ^{n+c}$.  (If the
restriction of $M$ no longer has strict support, for example because $X$ was
locally reducible, then we simply replace $M$ by any of the summands in the
decomposition by strict support, and $X$ by the support of that summand.) Let
$H ⊆ Y$ be the intersection of $Y$ with a generic hyperplane in $ℂ^{n+c}$.  The
intersection $H ∩ X$ is then reduced and irreducible of dimension $n-1 ≥ 2$.
The inclusion mapping $i_H \colon H ↪ Y$ is non-characteristic for $M$, and the
inverse image $M_H = H^{-1} i^*_H M$ is a polarisable Hodge module of weight
$(n-1)$ with strict support $H ∩ X$; see \parref{sec:non-characteristic} for a
discussion of non-characteristic restriction to smooth hypersurfaces.  Denoting
the underlying filtered $𝒟_H$-module by $(\Mmod_H, F_• \Mmod_H)$, we have
moreover
\begin{equation}\label{eq:jfgjhh}
  \Mmod_H ≅ 𝒪_H ⊗_{i_H^{-1} 𝒪_Y} i_H^{-1} \Mmod \quad \text{and} \quad
  F_• \Mmod_H ≅ 𝒪_H ⊗_{i_H^{-1} 𝒪_Y} i_H^{-1} F_• \Mmod.
\end{equation}
This is explained in \theoremref{thm:restr}.

\subsubsection*{Properties of $M_H$}

The isomorphisms in \eqref{eq:jfgjhh} imply that $F_{c-1} \Mmod_H = 0$, and so
$M_H$ also satisfies Assumption~\ref{en:pure-1}.  We claim that $M_H$ also
satisfies Assumption~\ref{en:pure-2}.  To this end, recall from
\propositionref{prop:restrComparison} that there exists a short exact sequence
of complexes,
\begin{equation}\label{eq:cbc}
  0 → N_{H \mid Y}^{\ast} ⊗_{𝒪_H} \gr_{p+1}^F \DR(\Mmod_H) → 𝒪_H ⊗_{𝒪_Y}
  \gr_p^F \DR(\Mmod) → \gr_p^F \DR(\Mmod_H) \decal{1} → 0,
\end{equation}
where $N_{H \mid Y}^{\ast}$ is the conormal bundle for the inclusion $H ⊆ Y$.
As $F_{c-1} \Mmod_H = 0$, one shows as before that the complex
$\gr_p^F \DR(\Mmod_H)$ is acyclic for every $p ≥ 1$.  This gives us
\[
  𝒪_H ⊗_{𝒪_Y} ℋ^{j-1} \gr_0^F \DR(\Mmod) ≅ ℋ^j \gr_0^F \DR(\Mmod_H),
\]
and because Assumption~\ref{en:pure-2} holds for $M$, we obtain that
\[
  \dim \Supp ℋ^j \gr_0^F \DR(\Mmod_H) = -1 + \dim \Supp ℋ^{j-1} \gr_0^F
  \DR(\Mmod) ≤ -(j+2)
\]
for every $j ≥ -\dim(H ∩ X) + 1$.  But this is exactly \ref{en:pure-2} for
$M_H$.

\subsubsection*{Conclusion}

We have established that $M_H ∈ \HM_{H ∩ X}(H, n-1)$ again satisfies the two
assumptions in \ref{en:pure-1} and \ref{en:pure-2}.  Since $\dim (H ∩ X) = n-1$,
we can therefore conclude by induction that
\[
  \dim \Supp ℋ^j \gr_p^F \DR(\Mmod_H) ≤ -(p+j+2), \quad\text{whenever } p+j ≥
  -(n-1) + 1.
\]
Taking cohomology, \eqref{eq:cbc} gives us an exact sequence of $𝒪_H$-modules,
\[
  N_{H \mid Y}^{\ast} ⊗ ℋ^j \gr_{p+1}^F \DR(\Mmod_H) → 𝒪_H ⊗_{𝒪_Y} ℋ^j \gr_p^F
  \DR(\Mmod) → ℋ^{j+1} \gr_p^F \DR(\Mmod_H),
\]
and therefore the inequality
\[
  \dim \Supp \bigl( 𝒪_H ⊗_{𝒪_Y} ℋ^j \gr_p^F \DR(\Mmod) \bigr) ≤
  -(p+j+3),\quad\text{whenever } p+j ≥ -n + 1.
\]
Since $H ⊆ Y$ was a generic hyperplane section of $Y$, this inequality clearly
implies that
\[
  \dim \Supp ℋ^j \gr_p^F \DR(\Mmod) ≤ -\min(p+j+2, 0),\quad\text{whenever } p+j
  ≥ -n + 1.
\]
This is enough for our purposes, because we have already shown in
\lemmaref{lem:reflexive-pure} that $ℋ^j \gr_p^F \DR(\Mmod) = 0$ whenever
$p+j ≥ -1$.  The proof of \theoremref{thm:reflexive-pure} is thus complete.
\qed

\subsection{The case of mixed Hodge modules}
\approvals{Christian & yes \\Stefan & yes}
\label{ssec:tcomHm}

In this section, we generalise \theoremref{thm:reflexive-pure} and
\corollaryref{cor:reflexive-pure-ng} to a certain class of mixed Hodge modules.
The results presented here will later be relevant to establish the extension
results for logarithmic forms, \theoremref{thm:extension-Kahler-log},
\theoremref{thm:main-log} and \theoremref{thm:extension-n-1}, as well as the
proof of local vanishing, \theoremref{thm:MOP}.  The reader who is primarily
interested in the extension for $p$-forms, \theoremref{thm:main-new}, might wish
to avoid the additional complications arising from the use of mixed Hodge
modules and skip this section on first reading.

The main line of argument follows \parref{par:pure}, though there are some
noteworthy differences.  To keep the text readable, we chose to include full
arguments, at the cost of introducing some repetition.

\begin{setting}\label{set:5.13}
  Let $Y$ be a complex manifold of pure dimension $n+c$, and let $X ⊆ Y$ be a
  complex subspace of pure dimension $n$.  As before, $c$ is equal to the
  codimension of the closed embedding $i_X \colon X ↪ Y$.  Suppose that
  $M ∈ \MHM(Y)$ is a graded-polarisable mixed Hodge module with support equal to
  $X$.  We denote the underlying filtered left $𝒟_Y$-module by
  $(\Mmod, F_• \Mmod)$, and make the following assumptions about $M$:
  \begin{enumerate}
  \item\label{en:mixed-2} One has $\dim \Supp ℋ^j \DR(\Mmod) ≤ -(j+1)$ for every
    $j ≥ -n + 1$.

  \item\label{en:mixed-3} The complex of $𝒪_Y$-modules $\gr_p^F \DR(\Mmod)$ is
    acyclic for every $p ≥ 1$.

  \item\label{en:mixed-4} One has $\dim \Supp ℋ^j \gr_0^F \DR(\Mmod) ≤ -(j+2)$
    for every $j ≥ -n + 1$.
  \end{enumerate}
  These are the natural generalisations of \ref{en:pure-1} and \ref{en:pure-2}
  to the mixed case, formulated in a way that is convenient for a proof by
  induction on the dimension.  As before, write $M' := 𝔻 M ∈ \MHM(Y)$ to denote
  the dual mixed Hodge module, which is again graded-polarisable, and write
  $(\Mmod', F_• \Mmod')$ for its underlying filtered left $𝒟_Y$-module.  Recall
  that the support does not change when taking duals, so
  $\Supp M' = \Supp M = X$.
\end{setting}

\begin{note}
  The cohomology sheaves of the de Rham complex $\DR(\Mmod)$ are constructible
  sheaves on $Y$.  Since $\DR(\Mmod)$ is a perverse sheaf, the dimension of the
  support of $ℋ^j \DR(\Mmod)$ is always at most $-j$ for every $j ∈ ℤ$.  In
  light of \propositionref{prop:perverse-quotients}, the condition in
  \ref{en:mixed-2} is saying that $\DR(\Mmod)$ does not admit nontrivial
  quotients whose support has dimension $≤ n-1$.
\end{note}

\begin{thm}[Inequalities for mixed Hodge modules]\label{thm:reflexive-mixed}
  Assume \settingref{set:5.13} and let $p ∈ ℤ$ be any integer.  Then one has
  \begin{equation}\label{eq:reflexive-mixed}
    \dim \Supp ℋ^j \gr_p^F \DR(\Mmod) ≤ -(p+j+2)
    \quad\text{for every $j$ with } p + j ≥ -n+1.
  \end{equation}
\end{thm}

The proof of \theoremref{thm:reflexive-mixed} is given in
\parref{pf:thm:reflexive-mixed1} and \parref{pf:thm:reflexive-mixed2} below.  As
before, \theoremref{thm:reflexive-mixed} leads to extension theorems for certain
coherent sheaves derived from the de Rham complex.

\begin{cor}[Extending sections]\label{cor:reflexive-mixed-ng}
  Assume \settingref{set:5.13}.  Then for any $p ∈ ℤ$, sections of
  $ℋ^p \gr_p^F \DR(\Mmod')$ extend uniquely across any complex subspace of
  dimension $≤ n-2$.
\end{cor}
\begin{proof}
  Write $K_p := \gr_{-p}^F \DR(\Mmod')[-p]$.  As in the proof of
  \corollaryref{cor:reflexive-pure-ng}, we begin by showing that
  $K_p ∈ \Dtcoh{≥ 0}(𝒪_X)$.  To this end, \propositionref{prop:duality-mixed},
  implies that
  \[
    \gr_ℓ^F \DR(\Mmod') ≅ \derR \sHom_{𝒪_Y} \Bigl( \gr_{-ℓ}^F \DR(\Mmod), ω_Y^•
    \Bigr) \quad\text{for every $ℓ ∈ ℤ$}.
  \]
  By \ref{en:mixed-3}, this complex is acyclic for all $ℓ ≤ -1$.  In particular,
  it follows from \lemmaref{lem:5-2} that $F_{d-1} \Mmod' = 0$.  The description
  \eqref{eq:Ax} of the graded pieces in the de Rham complex then implies that
  $ℋ^j \gr_p^F \DR(\Mmod') = 0$ for $j < -p$.  In other words, we obtain that
  $K_p ∈ \Dtcoh{≥ 0}(𝒪_X)$ as desired.
  
  As before, \propositionref{prop:duality-mixed} gives isomorphisms
  \[
    \derR \sHom_{𝒪_Y} \Bigl( K_p, ω_Y^• \Bigr) %
    = \derR \sHom_{𝒪_Y} \Bigl( \gr_{-p}^F \DR(\Mmod') \decal{-p}, ω_Y^• \Bigr) %
    ≅ \gr_p^F \DR(\Mmod) \decal{p}
  \]
  With these identifications, the inequalities \eqref{eq:reflexive-mixed} in
  \theoremref{thm:reflexive-mixed} take the form
  \[
    \dim \Supp R^j \sHom_{𝒪_X} \bigl( K_p, ω_Y^• \bigr) %
    = \dim \Supp ℋ^{j+p} \gr_{-p}^F \DR(\Mmod) ≤ -(j+2)
  \]
  for every $j ≥ -n+1$.  As before, we conclude from
  \propositionref{prop:reflexive-complex-ng} that sections of the coherent
  $𝒪_Y$-module $ℋ⁰ K_p = ℋ^p \gr_p^F \DR(\Mmod')$ extend uniquely across any
  complex subspace $A ⊆ Y$ with $\dim A ≤ n-2$.
\end{proof}

\subsubsection{Preparation for proof of Theorem~\ref*{thm:reflexive-mixed}}
\approvals{Christian & yes \\Stefan & yes}
\label{pf:thm:reflexive-mixed1}

In cases where $p+j ≥ \max(-n+1, -1)$, the inequality \eqref{eq:reflexive-mixed}
in \theoremref{thm:reflexive-mixed} is claiming that
$ℋ^j \gr_p^F \DR(\Mmod) = 0$.  We begin by proving that this is indeed the case.

\begin{lem}\label{lem:reflexive-mixed}
  Assume \settingref{set:5.13}.  If $p+j ≥ \max(-n+1, -1)$, then
  $ℋ^j \gr_p^F \DR(\Mmod) = 0$.
\end{lem}
\begin{proof}
  The complex $\gr_p^F \DR(\Mmod)$ is concentrated in non-positive degrees, and
  is acyclic for $p ≥ 1$ by Assumption~\ref{en:mixed-3}.  This means that
  $ℋ^j \gr_p^F \DR(\Mmod) = 0$ whenever $j ≥ 1$ or $p ≥ 1$.
  Assumption~\ref{en:mixed-4} implies the claim when $p = 0$.  This leaves only
  one case to consider, namely $p = -1$ and $j = 0$.  We show that
  $ℋ⁰ \gr_{-1}^F \DR(\Mmod) = 0$, too.
  
  The inclusion $F_0 \DR(\Mmod) ⊆ \DR(\Mmod)$ is a quasi-isomorphism; this
  follows from Assumption~\ref{en:mixed-3} and \propositionref{prop:DR-qiso}.
  In particular, the inclusion induces an isomorphism
  $ℋ⁰ F_0 \DR(\Mmod) ≅ ℋ⁰ \DR(\Mmod)$.  The inequality in \ref{en:mixed-2} shows
  that $ℋ⁰ \DR(\Mmod) = 0$, and therefore $ℋ⁰ F_0 \DR(\Mmod) = 0$.  Now consider
  the short exact sequence of complexes (of sheaves of $ℂ$-vector spaces)
  \[
    0 → F_{-1} \DR(\Mmod) → F_0 \DR(\Mmod) → \gr_0^F \DR(\Mmod) → 0.
  \]
  Since $ℋ^j \gr_0^F \DR(\Mmod) = 0$ for $j ≥ -1$, we obtain
  \[
    ℋ⁰ F_{-1} \DR(\Mmod) ≅ ℋ⁰ F_0 \DR(\Mmod) = 0
  \]
  from the long exact sequence in cohomology.  The rest of the proof now
  proceeds exactly as in \lemmaref{lem:reflexive-pure}.
\end{proof}

\subsubsection{Proof of Theorem~\ref*{thm:reflexive-mixed}}
\approvals{Christian & yes \\Stefan & yes}
\label{pf:thm:reflexive-mixed2}

We prove \theoremref{thm:reflexive-mixed} by induction on $n = \dim X$.  If
$n = 1$ or $n = 2$, then the desired statement follows from
\lemmaref{lem:reflexive-mixed} above, and we are done.  We will therefore assume
for the remainder of the proof that $n ≥ 3$, and that
\theoremref{thm:reflexive-mixed} is already known for smaller values of $n$.

\subsubsection*{Cutting down}

The statement we are trying to prove is local on $Y$, and so we can assume for
the remainder of the argument that $Y$ is an open ball in $ℂ^{n+c}$, and that
$X ⊆ Y$ is connected.  Let $H ⊆ Y$ be the intersection of $Y$ with a generic
hyperplane in $ℂ^{n+c}$.  The intersection $H ∩ X$ is then a connected complex
subspace of pure dimension $n-1 ≥ 2$.  The inclusion mapping $i_H \colon H ↪ Y$
is non-characteristic for $M$, and the inverse image $M_H = H^{-1} i^*_H M$ is
again a graded-polarisable mixed Hodge module with support $H ∩ X$; see
\theoremref{thm:restr} for the details.  Note that the support of
$M_H ∈ \MHM(H)$ still has codimension $c$ in the ambient complex manifold $H$.
Denoting the underlying filtered $𝒟_H$-module by $(\Mmod_H, F_• \Mmod_H)$,
\theoremref{thm:restr} give
\[
  \Mmod_H ≅ 𝒪_H ⊗_{i_H^{-1} 𝒪_Y} i_H^{-1} \Mmod \quad \text{and} \quad %
  F_• \Mmod_H ≅ 𝒪_H ⊗_{i_H^{-1} 𝒪_Y} i_H^{-1} F_• \Mmod,
\]
as well as an isomorphism of perverse sheaves
\begin{equation}\label{eq:oiu}
  \DR(\Mmod_H) ≅ i_H^{-1} \DR(\Mmod) \decal{-1}.
\end{equation}

\subsubsection*{Properties of $M_H$}
  
As before, we claim that $M_H ∈ \MHM(H)$ satisfies all assumptions made in
\settingref{set:5.13}.  We consider the assumptions one by one.  Because $M$
satisfies Assumption~\ref{en:mixed-2} and because of the choice of $H$ as a
\emph{generic} hyperplane section, \eqref{eq:oiu} yields
\[
  \dim \Supp ℋ^j \DR(\Mmod_H) = \dim \Bigl( H ∩ \Supp ℋ^{j-1} \DR(\Mmod) \Bigr)
  ≤ -(j+1),
\]
for every $j ≥ -\dim(H ∩ X) + 1$.  In other words, $M_H$ satisfies
\ref{en:mixed-2} as well.

According \propositionref{prop:restrComparison}, one has a short exact sequence
of complexes
\begin{equation}\label{eq:cbc-mixed}
  0 → N_{H \mid Y}^{\ast} ⊗_{𝒪_H} \gr_{p+1}^F \DR(\Mmod_H) → 𝒪_H ⊗_{𝒪_Y}
  \gr_p^F \DR(\Mmod) → \gr_p^F \DR(\Mmod_H) \decal{1} → 0,
\end{equation}
where $N_{H \mid Y}^{\ast}$ is the conormal bundle for the inclusion $H ⊆ Y$.
Since $\gr_p^F \DR(\Mmod_H)$ is acyclic for $p ≫ 0$, and since
Assumption~\ref{en:mixed-3} holds for $M$, we can use descending induction on
$p$ to show that $\gr_p^F \DR(\Mmod_H)$ is acyclic for every $p ≥ 1$, and hence
that $M_H$ satisfies \ref{en:mixed-3}.  It also follows that
\[
  𝒪_H ⊗_{𝒪_Y} ℋ^{j-1} \gr_0^F \DR(\Mmod) ≅ ℋ^j \gr_0^F \DR(\Mmod_H),
\]
and because of Assumption~\ref{en:mixed-4}, we get
\[
  \dim \Supp ℋ^j \gr_0^F \DR(\Mmod_H) = -1 + \dim \Supp ℋ^{j-1} \gr_0^F
  \DR(\Mmod) ≤ -(j+2)
\]
for every $j ≥ -\dim(H ∩ X) + 1$.  But this is exactly \ref{en:mixed-4} for
$M_H$.

\subsubsection*{Conclusion}

In summary, we have established that $M_H ∈ \MHM(H)$ also has the three
properties in \ref{en:mixed-2} to \ref{en:mixed-4}, but with
$\dim \Supp M_H = \dim (H ∩ X) = n-1$.  We can therefore conclude by induction
on the dimension of the support that
\[
  \dim \Supp ℋ^j \gr_p^F \DR(\Mmod_H) ≤ -(p+j+2) \quad\text{whenever
    $p+j ≥ -\dim(H ∩ X) + 1$.}
\]
Taking cohomology in the short exact in \eqref{eq:cbc-mixed}, we obtain an exact
sequence of coherent $𝒪_H$-modules
\[
  N_{H \mid Y}^{\ast} ⊗ ℋ^j \gr_{p+1}^F \DR(\Mmod_H) → 𝒪_H ⊗_{𝒪_Y} ℋ^j \gr_p^F
  \DR(\Mmod) → ℋ^{j+1} \gr_p^F \DR(\Mmod_H),
\]
and therefore the inequality
\[
  \dim \Supp \Bigl( 𝒪_H ⊗_{𝒪_Y} ℋ^j \gr_p^F \DR(\Mmod) \Bigr) ≤ -(p+j+3)
  \quad\text{whenever $p+j ≥ -\dim X + 1$.}
\]
Since $H ⊆ Y$ was a generic hyperplane section of $Y$, this inequality clearly
implies that
\[
  \dim \Supp ℋ^j \gr_p^F \DR(\Mmod) ≤ -\min(p+j+2, 0) \quad\text{for
    $p+j ≥ -\dim X + 1$.}
\]
This is enough for our purposes, because we have already shown that
$ℋ^j \gr_p^F \DR(\Mmod) = 0$ whenever $p+q ≥ -1$.  The proof of
\theoremref{thm:reflexive-mixed} is thus complete.  \qed

\phantomsection\addcontentsline{toc}{part}{Proofs of the main theorems}
%
%
\svnid{$Id: S07-proof-setup.tex 269 2020-01-20 11:28:53Z kebekus $}

\section{Setup for the proof}
\label{sec:pf-setup}
\subversionInfo
\approvals{Christian & yes \\Stefan & yes}

We will prove the main results of the present paper in the following sections.
Since we want to work locally, and since an irreducible complex space is not
necessarily locally irreducible, we relax the assumptions a little bit and allow
any reduced complex space of pure dimension.  Except for
\theoremref{thm:pullBack}, the proofs all work in essentially the same setup.
We will therefore fix the setup here and introduce notation that will be
consistently be used throughout the following sections.

\begin{setup}\label{setting:7-1}
  Consider a reduced complex space $X$ of pure dimension $n$, together with an
  embedding $i_X \colon X ↪ Y$ into an open ball.  Choose a strong log
  resolution $r \colon \wtilde{X} → X$ that is projective as a morphism of
  complex spaces.
\end{setup}

\begin{notation}
  We denote dimensions and codimensions by
  $$
  n := \dim X \quad \text{and} \quad c := \codim_Y X,
  $$
  which means that $Y$ is an open ball in $ℂ^{n+c}$.  The assumption that $r$ is
  a \emph{strong} log resolution implies that $X_{\reg}$ is isomorphic to its
  preimage $r^{-1}(X_{\reg})$.  Finally, let $E := r^{-1}(X_{\sing})$ be the
  reduced $r$-exceptional set.  The assumption that $r$ is a strong log
  resolution implies that $E ⊊ \wtilde{X}$ is a divisor with simple normal
  crossings; we write its irreducible components as $E = ∪_{i ∈ I} E_i$.  The
  following diagram summarises the relevant morphisms in our setting.
  \[
    \begin{tikzcd}[row sep=huge, column sep=18 ex]
      \wtilde{X} ∖ E \rar{j \text{, open embedding}}
      \dar[swap]{r|_{\wtilde{X}°}\text{, isomorphism}} & \wtilde{X}
      \dar[swap]{\txt{\scriptsize$r$, strong\\\scriptsize log resolution}}
      \arrow[bend left=15]{dr}{f := i_X ◦ r} \\
      X_{\reg} \rar[swap]{j\text{, open embedding}} & X
      \rar[swap]{i_X\text{, closed embedding}} & Y
    \end{tikzcd}
  \]
\end{notation}

%
%
\svnid{$Id: S08-proof-pure.tex 261 2020-01-13 15:02:45Z kebekus $}

\section{Pure Hodge modules and differentials on the resolution}
\label{ssec:7.pm}
\subversionInfo
\approvals{Christian & yes \\Stefan & yes}

Maintaining the assumptions and notation of \settingref{setting:7-1}, we explain
in this section how the (higher) direct images of $Ω^p_{\wtilde{X}}$ are related
to the intersection complex on $X$.  We begin with a discussion of the constant
Hodge module on the complex manifold $\wtilde{X}$.

\subsection{The constant Hodge module on the resolution}
\label{ssec:tchmXt}
\approvals{Christian & yes \\Stefan & yes}

On the complex manifold $\wtilde{X}$, consider the locally constant sheaf
$ℚ_{\wtilde{X}}$, viewed as a polarised variation of Hodge structure of type
$(0,0)$.  Following Saito \cite[Thm.~5.4.3]{Saito:HodgeModules}, we denote by
$ℚ^H_{\wtilde{X}}[n] ∈ \HM(\wtilde{X}, n)$ the corresponding polarised Hodge
module of weight $n$; see also \cite[Sect.~2, Ex.~4]{MR3821162}.  Its underlying
regular holonomic left $𝒟_{\wtilde{X}}$-module is $𝒪_{\wtilde{X}}$, with the
usual action by differential operators, and the Hodge filtration
$F_• 𝒪_{\wtilde{X}}$ is given by
$$
F_p 𝒪_{\wtilde{X}} =
\begin{cases}
  0 & \text{if } p ≤ -1\\
  𝒪_{\wtilde{X}} & \text{otherwise.}
\end{cases}
$$
The de~Rham complex $\DR(𝒪_{\wtilde{X}})$, which is quasi-isomorphic to
$ℂ_{\wtilde{X}} \decal{n}$, is
$$
\DR(𝒪_{\wtilde{X}}) = \Bigl[ 𝒪_{\wtilde{X}} \xrightarrow{d} Ω¹_{\wtilde{X}}
\xrightarrow{d} ⋯ \xrightarrow{d} Ω^n_{\wtilde{X}} \Bigr][n].
$$
It is filtered in the usual way, by degree, and the $(-p)$-th graded piece is
then
\begin{equation}\label{eq:xx1}
  gr^F_{-p} DR(𝒪_{\wtilde{X}}) ≅ Ω^p_{\wtilde{X}}[n-p].
\end{equation}

Following the discussion in \parref{sec:push-forward}, we consider the direct
image $f_+(R_F 𝒪_{\wtilde{X}})$ of the filtered $𝒟_{\wtilde{X}}$-module
$(𝒪_{\wtilde{X}}, F_• 𝒪_{\wtilde{X}})$, as an object of the bounded derived
category of coherent graded $R_F 𝒟_Y$-modules.  The direct image functor
commutes with taking the associated graded of the de Rham complex by
\propositionref{prop:fl-grDR}, which allows us to identify the graded pieces of
the de Rham complex for $f_+(R_F 𝒪_{\wtilde{X}})$ as
\begin{equation}\label{eq:DICHM}
  \gr_{-p}^F \DR \Bigl( f_+(R_F 𝒪_{\wtilde{X}}) \Bigr) %
  ≅ \derR f_* \gr_{-p}^F \DR(𝒪_{\wtilde{X}}) %
  ≅ \derR f_* Ω^p_{\wtilde{X}} \decal{n-p}.
\end{equation}

\subsection{The intersection complex of $X$}
\label{ssec:icpx}
\approvals{Christian & yes \\Stefan & yes}

Consider the constant variation of Hodge structure of type $(0,0)$ on
$X_{\reg}$.  By Saito's fundamental theorem
\cite[Thm.~3.21]{Saito:MixedHodgeModules}, applied to each irreducible component
of the complex space $X$, it determines a polarised Hodge module
$M_X ∈ \HM(Y, n)$ of weight $n = \dim X$ on the complex manifold $Y$, with
support equal to $X$.  Its underlying perverse sheaf is the intersection complex
of $X$.  Denoting the filtered regular holonomic $𝒟_Y$-module underlying $M_X$
by $(\cM_X, F_• \cM_X)$, we have $F_{c-1} \cM_X = 0$ by construction.  The de
Rham complex $\DR(\cM_X)$ is again filtered, and its subquotients are
\[
  \gr_p^F \DR(\cM_X) = \Bigl\lbrack \gr_p^F \cM_X → Ω¹_Y ⊗ \gr_{p+1}^F \cM_X → ⋯
  → Ω^{n+c}_Y ⊗ \gr_{p+n+c}^F \cM_X \Bigr\rbrack \decal{n+c}.
\]
Note that this complex is concentrated in degrees $-(n+c), …, 0$.

\subsection{Decomposition}
\approvals{Christian & yes \\Stefan & yes}

As discussed in \parref{sec:push-forward}, the fact that the holomorphic mapping
$f \colon \wtilde{X} → Y$ is projective implies that each
$H^ℓ f_* ℚ_{\wtilde{X}}^H \decal{n}$ is again a polarisable Hodge module of
weight $n+ℓ$ on $Y$.  Using the decomposition by strict support, we obtain
moreover
\[
  H^ℓ f_* M_{\wtilde{X}} ≅
  \begin{cases}
    M_X ⊕ M_0 &\text{if $ℓ = 0$,} \\
    M_ℓ &\text{if $ℓ ≠ 0$,}
  \end{cases}
\]
where $M_X ∈ \HM(Y, n)$ is as above, and where the other summands
$M_ℓ ∈ \HM(Y, n+ℓ)$ are polarisable Hodge modules on $Y$ whose support is
contained inside $X_{\sing}$.  Denoting the associated $𝒟_Y$-modules by
$\Mmod_ℓ$, the properties of the direct image functor imply that
$F_c \Mmod_ℓ = 0$, as a special case of \propositionref{prop:Fc}.

\begin{note}
  For dimension reasons, one has $M_ℓ = 0$ once $\abs{ℓ}$ is greater than the
  ``defect of semismallness'' of $r \colon \wtilde{X} → X$; in particular, this
  holds for $\abs{ℓ} ≥ n-1$.
\end{note}

\subsection{Relation with differential forms}
\approvals{Christian & yes \\Stefan & yes}

Saito's version of the Decomposition Theorem,
\corollaryref{cor:decomposition-MF}, together with the isomorphism in
\eqref{eq:DICHM}, allows us to identify, for every $p ∈ ℤ$, the derived push
forward of the sheaf of $p$-forms on $\wtilde{X}$ as
\begin{equation}\label{eq:DS}
  \derR f_* Ω^p_{\wtilde{X}} \decal{n-p} %
  ≅ \gr_{-p}^F \DR(\Mmod_X) ⊕ \bigoplus_{ℓ ∈ ℤ} \gr_{-p}^F \DR(\Mmod_ℓ) \: \decal{-ℓ}.
\end{equation}
In the situation at hand, the relation between $f_* Ω^p_{\wtilde{X}}$ and the
intersection complex of $X$ is an almost direct consequence of the isomorphism
in \eqref{eq:DS} above.

\begin{prop}\label{prop:7-1-ng}
  Maintaining \settingref{setting:7-1} and using the notation introduced above,
  we have
  \[
    f_* Ω^p_{\wtilde{X}} ≅ ℋ^{-(n-p)} \gr_{-p}^F \DR(\cM_X) \quad\text{for every
      $p ∈ ℤ$.}
  \]
\end{prop}
\begin{proof}
  Recall from \eqref{eq:DS} that we have a decomposition
  \[
    \derR f_* Ω^p_{\wtilde{X}} \decal{n-p} ≅ \gr_{-p}^F \DR(\cM_X) ⊕
    \operatorname{Rest}_p,
  \]
  in which the support of the complex $\operatorname{Rest}_p ∈ \Dbcoh(𝒪_Y)$ is
  contained inside $X_{\sing}$.  Taking cohomology in degree $-(n-p)$, we get
  \[
    f_* Ω^p_{\wtilde{X}} ≅ \underbrace{ℋ^{-(n-p)} \gr_{-p}^F \DR(\cM_X)}_{=: \sf
      A} ⊕ \underbrace{ℋ^{-(n-p)} \operatorname{Rest}_p}_{=: \sf B},
  \]
  and therefore $f_* Ω^p_{\wtilde{X}}$ is the direct sum of $\sf A$ and a
  coherent $𝒪_X$-module $\sf B$ supported on $X_{\sing}$.  The claim follows
  because $Ω^p_{\wtilde{X}}$ is torsion free: the functor $f^*$ is a left
  adjoint for $f_*$, and the adjoint morphism $f^* \sf B → Ω^p_{\wtilde{X}}$
  vanishes because $f^* \sf B$ is supported on $f^{-1}(X_{\sing})$.
\end{proof}

\begin{note}
  The proof shows once again that $F_c \Mmod_ℓ = 0$ for every $ℓ ∈ ℤ$.  (Use
  \lemmaref{lem:5-2}.) This fact is also proved in much greater generality in
  \cite[Prop.~2.6]{Saito:Kollar}.
\end{note}

The two values $p = n$ and $p = 0$ are special, because there is no contribution
from the Hodge modules $M_ℓ$ in those cases.

\begin{prop}\label{prop:spidz}
  Maintaining \settingref{setting:7-1} and using the notation introduced above,
  we have
  \[
    \derR f_* Ω^n_{\wtilde{X}} ≅ \gr_{-n}^F \DR(\Mmod_X) \quad \text{and} \quad
    \derR f_* 𝒪_{\wtilde{X}} \decal{n} ≅ \gr_0^F \DR(\Mmod_X).
  \]
\end{prop}
\begin{proof}
  By \propositionref{prop:Fc}, we have $F_c \Mmod_ℓ = 0$ for every $ℓ ∈ ℤ$, and
  so $\gr_{-n}^F \DR(\Mmod_ℓ) = 0$.  Together with \eqref{eq:DS}, this implies
  the first isomorphism.  The second isomorphism follows by duality, using
  \corollaryref{cor:duality-pure} and the fact that $M_X ∈ \HM_X(Y, n)$.
\end{proof}

The higher direct images of $Ω^p_{\wtilde{X}}$ can of course also be computed
from \eqref{eq:DS}, but they generally involve some of the other terms $M_ℓ$.
We give one example, in the special case $p=1$, that will serve to illustrate
the general technique.

\begin{prop}\label{prop:xa2}
  Maintaining \settingref{setting:7-1} and using the notation introduced above,
  we have
  \[
    R^{n-1} f_* Ω¹_{\wtilde{X}} %
    ≅ ℋ⁰ \gr_{-1}^F \DR(\Mmod_X) ⊕ ℋ⁰ \gr_{-1}^F \DR(\Mmod_0).
  \]
\end{prop}
\begin{proof}
  Formula~\eqref{eq:DS} identifies the left side of the desired equality as
  \[
    R^{n-1} f_* Ω¹_{\wtilde{X}} %
    ≅ ℋ⁰ \gr_{-1}^F \DR(\Mmod_X) ⊕ \bigoplus_{ℓ ≥ 0} ℋ^{-ℓ} \gr_{-1}^F
    \DR(\Mmod_ℓ).
  \]
  To prove \propositionref{prop:xa2}, it is therefore enough to show that
  $\gr_{-1}^F \DR(\Mmod_ℓ)$ is acyclic for every $ℓ ≥ 1$.  But using the fact
  that the Hodge modules $M_ℓ ∈ \HM(Y, n+ℓ)$ are polarisable of weight $n+ℓ$,
  \corollaryref{cor:duality-pure} yields
  \[
    \gr_{-1}^F \DR(\Mmod_ℓ) %
    ≅ \derR \sHom_{𝒪_Y} \Bigl( \gr_{1-(n+ℓ)}^F \DR(\Mmod_ℓ), ω_Y^• \Bigr).
  \]
  Now a look back at the description of the filtration on the de Rham complex,
  in \eqref{eq:Aw}, reveals that the complex $\gr_{1-(n+ℓ)}^F \DR(\Mmod_ℓ)$ only
  involves the $𝒪_Y$-modules $\gr_k^F \Mmod_ℓ$ with $k ≤ c + 1 - ℓ$.  As
  $F_c \Mmod_ℓ = 0$, it follows that $\gr_{1-(n+ℓ)}^F \DR(\Mmod_ℓ) = 0$ for
  every $ℓ ≥ 1$.
\end{proof}

\subsection{Application to the extension problem}
\label{sec:08eoetf}
\approvals{Christian & yes \\Stefan & yes}

We conclude this section with a brief discussion of the effect that
extendability of $n$-forms has on $\DR(\Mmod_X)$ and its subquotients.  The
following result, together with \corollaryref{cor:reflexive-pure-ng}, can be
used to prove that if $n$-forms extend, then all forms extend.  As explained in
\parref{sec:second-proof}, this gives another proof for
\theoremref{thm:main-new} in the (most important) case $k = n$.

\begin{prop}[Extension of $n$-forms and $M_X$]\label{prop:main-ineq}
  Maintaining \settingref{setting:7-1} and using the notation introduced above,
  assume that $r_* Ω^n_{\wtilde{X}} ↪ j_* Ω^n_{X_{\reg}}$ is an isomorphism.
  Then one has
  \[
    \dim \Supp ℋ^j \gr_p^F \DR(\Mmod_X) ≤ -(j+p+2)
  \]
  for all integers $p, j ∈ ℤ$ with $p+j ≥ -n + 1$.
\end{prop}
\begin{proof}
  After replacing the Hodge module $M_X ∈ \HM(Y, n)$ by any of the summands in
  its decomposition by strict support, and $X$ by the support of that summand,
  we may assume without loss of generality that $X$ is reduced, irreducible, and
  $n$-dimensional, and that $M_X$ has strict support $X$; in symbols,
  $M_X ∈ \HM_X(Y, w)$.  We aim to apply \theoremref{thm:reflexive-pure}.
  Recalling from \parref{ssec:icpx} that $F_{c-1} \cM_X = 0$, where
  $c = \dim Y - \dim X$, all the conditions in \theoremref{thm:reflexive-pure}
  hold in our context, provided we manage to prove the inequalities
  \[
    \dim \Supp ℋ^ℓ \gr_0^F \DR(\Mmod_X) ≤ -(ℓ+2)
  \]
  for every number $ℓ ≥ -n+1$.  But we have
  \begin{align*}
    -(ℓ+2) & ≥ \dim \Supp R^ℓ \sHom_{𝒪_Y}\bigl(f_* Ω^n_{\wtilde{X}}, ω_Y^•\bigr) && \text{by \corollaryref{cor:ccor-ng}} \\
           & = \dim \Supp R^ℓ \sHom_{𝒪_Y} \bigl( \gr_{-n}^F \DR(\Mmod_X), ω_Y^•
             \bigr) && \text{by \propositionref{prop:spidz}} \\
           & = \dim \Supp ℋ^ℓ \gr_0^F \DR(\Mmod_X) && \text{by \corollaryref{cor:duality-pure}}
  \end{align*}
  This completes the proof.
\end{proof}

%
%
\svnid{$Id: S09-proof-mixed.tex 269 2020-01-20 11:28:53Z kebekus $}

\section{Mixed Hodge modules and log differentials on the resolution}
\label{sec:pf-mixed}
\subversionInfo
\approvals{Christian & yes \\Stefan & yes}

We maintain the assumptions and notation of \settingref{setting:7-1}.  While the
direct images of $Ω^p_{\wtilde{X}}$ are described in terms of the pure Hodge
modules discussed in the previous \parref{ssec:7.pm}, the study of logarithmic
differentials requires us to look at certain mixed Hodge modules.  As with
\parref{ssec:tcomHm}, we feel that readers who are primarily interested in
extension results for (non-logarithmis) $p$-forms, \theoremref{thm:main-new} and
related results, might consider skipping this section on first reading.

\subsection{The mixed Hodge module on the complement of the exceptional divisor}
\approvals{Christian & yes \\Stefan & yes}

Recall that $X$ is a reduced complex space of pure dimension $n$, and that
$r \colon \wtilde{X} → X$ is a log resolution with exceptional divisor $E$.  We
denote by $j \colon \wtilde{X} ∖ E ↪ \wtilde{X}$ the open embedding of the
complement of the normal crossing divisor $E$.  By analogy with the argument in
\parref{ssec:tchmXt}, we consider the constant Hodge module
$ℚ_{\wtilde{X} ∖ E}^H \decal{n}$ on the complement of $E$, and its extension to
a mixed Hodge module
\[
  j_* ℚ_{\wtilde{X} ∖ E}^H \decal{n} ∈ \MHM(\wtilde{X})
\]
on $\wtilde{X}$, as discussed in \cite[Thm.~3.27]{Saito:MixedHodgeModules}.  For
the reader's convenience, we summarise its main properties, properly translated
to our convention of using left $𝒟$-modules.

\subsubsection{Perverse sheaf and filtered $𝒟$-module}
\approvals{Christian & yes \\Stefan & yes}

The underlying perverse sheaf of the mixed Hodge module
$j_* ℚ_{\wtilde{X} ∖ E}^H \decal{n}$ is, by construction,
$\derR j_* ℚ_{\wtilde{X} ∖ E} \decal{n}$.  The underlying regular holonomic
$𝒟_{\wtilde{X}}$-module is $𝒪_{\wtilde{X}}(*E)$, the sheaf of meromorphic
functions on the complex manifold $\wtilde{X}$ that are holomorphic outside the
normal crossing divisor $E$.  The Hodge filtration is given by
\[
  F_p 𝒪_{\wtilde{X}}(\ast E) =
  \begin{cases}
    0 & \text{if $p ≤ -1$,} \\
    F_p 𝒟_{\wtilde{X}} · 𝒪_{\wtilde{X}}(E) & \text{if $p ≥ 0$.}
  \end{cases}
\]
The de Rham complex of $𝒪_{\wtilde{X}}(\ast E)$ is the complex of meromorphic
differential forms
\[
  \DR \bigl( 𝒪_{\wtilde{X}}(\ast E) \bigr) = \Bigl\lbrack 𝒪_{\wtilde{X}}(\ast E)
  \xrightarrow{d} Ω¹_{\wtilde{X}}(\ast E) \xrightarrow{d} \dotsb \xrightarrow{d}
  Ω^n_{\wtilde{X}}(\ast E) \Bigr\rbrack \decal{n},
\]
placed in degrees $-n, …, 0$ as always.  Saito
\cite[Prop.~3.11]{Saito:MixedHodgeModules} has shown that this complex, with the
filtration induced by $F_p 𝒪_{\wtilde{X}}(*E)$, is filtered quasi-isomorphic to
the log de Rham complex $Ω^•_{\wtilde{X}}(\log E) \decal{n}$, with the usual
filtration by degree; in fact, the Hodge filtration on $𝒪_{\wtilde{X}}(*E)$ is
defined so as to make this true.

\begin{prop}\label{fact:log-DR}
  Maintaining \settingref{setting:7-1} and using the notation introduced above,
  the natural inclusion
  $Ω^•_{\wtilde{X}}(\log E) \decal{n} ↪ \DR \bigl( 𝒪_{\wtilde{X}}(*E) \bigr)$ is
  a filtered quasi-isomorphism.  In particular, we have canonical isomorphisms
  \[
    Ω^p_{\wtilde{X}}(\log E) \decal{n-p} %
    ≅ \gr_{-p}^F Ω^•_{\wtilde{X}}(\log E) \decal{n} %
    ≅ \gr_{-p}^F \DR \bigl( 𝒪_{\wtilde{X}}(\ast E) \bigr).  \eqno\qed
  \]
\end{prop}

\subsubsection{Weight filtration}
\label{subsubsec:weights}
\approvals{Christian & yes \\Stefan & yes}

The weight filtration on the mixed Hodge module
$j_* ℚ_{\wtilde{X} ∖ E}^H \decal{n}$ is governed by how the components of the
normal crossing divisor $E$ intersect.  Since this fact is not explicitly
mentioned in \cite[Thm.~3.27]{Saito:MixedHodgeModules}, we include a precise
statement and a proof.

\begin{prop}[Description of weight filtration]\label{prop:wf}
  Maintaining \settingref{setting:7-1} and using the notation introduced above,
  the first pieces of the weight filtration on the mixed Hodge module
  $j_* ℚ_{\wtilde{X} ∖ E}^H \decal{n}$ of the filtrations are given by
  \[
    W_{n-1} \: j_* ℚ_{\wtilde{X} ∖ E}^H \decal{n} = 0
    \qquad \text{and} \qquad
    W_n \: j_* ℚ_{\wtilde{X} ∖ E}^H \decal{n} ≅
    ℚ_{\wtilde{X}}^H \decal{n}.
  \]
  Likewise, for $ℓ ≥ 1$, the Hodge module
  $\gr_{n+ℓ}^W j_* ℚ_{\wtilde{X} ∖ E}^H \decal{n} ∈ \HM(\wtilde{X}, n+ℓ)$ is
  isomorphic to the direct sum, over all subsets $J ⊆ I$ of size $ℓ$, of the
  Hodge modules
  \[
    ℚ_{E_J}^H(-ℓ) \decal{n-ℓ} ∈ \HM(E_J, n+ℓ),
  \]
  pushed forward from the complex submanifold $E_J := \bigcap_{i ∈ J} E_i$ into
  $\wtilde{X}$.
\end{prop}
\begin{proof}
  One possibility is to factor $j_*$ as a composition of open embeddings over
  the irreducible components of the simple normal crossing divisor $E$, as in
  \cite[Thm.~3.27]{Saito:MixedHodgeModules}.  Here, we explain a different
  argument, based on Saito's computation of the nearby cycles functor in the
  normal crossing case \cite[Thm.~3.3]{Saito:MixedHodgeModules}.
  
  To begin with, we observe that the weight filtration on a graded-polarisable
  mixed Hodge module is, even locally, unique: the reason is that there are no
  nontrivial morphisms between polarisable Hodge modules of different weights.
  This reduces the problem to the case where $\wtilde{X}$ is a polydisk, say
  with coordinates $x_1, …, x_n$, and where $E$ is the divisor
  $g = x_1⋯x_r = 0$.  Moreover, it is enough to prove the statement for the
  underlying $𝒟$-modules.  Indeed, by \cite[Thm.~3.21]{Saito:HodgeModules},
  every polarisable Hodge module on $\wtilde{X}$, whose underlying $𝒟$-module is
  the direct image of $𝒪_{E_J}$, comes from a polarisable variation of Hodge
  structure on $E_J$, hence must be isomorphic to the push forward of
  $ℚ_{E_J}^H(k)$ for some $k ∈ ℤ$.  The Tate twist is then determined by the
  weight, because $n + ℓ = \dim E_J + k$.
  
  After embedding $\wtilde{X}$ into $\wtilde{X}⨯ℂ$, via the graph of
  $g = x_1⋯x_r$, we have, according to
  \cite[(2.11.10)]{Saito:MixedHodgeModules}, that
  \[
    \gr_{n+ℓ}^W j_* ℚ_{\wtilde{X} ∖ E}^H \decal{n} ≅
    \begin{cases}
      0 & \text{if $ℓ < 0$,} \\
      ℚ_{\wtilde{X}}^H \decal{n} & \text{if $ℓ = 0$,} \\
      P_N \gr_{n + ℓ - 2}^W ψ_{g,1} ℚ_{\wtilde{X}}^H \decal{n}(-1)
      & \text{if $ℓ > 0$,}
    \end{cases}
  \]
  where $ψ_{g,1}$ denotes the nearby cycles functor (with respect to the
  coordinate function $t$ on $\wtilde{X} ⨯ ℂ$).  In our normal crossing setting,
  the nearby cycles functor is computed explicitly in
  \cite[Thm.~3.3]{Saito:MixedHodgeModules}.  In the notation introduced in
  \cite[§3.4]{Saito:MixedHodgeModules}, the right $𝒟_{\wtilde{X}}$-module
  associated to $𝒪_{\wtilde{X}}$ is isomorphic to $M(μ, \varnothing)$, where
  $μ = (-1, …, -1) ∈ ℤ^n$.  By \cite[(3.5.4)]{Saito:MixedHodgeModules}, the
  right $𝒟_{\wtilde{X}}$-module underlying
  $P_N \gr_{n + ℓ - 2}^W ψ_{g,1} ℚ_{\wtilde{X}}^H \decal{n}(-1)$ is therefore
  isomorphic to the direct sum of $M(μ, J)$, where $J ⊆ \{1, …, r\}$ runs over
  all subsets of size $ℓ$.  But $M(μ, J)$ is exactly the right
  $𝒟_{\wtilde{X}}$-module associated to the push forward of $𝒪_{E_J}$, and so we
  get the desired result.
\end{proof}

\subsection{Push forward to $Y$}
\approvals{Christian & yes \\Stefan & yes}

Recall that $f \colon \wtilde{X} → Y$ is the projective holomorphic mapping
obtained by composing our resolution of singularities $r \colon \wtilde{X} → X$
with the closed embedding $i_X \colon X ↪ Y$.  We now define a family of mixed
Hodge modules $N_ℓ ∈ \MHM(Y)$, indexed by $ℓ ∈ ℤ$, by setting
\[
  N_ℓ := H^ℓ f_* \Bigl( j_* ℚ_{\wtilde{X} ∖ E}^H \decal{n} \Bigr).
\]
Note that each $N_ℓ$ is again a graded-polarisable mixed Hodge module on $Y$,
due to the fact that $f$ is a projective morphism (see
\theoremref{thm:direct-image}).  Clearly, $\Supp N_0 = X$, and
$\Supp N_ℓ ⊆ X_{\sing}$ for $ℓ ≠ 0$.

\begin{lem}\label{lem:Nl}
  Maintaining \settingref{setting:7-1} and using the notation introduced above,
  we have $N_ℓ = 0$ for $ℓ ≤ -1$.  The mixed Hodge module $N_0$ has no
  nontrivial subobjects whose support is contained in $X_{\sing}$.
\end{lem}
\begin{proof}
  It suffices to prove this for the underlying perverse sheaves $\rat N_ℓ$.  By
  construction, $\rat N_ℓ$ is the $ℓ$-th perverse cohomology sheaf of the
  constructible complex
  \[
    \derR f_* \bigl( j_* ℚ_{\wtilde{X} ∖ E} \decal{n} \bigr) %
    ≅ \derR j_* ℚ_{X_{\reg}} \decal{n}.
  \]
  Now, if $K ∈ \Dbc(ℚ_X)$ is any constructible complex, then
  \[
    \Hom_{\Dbc(ℚ_X)} \Bigl( K,\, \derR j_* ℚ_{X_{\reg}} \decal{n} \Bigr) %
    ≅ \Hom_{\Dbc(ℚ_{X_{\reg}})} \Bigl( j^{-1} K,\, ℚ_{X_{\reg}} \decal{n}
    \Bigr),
  \]
  and the right-hand side vanishes if $\Supp K ⊆ X_{\sing}$.  The first
  assertion of \lemmaref{lem:Nl} thus follows by taking
  $K = \rat N_ℓ \decal{-ℓ}$ for $ℓ ≤ -1$.  Once it is known that $N_ℓ = 0$ for
  $ℓ ≤ -1$, the second assertion follows by taking $K$ to be any subobject of
  $\rat N_0$.
\end{proof}

Each mixed Hodge module $N_ℓ$ has weight $≥ n+ℓ$, in the following sense.

\begin{lem}\label{lem:Nell-weight}
  Maintaining \settingref{setting:7-1} and using the notation introduced above,
  we have $W_{n+ℓ-1} N_ℓ = 0$.  The module $W_{n+ℓ} N_ℓ$ is a quotient of
  $H^ℓ f_* ℚ_{\wtilde{X}}^H \decal{n}$.
\end{lem}
\begin{proof}
  This is proved in \cite[Prop.~2.26]{Saito:MixedHodgeModules}.  For the
  convenience of the reader, we explain how to deduce it from the degeneration
  of the weight spectral sequence in \theoremref{thm:fl-mixed}.  Since $f$ is a
  projective morphism, the weight spectral sequence
  \[
    E_1^{p,q} = H^{p+q} f_* \gr_{-p}^W j_* ℚ_{\wtilde{X} ∖ E}^H \decal{n}
    \Longrightarrow N_{p+q}
  \]
  degenerates at $E_2$, and the induced filtration on $N_ℓ$ is the weight
  filtration $W_• N_ℓ$.  More precisely, $E_1^{p,q}$ and $E_2^{p,q}$ are Hodge
  modules of weight $q$, and
  \[
    \gr_q^W N_{p+q} ≅ E_2^{p,q}.
  \]
  As $j_* ℚ_{\wtilde{X} ∖ E}^H \decal{n}$ has weight $≥ n$, we have
  $E_1^{p,q} = 0$ for $p ≥ -n+1$, whence $\gr_w^W N_ℓ = 0$ for $w ≤ n+ℓ-1$.
  This also shows that $W_{n+ℓ} N_ℓ$ is a quotient of
  $E_1^{-n,n+ℓ} = H^ℓ f_* ℚ_{\wtilde{X}}^H \decal{n}$.
\end{proof}

\subsection{Relation with logarithmic differentials on the resolution}
\approvals{Christian & yes \\Stefan & yes}

Now we can relate the coherent $𝒪_Y$-module $f_* Ω^p_{\wtilde{X}}(\log E)$ to
the de Rham complex of the mixed Hodge module $N_0$.  In line with the notation
used before, write $(\Nmod_ℓ, F_• \Nmod_ℓ)$ for the filtered regular holonomic
$𝒟_Y$-module underlying the mixed Hodge module $N_ℓ$.

\begin{prop}\label{prop:Xwrx}
  Maintaining \settingref{setting:7-1} and using the notation introduced above,
  we have $f_* Ω^p_{\wtilde{X}}(\log E) ≅ ℋ^{p-n} \gr_{-p}^F \DR(\Nmod_0)$ for
  every $p ∈ ℤ$.
\end{prop}
\begin{proof}
  Fix an integer $p ∈ ℤ$.  \propositionref{fact:log-DR}, together with
  \propositionref{prop:fl-grDR} about the compatibility of the de Rham complex
  with direct images, implies that
  \[
    \derR f_* Ω^p_{\wtilde{X}}(\log E) \decal{n-p} ≅ \gr_{-p}^F \DR \Bigl( f_{+}
    \bigl( R_F 𝒪_{\wtilde{X}}(\ast E) \bigr) \Bigr).
  \]
  Because the complex computing the direct image is strict by
  \theoremref{thm:fl-mixed}, we have a convergent spectral sequence
  \[
    E_2^{a,b} = ℋ^a \gr_{-p}^F \DR(\Nmod_b) \Longrightarrow ℋ^{a+b} \gr_{-p}^F
    \DR \Bigl( f_{+} \bigl( R_F 𝒪_{\wtilde{X}}(\ast E) \bigr) \Bigr),
  \]
  and we are interested in the terms with $a+b = p-n$.  \propositionref{prop:Fc}
  guarantees that $F_{c-1} \Nmod_ℓ = 0$ for every $ℓ ∈ ℤ$, whence
  $E_2^{a,b} = 0$ for $a ≤ p-n-1$.  Also, $\Nmod_ℓ = 0$ for $ℓ ≤ -1$ by
  \lemmaref{lem:Nl}, and so $E_2^{a,b} = 0$ for $b ≤ -1$.  The spectral sequence
  therefore gives us the desired isomorphism.
\end{proof}

The analysis of the higher direct images quickly gets complicated.  For that
reason, we shall only consider what happens in the case of $1$-forms with log
poles.  Here, one has the following simple relation between
$\derR f_* Ω¹_{\wtilde{X}}(\log E)$ and the complex $\gr_{-1}^F \DR(\Nmod_0)$.

\begin{prop}\label{prop:1-forms-N0}
  Maintaining \settingref{setting:7-1} and using the notation introduced above,
  we have a canonical isomorphism
  \[
    \derR f_* Ω¹_{\wtilde{X}}(\log E) \decal {n-1} %
    ≅ \gr_{-1}^F \DR(\Nmod_0).
  \]
  In particular,
  $R^{n-1} f_* Ω¹_{\wtilde{X}}(\log E) ≅ ℋ⁰ \gr_{-1}^F \DR(\Nmod_0)$.
\end{prop}

The proof of \propositionref{prop:1-forms-N0} relies on the following lemma,
which we discuss first.

\begin{lem}\label{lem:grF-1-acyclic}
  Maintaining \settingref{setting:7-1} and using the notation introduced above,
  the complex $\gr_{-1}^F \DR(\Nmod_ℓ)$ is acyclic for every $ℓ ≠ 0$.
\end{lem}
\begin{proof}
  Recall from \lemmaref{lem:Nell-weight} that $N_ℓ ∈ \MHM(Y)$ has weight
  $≥ n+ℓ$, which means that $\gr_w^W N_ℓ = 0$ for $w ≤ n+ℓ-1$.
  \propositionref{prop:Fc} guarantees that $F_{c-1} \Nmod_ℓ = 0$ for every
  $ℓ ≥ 0$.  This implies that $F_{c-1} \gr_w^W \Nmod_ℓ = 0$ for every $w ∈ ℤ$.
  According to \corollaryref{cor:duality-pure}, we have
  \[
    \gr_{-1}^F \DR(\gr_w^W \!  \Nmod_ℓ) %
    ≅ \derR \sHom_{𝒪_Y} \Bigl( \gr_{1-w}^F \DR(\gr_w^W \!  \Nmod_ℓ), ω_Y^•
    \Bigr),
  \]
  and the complex $\gr_{1-w}^F \DR(\gr_w^W \!  \Nmod_ℓ)$ only uses the
  $𝒪_Y$-modules $\gr_p^F \gr_w^W \!  \Nmod_ℓ$ in the range
  \[
    1-w ≤ p ≤ 1-w + \dim Y = c + 1 - ℓ - \bigl( w - (n+ℓ) \bigr).
  \]
  As $F_{c-1} \gr_w^W \!  \Nmod_ℓ = 0$ and $ℓ ≥ 1$, we see that
  $\gr_{1-w}^F \DR(\gr_w^W \!  \Nmod_ℓ) = 0$, except maybe in the special case
  $w = n+ℓ$.  But by the $E_2$-degeneration of the weight spectral sequence,
  $\gr_{n+ℓ}^W N_ℓ$ is a quotient of $M_ℓ = H^ℓ f_* ℚ_{\wtilde{X}}^H \decal{n}$,
  and since we already know that $F_c \Mmod_ℓ = 0$, we also have
  $F_c \gr_{n+ℓ}^W \Nmod_ℓ = 0$ for $ℓ ≥ 1$.  This proves that
  $\gr_{-1}^F \DR(\gr_w^W \!  \Nmod_ℓ)$ is acyclic for every $ℓ ≥ 1$ and every
  $w ∈ ℤ$.  Since the functor $\gr_{-1}^F \DR$ is exact on mixed Hodge modules,
  it follows that the complex $\gr_{-1}^F \DR(\Nmod_ℓ)$ is also acyclic.
\end{proof}

\begin{proof}[Proof of \propositionref{prop:1-forms-N0}]
  Because $\Nmod_j = 0$ for $j ≤ -1$, and because the complex computing the
  direct image is strict by \theoremref{thm:fl-mixed}, we have a canonical
  morphism
  \[
    (\Nmod_0, F_• \Nmod_0) → f_+ \bigl( R_F 𝒪_{\wtilde{X}}(\ast E) \bigr)
  \]
  in the derived category $\Dbcoh G(R_F 𝒟_Y)$.  As a first step, we are going to
  show that the induced morphism
  \begin{equation}\label{eq:mor-ind}
    \gr_{-1}^F \DR(\Nmod_0) → \gr_{-1}^F \DR
    \Bigl( f_+ \bigl( R_F 𝒪_{\wtilde{X}}(\ast E) \bigr) \Bigr)
  \end{equation}
  between complexes of $𝒪_Y$-modules is a quasi-isomorphism.
  \lemmaref{lem:grF-1-acyclic} implies that the spectral sequence
  \[
    E_2^{a,b} = ℋ^a \gr_{-1}^F \DR(\Nmod_b) \Longrightarrow ℋ^{a+b} \gr_{-1}^F
    \DR \Bigl( f_+ \bigl( R_F 𝒪_{\wtilde{X}}(\ast E) \bigr) \Bigr),
  \]
  degenerates at $E_2$, and so we have a collection of isomorphisms
  \[
    \mathscr{H}^a \gr_{-1}^F \DR(\Nmod_0) ≅ \mathscr{H}^a \gr_{-1}^F \DR \Bigl(
    f_+ \bigl( R_F 𝒪_{\wtilde{X}}(\ast E) \bigr) \Bigr).
  \]
  These isomorphisms are induced by the morphism in \eqref{eq:mor-ind}, which is
  therefore a quasi-isomorphism.  Now the compatibility of the de Rham complex
  with direct images, together with \propositionref{fact:log-DR}, implies that
  \[
    \gr_{-1}^F \DR(\Nmod_0) %
    ≅ \derR f_* \gr_{-1}^F \DR \bigl( 𝒪_{\wtilde{X}}(\ast E) \bigr) %
    ≅ \derR f_* Ω¹_{\wtilde{X}}(\log E) \decal{n-1},
  \]
  as asserted by the proposition.
\end{proof}

\subsection{The weight filtration on $N_0$}
\approvals{Christian & yes \\Stefan & yes}
\label{ssec:twfoN0}

We describe how the weight filtration interacts with the complex
$\gr_{-1}^F \DR(\Nmod_0)$.

\begin{prop}[The complex $\gr_{-1}^F \DR(\Nmod_0)$]\label{prop:grF-1-N0}
  Maintaining \settingref{setting:7-1} and using the notation introduced above,
  the complex $\gr_{-1}^F \DR(\gr_w^W \Nmod_0)$ is acyclic for $w ∉ \{n, n+1\}$
  and
  \begin{align}
    \gr_{-1}^F \DR(\gr_n^W \!  \Nmod_0) &≅ \gr_{-1}^F \DR(\Mmod_X) \label{eq:Wn} \\
    \gr_{-1}^F \DR(\gr_{n+1}^W \!  \Nmod_0) &≅ \bigoplus_{i ∈ I} \derR f_* 𝒪_{E_i} \decal{n-1}.  \label{eq:Wn+1}
  \end{align}
\end{prop}
\begin{proof}
  Consider again the weight spectral sequence
  \[
    E_1^{p,q} = H^{p+q} f_* \gr_{-p}^W j_* ℚ_{\wtilde{X} ∖ E}^H \decal{n}
    \Longrightarrow N_{p+q}.
  \]
  Because $f$ is projective, the spectral sequence degenerates at $E_2$, and the
  induced filtration on $N_ℓ$ is the weight filtration $W_• N_ℓ$, see
  \theoremref{thm:fl-mixed}.  More precisely, what happens is that $E_1^{p,q}$
  and $E_2^{p,q}$ are polarisable Hodge modules of weight $q$, and
  \[
    \gr_q^W \!  N_{p+q} ≅ E_2^{p,q}.
  \]
  Now $j_* ℚ_{\wtilde{X} ∖ E}^H \decal{n}$ has weight $≥ n$, and so
  $E_1^{p,q} = 0$ for $p ≥ -n+1$, and $W_{n-1} N_0 = 0$.  Moreover, $W_n N_0$ is
  the cokernel of the morphism $d_1 \colon E_1^{-n-1,n} → E_1^{-n,n}$.  Using
  the description of the weight filtration in \propositionref{prop:wf}, we
  compute that
  \[
    E_1^{-n,n} = H⁰ f_* \gr_n^W j_* ℚ_{\wtilde{X} ∖ E}^H \decal{n} %
    ≅ H⁰ f_* ℚ_{\wtilde{X}}^H \decal{n} ≅ M_X ⊕ M_0
  \]
  and that the support of $E_1^{-n+1,n}$ is contained inside $X_{\sing}$.
  Because $N_0$ has no subobjects that are supported inside $X_{\sing}$ (by
  \lemmaref{lem:Nl}) , and $M_X$ has neither subobjects nor quotient objects
  that are supported inside $X_{\sing}$ (by construction), we conclude that
  $W_n N_0 ≅ M_X$.  This already proves \eqref{eq:Wn}.

  Likewise, $\gr_{n+1}^W \!  N_0$ is the cohomology of the complex of Hodge
  modules of weight $n+1$
  \begin{equation} \label{eq:cx-HM}
    \begin{tikzcd}
      E_1^{-n-2,n+1} \rar{d_1} & E_1^{-n-1,n+1} \rar{d_1} & E_1^{-n,n+1}.
    \end{tikzcd}
  \end{equation}
  By a similar computation as above, we have $E_1^{-n,n+1} ≅ M_1$ and
  \begin{align*}
    E_1^{-n-1,n+1} &≅ \bigoplus_{i ∈ I} H⁰ f_* ℚ_{E_i}^H (-1) \, \decal{n-1} \\
    E_1^{-n-2,n+1} &≅ \bigoplus_{i,j ∈ I} H^{-1} f_* ℚ_{E_i ∩ E_j}^H (-2) \, \decal{n-2}.
  \end{align*}
  We showed during the proof of \propositionref{prop:xa2} that
  $\gr_{-1}^F \DR(\Mmod_1)$ is acyclic.  At the same time, using the
  compatibility of the de Rham complex with direct images, we have
  \[
    \gr_{-1}^F \DR(ℰ_1^{-n-1,n+1}) %
    ≅ \bigoplus_{i ∈ I} \derR f_* \gr_0^F \DR(𝒪_{E_i}) %
    ≅ \bigoplus_{i ∈ I} \derR f_* 𝒪_{E_i} \decal{n-1}.
  \]
  By a similar calculation and the Decomposition Theorem, the complex
  $\gr_{-1}^F \DR(ℰ_1^{-n-2,n+1})$ is isomorphic to a direct summand in
  \[
    \bigoplus_{i,j ∈ I} \derR f_* \gr_1^F \DR(𝒪_{E_i ∩ E_j})
  \]
  and therefore acyclic.  Since morphisms between mixed Hodge modules strictly
  preserve the Hodge filtration, it now follows from \eqref{eq:cx-HM} that
  \[
    \gr_{-1}^F \DR(\gr_{n+1}^W \!  \Nmod_0) %
    ≅ \bigoplus_{i ∈ I} \derR f_* 𝒪_{E_i} \decal{n-1},
  \]
  proving \eqref{eq:Wn+1}.

  Since $W_{n-1} N_0 = 0$, the complex $\gr_{-1}^F \DR(\gr_w^W \!  \Nmod_0)$ is
  certainly acyclic for $w ≤ n-1$.  It remains to show that it is also acyclic
  for $w ≥ n+2$.  The proof of this fact is the same as that of
  \lemmaref{lem:grF-1-acyclic}, and so we omit it.
\end{proof}

\begin{cor}\label{cor:x7}
  Maintaining \settingref{setting:7-1} and using the notation introduced above,
  we obtain a long exact sequence
  \[
    \dotsb → ℋ^j \gr_{-1}^F \DR(\Mmod_X) %
    → ℋ^j \gr_{-1}^F \DR(\Nmod_0) %
    → \bigoplus_{i ∈ I} R^{n-1+j} f_* 𝒪_{E_i}
	 → \dotsb
  \]
\end{cor}
\begin{proof}
  \propositionref{prop:grF-1-N0} implies that the complex
  $\gr_{-1}^F \DR(W_{n-1} \Nmod_0)$ is acyclic, and that the natural morphism
  \[
    \gr_{-1}^F \DR(W_{n+1} \Nmod_0) → \gr_{-1}^F \DR(\Nmod_0)
  \]
  is a quasi-isomorphism.  We therefore get a distinguished triangle
  \[
    \gr_{-1}^F \DR(\Mmod_X) → \gr_{-1}^F \DR(\Nmod_0) → \bigoplus_{i ∈ I}
    \derR f_* 𝒪_{E_i} \decal{n-1} →
    \gr_{-1}^F \DR(\Mmod_X) \decal{1}
  \]
  in the derived category $\Dbcoh(𝒪_Y)$.  The claim follows by passing to cohomology.
\end{proof}

\subsection{Application to the extension problem}
\label{sec:08eoetflog}
\approvals{Christian & yes \\Stefan & yes}

In analogy with \parref{sec:08eoetf}, we conclude with a brief discussion of the
effect that extendability of log $n$-forms has on $\DR(\Nmod_0)$.  Once again,
\corollaryref{cor:reflexive-mixed-ng} and the result below can be used to show
if $n$-forms extend with log poles, then all forms extend with log poles.  This
gives another proof for \theoremref{thm:main-log} in the (most important) case
$k = n$.  Since we are now working with mixed Hodge modules, the reader may find
it instructive to compare the proof below with that of the analogous result for
pure Hodge modules in \parref{sec:08eoetf}

\begin{prop}[Extension of log $n$-forms and $N_Y$] \label{prop:main-ineqlog}
  Maintaining \settingref{setting:7-1} and using the notation introduced above,
  assume that the morphism $r_* Ω^n_{\wtilde{X}}(\log E) ↪ j_* Ω^n_{X_{\reg}}$
  is an isomorphism.  Then one has
  \[
    \dim \Supp ℋ^j \gr_p^F \DR(\Nmod_Y) ≤ -(j+p+2)
  \]
  for all integers $j, p ∈ ℤ$ with $p+j ≥ -n + 1$.
\end{prop}
\begin{proof}
  This time, we aim to apply \theoremref{thm:reflexive-mixed}.  Recall that $X$
  is reduced of pure dimension $n$; that the mixed Hodge module $N_0 ∈ \MHM(Y)$
  has support equal to $X$; and that we defined $N_Y := 𝔻(N_0)(-n) ∈ \MHM(Y)$ by
  taking the $(-n)$-th Tate twist of the dual mixed Hodge module.  Taking into
  account the Tate twist, the formula for the de Rham complex of the dual mixed
  Hodge module in \propositionref{prop:duality-mixed} becomes
  \begin{equation}\label{eq:rel-NY-N0}
    \gr_p^F \DR(\Nmod_Y) %
    ≅ \derR \sHom_{𝒪_Y} \Bigl( \gr_{-(p+n)}^F \DR(\Nmod_0), ω_Y^• \Bigr).
  \end{equation}
  Let us now verify that all the conditions in \theoremref{thm:reflexive-mixed}
  are satisfied in our setting.

  \begin{claim}\label{cl:mixed-2a}
    One has $\dim \Supp ℋ^j \DR(\Nmod_Y) ≤ -(j+1)$ for every $j ≥ -n + 1$.
  \end{claim}
  \begin{proof}[Proof of \claimref{cl:mixed-2a}]
    Recall that the module $N_0$ has weight $≥ n$, in the sense that
    $W_{n-1} N_0 = 0$, and that its support is $\Supp N_0 = X$.  The dual module
    $N_Y$ will then have weight $≤ n$, in the sense that $W_n N_Y = N_Y$, and
    $\Supp N_Y = X$.  By \lemmaref{lem:Nl}, the perverse sheaf $\DR(\Nmod_0)$
    has no nontrivial subobjects whose support is contained in $X_{\sing}$.
    Consequently, the perverse sheaf $\DR(\Nmod_Y)$, isomorphic to the Verdier
    dual of $\DR(\Nmod_0)$, has no nontrivial quotient objects whose support is
    contained in $X_{\sing}$.  Now apply
    \propositionref{prop:perverse-quotients}.  \qedhere~(\claimref{cl:mixed-2a})
  \end{proof}
  
  \begin{claim}\label{cl:mixed-3a}
    The complex of $𝒪_Y$-modules $\gr_p^F \DR(\Nmod_Y)$ is acyclic for every
    $p ≥ 1$.
  \end{claim}
  \begin{proof}[Proof of \claimref{cl:mixed-3a}]
    Recall that $F_{c-1} \Nmod_0 = 0$, where $c = \dim Y - \dim X$.  For
    dimension reasons, the complex $\gr_{-(p+n)}^F \DR(\Nmod_0)$ is trivial for
    $p ≥ 1$.  Now \eqref{eq:rel-NY-N0} implies that the complex
    $\gr_p^F \DR(\Nmod_Y)$ is acyclic for every $p ≥ 1$.
    \qedhere~(\claimref{cl:mixed-3a})
  \end{proof}
    
  \begin{claim}\label{cl:mixed-4a}
    One has $\dim \Supp ℋ^j \gr_0^F \DR(\Nmod_Y) ≤ -(j+2)$ for every
    $j ≥ -n + 1$.
  \end{claim}
  \begin{proof}[Proof of \claimref{cl:mixed-4a}]
    Since $F_{c-1} \Nmod_0 = 0$, the formula in \eqref{eq:Ax} implies that the
    complex
    \begin{equation}\label{eq:N0-shf}
      \gr_{-n}^F \DR(\Nmod_0) ≅ ℋ⁰ \gr_{-n}^F \DR(\Nmod_0)
    \end{equation}
    is actually a sheaf in degree $0$.  Using the assumption that
    $r_* Ω^n_{\wtilde{X}}(\log E) ≅ j_* Ω^n_{X_{\reg}}$, the following
    inequalities will therefore hold for all $j ≥ -n+1$:
    \begin{align*}
      -(j+2) & ≥ \dim \Supp R^j \sHom_{𝒪_Y} \bigl( f_* Ω^n_{\wtilde{X}}(\log E), ω^•_Y \bigr) && \text{by \corollaryref{cor:ccor-ng}} \\
             & = \dim \Supp R^j \sHom_{𝒪_Y} \bigl( ℋ⁰ \gr_{-n}^F \DR(\cN_0), ω^•_Y \bigr) && \text{by \propositionref{prop:Xwrx}} \\
             & = \dim \Supp R^j \sHom_{𝒪_Y} \bigl( \gr_{-n}^F \DR(\cN_0), ω_Y^• \bigr) && \text{by \eqref{eq:N0-shf}} \\
             & = \dim \Supp ℋ^j \gr_0^F \DR(\Nmod_Y) && \text{by \eqref{eq:rel-NY-N0}}
    \end{align*}
    This gives us the desired result.  \qedhere~(\claimref{cl:mixed-4a})
  \end{proof}
  
  Having checked all the conditions, we can now apply
  \theoremref{thm:reflexive-mixed} and conclude the proof of
  \propositionref{prop:main-ineqlog}.
\end{proof}

%
%
\svnid{$Id: S10-proof-ext.tex 269 2020-01-20 11:28:53Z kebekus $}

\section{Intrinsic description, proof of Theorems~\ref*{thm:extension-Kahler} and \ref*{thm:extension-Kahler-log}}
\subversionInfo
\label{par:short}

\subsection{Proof of Theorem~\ref*{thm:extension-Kahler}}
\approvals{Christian & yes \\Stefan & yes}

In this section, we prove the criterion for extension of holomorphic forms in
\theoremref{thm:extension-Kahler}.  In fact, the result is really just a
reformulation of \propositionref{prop:7-1-ng}, although it takes some work to
see that this is the case.

\subsection*{Setup}
\approvals{Christian & yes \\Stefan & yes}

Let $X$ be a reduced complex space of pure dimension $n$.  Since the statement
to be proved is local on $X$, we may assume that we are in the setting described
in \parref{sec:pf-setup}.  In particular, $X$ is a complex subspace of an open
ball $Y ⊆ ℂ^{n+c}$, and $f \colon \wtilde{X} → Y$ denotes the composition of a
projective resolution of singularities $r \colon \wtilde{X} → X$ with the closed
embedding $i_X \colon X ↪ Y$.  Because $Y$ is a Stein manifold, all Kähler
differentials on $X$ are restrictions of holomorphic differential forms from
$Y$; in particular, if $z_1, …, z_{n+c}$ are holomorphic coordinates on $Y$,
then the sheaf $Ω^p_X$ is generated by the global sections
\[
  i_X^*(\dz_{i_1} Λ \dotsb Λ \dz_{i_p}),
\]
where $1 ≤ i_1 < i_2 < \dotsb < i_p ≤ n+c$.  Having set up the notation, we can
now prove the following (slightly more precise) local version of
\theoremref{thm:extension-Kahler}.

\begin{thm}[Local version of Theorem~\ref*{thm:extension-Kahler}]
  In the setting above, a holomorphic $p$-form $α ∈ H⁰(X_{\reg}, Ω^p_X)$ extends
  to a holomorphic $p$-form on $\wtilde{X}$ if, and only if, the holomorphic
  $n$-forms $α Λ \dz_{i_1} Λ \dotsb Λ \dz_{i_{n-p}}$ and
  $d α Λ \dz_{i_{1}} Λ \dotsb Λ \dz_{i_{n-p-1}}$ on $X_{\reg}$ extend to
  holomorphic $n$-forms on $\wtilde{X}$, for every choice of indices
  $1 ≤ i_1 ≤ i_2 ≤ \dotsb ≤ i_{n-p} ≤ n+c$.
\end{thm}

\subsection*{The intersection complex}
\approvals{Christian & yes \\Stefan & yes}

As in \parref{ssec:7.pm}, we use the notation $M_X ∈ \HM(Y, n)$ for the
polarisable Hodge module on $Y$ whose underlying perverse sheaf is the
intersection complex of $X$, and we let $(\Mmod_X, F_• \Mmod_X)$ be its
underlying filtered $𝒟_Y$-module.  According to \propositionref{prop:7-1-ng}, we
have
\[
  f_* Ω^p_{\wtilde{X}} ≅ ℋ^{-(n-p)} \gr_{-p}^F \DR(\Mmod_X).
\]
Recall from \parref{ssec:sqdR} that the de Rham complex
\[
  \DR(\Mmod_X) = \Bigl\lbrack \Mmod_X \xrightarrow{∇} Ω¹_Y ⊗ \Mmod_X
  \xrightarrow{∇} \dotsb \xrightarrow{∇} Ω^{n+c}_Y ⊗ \Mmod_X \Bigr\rbrack,
\]
is concentrated in degrees $-(n+c), …, 0$.  Since $\dim Y - \dim X = c$, one has
$F_{c-1} \Mmod_X = 0$, which means that the complex of coherent $𝒪_Y$-modules
\[
  \gr_{-p}^F \DR(\Mmod_X) = \Bigl\lbrack Ω^{p+c}_Y ⊗ F_c \Mmod_X \xrightarrow{∇}
  Ω^{p+c+1}_Y ⊗ \gr_{c+1}^F \Mmod_X \xrightarrow{∇} \dotsb \xrightarrow{∇}
  Ω^{n+c}_Y ⊗ \gr_{n-p+c}^F \Mmod_X \Bigr\rbrack
\]
is concentrated in degrees $-(n-p), …, 0$.  The result in
\propositionref{prop:7-1-ng} therefore becomes
\begin{equation}\label{eq:iso-any}
  f_* Ω^p_{\wtilde{X}} %
  ≅ \ker \Bigl( ∇ \colon Ω^{p+c}_Y ⊗ F_c \Mmod_X → Ω^{p+c+1}_Y ⊗ \gr_{c+1}^F \Mmod_X \Bigr).
\end{equation}
This yields an isomorphism between the space of holomorphic $p$-forms on the
resolution $\wtilde{X}$, and the space of holomorphic $(p+c)$-forms on $Y$ with
coefficients in the coherent $𝒪_Y$-module $F_c \Mmod_X$ whose image under the
differential in the de Rham complex is again a holomorphic $(p+c+1)$-form on $Y$
with coefficients in $F_c \Mmod_X$.  The isomorphism
\begin{equation}\label{eq:iso-top}
f_* Ω_{\wtilde{X}}^n ≅ Ω_Y^{n+c} ⊗ F_c \Mmod_X
\end{equation}
is an important special case of this.

\begin{claim}\label{claim:wedge-test}
  With notation as above, the image of the restriction morphism
  \[
    H⁰ \bigl( Y, Ω^{p+c}_Y ⊗ F_c \Mmod_X \bigr) %
    → H⁰ \bigl( Y ∖ X_{\sing}, Ω^{p+c}_Y ⊗ F_c \Mmod_X \bigr)
  \]
  consists exactly of those $(p+c)$-forms with values in $F_c \Mmod_X$ whose
  wedge product with any element of $H⁰(Y, Ω^{n-p}_Y)$ belongs to the image of
  \[
  H⁰ \bigl( Y, Ω_Y^{n+c} ⊗ F_c \Mmod_X \bigr) %
  → H⁰ \bigl( Y ∖ X_{\sing}, Ω^{n+c}_Y ⊗ F_c \Mmod_X \bigr).
  \]
\end{claim}
\begin{proof}[Proof of \claimref{claim:wedge-test}]
  The isomorphism in \eqref{eq:iso-top} shows that $F_c \Mmod_X$ is a rank-one
  coherent sheaf supported on $X$, whose restriction to $X_{\reg}$ is isomorphic
  to the line bundle $\det N_{X_{\reg} \mid Y}$.  Using the coordinate functions
  $z_1, …, z_{n+c}$ on the ball $Y$, we may write any given element of
  $H⁰ \bigl( Y ∖ X_{\sing}, Ω^{p+c}_Y ⊗ F_c \Mmod_X \bigr)$ uniquely in the form
  \[
    \sum (\dz_{i_1} Λ \dotsb Λ \dz_{i_{p+c}}) ⊗ λ_{i_1, …, i_{p+c}},
  \]
  with coefficients
  $λ_{i_1, …, i_{p+c}} ∈ H⁰ \bigl( Y ∖ X_{\sing}, F_c \Mmod_X \bigr)$.  Clearly
  such an element belongs to the image of the restriction morphism if and only
  if all the coefficients are in the image of $H⁰(Y, F_c \Mmod_X)$.  The
  assertion now follows by taking wedge products with all possible $(n-p)$-forms
  of the type $\dz_{i_1} Λ \dotsb Λ \dz_{i_{n-p}}$.
  \qedhere\quad(\claimref{claim:wedge-test})
\end{proof}

\subsection*{End of proof}
\approvals{Christian & yes \\Stefan & yes}

Now suppose we are given a holomorphic $p$-form $α ∈ H⁰(X_{\reg}, Ω^p_X)$ on the
set of nonsingular points of $X$.  Using the isomorphism in \eqref{eq:iso-any},
it determines a unique element
$\wtilde{α} ∈ H⁰ \bigl( Y ∖ X_{\sing}, Ω^{p+c}_Y ⊗ F_c \Mmod_X \bigr)$ with the
property that
\[
  ∇ \wtilde{α} ∈ H⁰ \bigl( Y ∖ X_{\sing}, Ω^{p+c+1}_Y ⊗ F_c \Mmod_X
  \bigr),
\]
and one checks easily that $∇ \wtilde{α}$ corresponds to the $(p+1)$-form $d α$
under the isomorphism in \eqref{eq:iso-any}.  Again using \eqref{eq:iso-any}, we
conclude that $α$ extends to a holomorphic $p$-form on $\wtilde{X}$ if and only
$\wtilde{α}$ belongs to the image of
\[
  H⁰ \bigl( Y, Ω^{p+c}_Y ⊗ F_c \Mmod_X \bigr) %
  → H⁰ \bigl( Y ∖ X_{\sing}, Ω^{p+c}_Y ⊗ F_c \Mmod_X \bigr)
\]
and $∇ \wtilde{α}$ belongs to the image of
\[
  H⁰ \bigl( Y, Ω^{p+c+1}_Y ⊗ F_c \Mmod_X \bigr) %
  → H⁰ \bigl( Y ∖ X_{\sing}, Ω^{p+c+1}_Y ⊗ F_c \Mmod_X \bigr).
\]
According to \claimref{claim:wedge-test}, we can test for these two conditions
after taking wedge products with elements in $H⁰(Y, Ω^{n-p}_Y)$ respectively
$H⁰(Y, Ω^{n-p-1}_Y)$.  Because the restriction mapping from the differentials on
$Y$ to the Kähler differentials on $X$ is surjective, we get the desired
conclusion.  This ends the proof of \theoremref{thm:extension-Kahler}.  \qed

\subsection{Proof of Theorem~\ref*{thm:extension-Kahler-log}}
\approvals{Christian & yes \\Stefan & yes}

The proof of \theoremref{thm:extension-Kahler-log} is nearly identical to that
of \theoremref{thm:extension-Kahler}.  The only difference is that one has to
work with $Ω^p_{\wtilde{X}}(\log E)$ instead of $Ω^p_{\wtilde{X}}$; that one has
to use the mixed Hodge module $N_0$ instead of the pure Hodge module $M_X$; and
that one should apply \propositionref{prop:Xwrx} instead of
\propositionref{prop:7-1-ng}.  We leave the details to the care of the reader.
\qed

%
%
\svnid{$Id: S11-proof-ext.tex 269 2020-01-20 11:28:53Z kebekus $}

\section{Extension, proof of Theorems~\ref*{thm:main-new} and \ref*{thm:main-log}}
\subversionInfo
\label{sec:pf-1-x}

\subsection{Proof of Theorem~\ref*{thm:main-new}}
\approvals{Christian & yes \\Stefan & yes}

It clearly suffices to prove \theoremref{thm:main-new} only in the case
$p = k-1$, with $1 ≤ k ≤ n$.  Again, we relax the assumptions a little bit and
allow $X$ to be any reduced complex space of pure dimension $n$.  This makes the
entire problem local on $X$.  After shrinking $X$, if necessary, we may
therefore assume that we are given a holomorphic form
$α ∈ H⁰(X_{\reg}, Ω_X^{k-1})$; our task is to show that $α$ extends
holomorphically to the complex manifold $\wtilde{X}$.  We aim to apply
\theoremref{thm:extension-Kahler}, and so we consider an arbitrary open subset
$U ⊆ X$ and a pair of Kähler differentials $β ∈ H⁰(U, Ω^{n-k+1}_X)$ and
$γ ∈ H⁰(U, Ω^{n-k}_X)$.  We need to check that the holomorphic $n$-forms $α Λ β$
and $d α Λ γ$ on $U_{\reg}$ extend to holomorphic $n$-forms on $r^{-1}(U)$.
This is again a local problem, and after further shrinking $X$, we may therefore
assume without loss of generality that $U = X$ and that we have a closed
embedding $i_X \colon X ↪ Y$, where $Y$ is an open ball in $ℂ^{n+c}$.  Letting
$z_1, \dotsc, z_{n+c}$ be holomorphic coordinates on $Y$, the sheaf of Kähler
differentials $Ω_X^p$ is then generated by the global sections
\[
i_X^{\ast}(\dz_{i_1} Λ \dotsb Λ \dz_{i_p}),
\]
where $1 ≤ i_1 < i_2 < \dotsb < i_p ≤ n+c$.  Since $n-k+1 ≥ 1$, we can thus write
$$
β = \sum_{j=1}^{n+c} i_X^{\ast}(dz_j) Λ β_j
$$
for certain Kähler differentials $β_j ∈ H⁰(X, Ω_X^{n-k})$.  The holomorphic
$k$-forms $α Λ i_X^{\ast}(dz_j)$ and $d α$ extend holomorphically to
$\wtilde{X}$, by assumption, and so \theoremref{thm:extension-Kahler} guarantees
that the holomorphic $n$-forms $α Λ i_X^{\ast}(dz_j) Λ β_j$ and $dα Λ γ$ extend
to $\wtilde{X}$ as well.  It follows that $α Λ β$ and $dα Λ γ$ extend to
$\wtilde{X}$, and this implies that $α$ itself extends to $\wtilde{X}$, by
another application of \theoremref{thm:extension-Kahler}.  \qed

\subsection{Proof of Theorem~\ref*{thm:main-log}}
\approvals{Christian & yes \\Stefan & yes}

The proof of \theoremref{thm:main-log} is nearly identical to the proof of
\theoremref{thm:main-new}.  The only difference is that one uses
\theoremref{thm:extension-Kahler-log} instead of
\theoremref{thm:extension-Kahler}.  \qed

%
%
\svnid{$Id: S12-proof-ext.tex 269 2020-01-20 11:28:53Z kebekus $}

\section{Extension for $(n-1)$-forms, proof of Theorem~\ref*{thm:extension-n-1}}
\subversionInfo
\approvals{Christian & yes \\Stefan & yes}

We maintain the notation and assumptions of \theoremref{thm:extension-n-1}, but
we allow $X$ to be any reduced complex space of pure dimension $n$.  Recall that
$r \colon \wtilde{X} → X$ is a log resolution such that the natural morphism
$r_* Ω_{\wtilde{X}}^n ↪ j_* Ω_{X_{\reg}}^n$ is an isomorphism.  Our task is to
show that the natural morphism
\[
  r_* Ω^{n-1}_{\wtilde{X}}(\log E)(-E) ↪ j_* Ω^{n-1}_{X_{\reg}}
\]
is an isomorphism, or equivalently, that sections of
$f_* Ω^{n-1}_{\wtilde{X}}(\log E)(-E)$ extend uniquely across $X_{\sing}$.  It
is easy to see by duality that all the sheaves
$r_* Ω^p_{\wtilde{X}}(\log E)(-E)$ are independent of the choice of log
resolution.  Shrinking $X$ and replacing $r$ with the canonical strong
resolution of singularities, we may assume that we are in the setting described
in \parref{sec:pf-setup} and \parref{sec:pf-mixed}.  We use the notation
introduced there.

\subsection*{The weight filtration on $N_0$}
\approvals{Christian & yes \\Stefan & yes}

The proof relies the results of \parref{ssec:twfoN0}, where we analysed the
weight filtration on the mixed Hodge module
$N_0 = H⁰ f_* \bigl( j_* ℚ_{\wtilde{X} ∖ E}^H \decal{n} \bigr) ∈ \MHM(Y)$.  To
begin, recall from \propositionref{prop:1-forms-N0} that we have an isomorphism
\[
  \derR f_* Ω¹_{\wtilde{X}}(\log E) \decal {n-1} %
  ≅ \gr_{-1}^F \DR(\Nmod_0).
\]
Using Grothendieck duality for the proper holomorphic mapping
$f \colon \wtilde{X} → Y$, we obtain
\[
  \derR \sHom_{𝒪_Y} \Bigl( \derR f_* Ω^{n-1}_{\wtilde{X}}(\log E)(-E), ω^•_Y \Bigr) %
  ≅ \derR f_* Ω¹_{\wtilde{X}}(\log E) \, \decal{n} %
  ≅ \gr_{-1}^F \DR(\Nmod_0) \, \decal{1}.
\]
According to the extension criterion for complexes in
\propositionref{prop:reflexive-complex-ng}, it is therefore sufficient to prove
the collection of inequalities
\begin{equation} \label{eq:ineq-N0}
  \dim \Supp ℋ^j \gr_{-1}^F \DR(\Nmod_0) ≤ -(j+1) \quad \text{for every $j ≥ -n+2$.}
\end{equation}
On the other hand, recall from \corollaryref{cor:x7} that, for all $j ∈ ℤ$, one
has an exact sequence
\begin{equation}\label{eq:u23}
  ℋ^j \gr_{-1}^F \DR(\Mmod_X) %
  → ℋ^j \gr_{-1}^F \DR(\Nmod_0) %
  → \bigoplus_{i ∈ I} R^{n-1+j} f_* 𝒪_{E_i},
\end{equation}
The inequalities in \eqref{eq:ineq-N0} will follow from the analogous
inequalities for the dimension of the support of the first and third term
in~\eqref{eq:u23}.

\subsection*{The first term in~\eqref{eq:u23}}
\approvals{Christian & yes \\Stefan & yes}

The first term is easily dealt with.  Since we are in the setting of
\theoremref{thm:main-new}, an application of \propositionref{prop:main-ineq}
gives the additional inequalities
\begin{equation}\label{eq:u23-1}
  \dim \Supp ℋ^j \gr_{-1}^F \DR(\Mmod_X) ≤ -(j+1) \quad \text{for every
    $j ≥ -n+2$.}
\end{equation}
This is half of what we need to prove \eqref{eq:ineq-N0}.

\subsection*{The third term in~\eqref{eq:u23}}
\approvals{Christian & yes \\Stefan & yes}

Now we turn to the third term.  Fix an index $i ∈ I$.  Pushing forward the
standard short exact sequence
\[
  0 → 𝒪_{\wtilde{X}}(-E_i) → 𝒪_{\wtilde{X}} → 𝒪_{E_i} → 0
\]
along $f \colon \wtilde{X} → Y$ gives us an exact sequence
\[
  \underbrace{R^{n-1+j} f_* 𝒪_{\wtilde{X}}}_{=: \sf A} %
  → R^{n-1+j} f_* 𝒪_{E_i} %
  → \underbrace{R^{n+j} f_* 𝒪_{\wtilde{X}}(-E_i)}_{=: \sf B}.
\]
But then, the following inequalities will hold for every $j ≥ -n+2$,
\begin{align*}
  \dim \Supp \sf B & ≤ -(j+1) && \text{for dimension reasons} \\
  \dim \Supp \sf A & = \dim \Supp ℋ^{j-1} \gr_0^F \DR(\Mmod_X) && \text{by \propositionref{prop:spidz}} \\
                   & ≤ -(j-1+2) && \text{by \propositionref{prop:main-ineq}}
\end{align*}
In summary, we have $\dim \Supp R^{n-1+j} f_* 𝒪_{E_i} ≤ -(j+1)$ for every
$i ∈ I$ and every $j ≥ -n+2$.  As discussed above, together with
\eqref{eq:u23-1} this suffices to the inequalities in \eqref{eq:ineq-N0}.  The
proof of \theoremref{thm:extension-n-1} is therefore complete.  \qed

\vspace{\baselineskip}

We again record the following corollary of the proof.

\begin{cor}
  In the setting of \theoremref{thm:extension-n-1}, one has
  \[
    \dim \Supp R^j f_* Ω¹_{\wtilde{X}}(\log E) ≤ n-2-j, \quad\text{for every $j ≥ 1$.} \eqno \qed
  \]
\end{cor}

%
%
\svnid{$Id: S13-proof-MOP.tex 269 2020-01-20 11:28:53Z kebekus $}

\section{Local vanishing, proof of Theorem~\ref*{thm:MOP}}
\subversionInfo
\approvals{Christian & yes \\Stefan & yes}

We maintain the notation and assumptions of \theoremref{thm:MOP}, but we allow
$X$ to be any reduced complex space of pure dimension $n$.  Recall that
$r \colon \wtilde{X} → X$ is a log resolution of singularities such that
$R^{n-1} r_* 𝒪_{\wtilde{X}} = 0$.  Our goal is to prove that
$R^{n-1} r_* Ω¹_{\wtilde{X}}(\log E) = 0$.  Both the assumptions and the
conclusion of \theoremref{thm:MOP} are independent of the choice of the
resolution: the former because complex manifolds have rational singularities,
the latter by \cite[Lem.~1.1]{Mustata+Olano+Popa:LocalVanishing}.  We may
therefore assume that we are in the setting described in \parref{sec:pf-setup},
and use the notation introduced there.

\subsection*{Reduction to a statement about $M_X$}
\approvals{Christian & yes \\Stefan & yes}

We have already done pretty much all the necessary work during the proof of
\theoremref{thm:extension-n-1}, and so we shall be very brief.  As in the proof
of \theoremref{thm:extension-n-1}, we have an isomorphism
\[
  R^{n-1} f_* Ω¹_{\wtilde{X}}(\log E) ≅ ℋ⁰ \gr_{-1}^F \DR(\Nmod_0).
\]
\corollaryref{cor:x7} provides us with an exact sequence
\[
  ℋ⁰ \gr_{-1}^F \DR(\Mmod_X) %
  → ℋ⁰ \gr_{-1}^F \DR(\Nmod_0) %
  → \bigoplus_{i ∈ I} R^{n-1} f_* 𝒪_{E_i}.
\]
The assumption that $R^{n-1} f_* 𝒪_{\wtilde{X}} = 0$ yields
$R^{n-1} f_* 𝒪_{E_i} = 0$ for every $i ∈ I$, because $𝒪_{E_i}$ is a quotient of
$𝒪_{\wtilde{X}}$.  To prove \theoremref{thm:MOP}, it will therefore suffice to
prove the vanishing of $ℋ⁰ \gr_{-1}^F \DR(\Mmod_X)$, and this is what we will do
next.

\subsection*{End of proof}
\approvals{Christian & yes \\Stefan & yes}

Recall from \eqref{eq:DS} that
$ℋ^{-1} \gr_0^F \DR(\Mmod_X) ≅ R^{n-1} f_* 𝒪_{\wtilde{X}}$, which vanishes by
assumption.  As in the proof of \lemmaref{lem:reflexive-pure}, consider the
short exact sequence of complexes
\[
  0 → F_{-1} \DR(\Mmod_X) → F_0 \DR(\Mmod_X) → \gr^F_0 \DR(\Mmod_X) → 0,
\]
and the associated sequence of cohomology sheaves
\[
  ⋯ → \underbrace{ℋ^{-1} \gr_0^F \DR(\Mmod_X)}_{= 0\text{ by ass.}} %
  → ℋ⁰ F_{-1} \DR(\Mmod_X) %
  → \underbrace{ℋ⁰ F_0 \DR(\Mmod_X)}_{= 0\text{ by
      Cor.~\ref{cor:IC-vanishing}}} %
  → ⋯
\]
to see that $ℋ⁰ F_{-1} \DR(\Mmod_X) = 0$.  Next, we look at the sequence
\[
  0 → F_{-2} \DR(\Mmod_X) → F_{-1} \DR(\Mmod_X) → \gr^F_{-1} \DR(\Mmod_X) → 0
\]
and its cohomology,
\[
  ⋯ → \underbrace{ℋ⁰ F_{-1} \DR(\Mmod_X)}_{= 0} %
  → ℋ⁰ \gr^F_{-1} \DR(\Mmod_X) %
  → \underbrace{ℋ¹ F_{-2} \DR(\Mmod_X)}_{\mathclap{= 0 \text{, since
        concentr.~in non-pos.~degrees}}} %
  → ⋯,
\]
to conclude the proof.  \qed

%
%
\svnid{$Id: S14-proof-pullBack.tex 269 2020-01-20 11:28:53Z kebekus $}

\section{Pull-back, proof of Theorem~\ref*{thm:pullBack}}
\label{ssec:pbform}
\subversionInfo
\approvals{Christian & yes \\Stefan & yes}

As promised in \parref{sec:fpb}, the following result specifies the ``natural
universal properties'' mentioned in \theoremref{thm:pullBack}.  With
\theoremref{thm:main-new} at hand, the proof is almost identical to the proof
given in \cite{MR3084424} for spaces with klt singularities.

\begin{thm}[Functorial pull-back for reflexive forms]\label{thm:PB-thmA}
  Let ${\sf RSing}$ be the category of complex spaces with rational
  singularities, where morphisms are simply the holomorphic mappings.  Then,
  there exists a unique contravariant functor,
  \begin{equation}\label{eq:CVF}
    \begin{matrix}
      \drefl \colon & {\sf RSing} & → & \{ ℂ\text{-vector spaces} \},\\
      & X & ↦ & H⁰ \bigl( X,\, Ω^{[p]}_X \bigr)
    \end{matrix}
  \end{equation}
  that satisfies the following ``compatibility with Kähler differentials''.  If
  $f \colon Z → X$ is any morphism in ${\sf RSing}$ such that the open set
  $Z° := Z_{\reg} ∩ f^{-1}(X_{\reg})$ is not empty, then there exists a
  commutative diagram
  \[
    \begin{tikzcd}[column sep=huge, row sep=large]
      H⁰ \bigl( X,\, Ω^{[p]}_X \bigr) \rar{\drefl f} \dar[swap]{\restr_X} &
      H⁰ \bigl( Z,\, Ω^{[p]}_Z \dar{\restr_Z} \bigr) \\
      H⁰ \bigl( X_{\reg},\, Ω^p_{X_{\reg}} \bigr) \rar[swap]{\dK(f|_{Z°})} &
      H⁰ \bigl( Z°,\, Ω^p_{Z°} \bigr),
    \end{tikzcd}
  \]
  where $\dK (f|_{Z°})$ denotes the usual pull-back of Kähler differentials, and
  where $\dK (f|_{Z°})$ denotes the usual pull-back of Kähler differentials, and
  $\drefl f$ denotes the linear map of complex vector spaces induced by the
  contravariant functor \eqref{eq:CVF}.
\end{thm}

\begin{rem}[Rational vs.~weakly rational singularities in Theorem~\ref{thm:PB-thmA}]\label{rem:rvwr}
  We do not expect Theorem~\ref{thm:PB-thmA} to hold true if one replaces
  ``rational'' by ``weakly rational'' singularities.  As we will see in Step~2
  of the sketched proof, the result relies on a theorem of Namikawa which is
  specific to rational singularities.
\end{rem}

The universal properties spelled out in \theoremref{thm:PB-thmA} above have a
number of useful consequences that we briefly mention.  Again, statements and
proof are similar to the algebraic, klt case.  To avoid repetition, we merely
mention those consequences and point to the paper \cite{MR3084424} for precise
formulations and proofs.

\begin{fact}[\protect{Additional properties of pull-back, \cite[§5]{MR3084424}}]
  The pull-back functor of \theoremref{thm:PB-thmA} has the following additional
  properties.
  \begin{enumerate}
  \item Compatibility with open immersions, \cite[Prop.~5.6]{MR3084424}.
  \item Compatibility with Kähler differentials for morphisms to smooth targets
    varieties, \cite[Prop.~5.7]{MR3084424}.
  \item Induced pull-back morphisms at the level of sheaves,
    \cite[Cor.~5.10]{MR3084424}.
  \item Compatibility with wedge products and exterior derivatives,
    \cite[Prop.~5.13]{MR3084424}.  \qed
  \end{enumerate}
\end{fact}

\subsection{Sketch of proof for Theorem~\ref*{thm:PB-thmA}}
\approvals{Christian & yes \\Stefan & yes}

For quasi-projective varieties with klt singularities, the result has already
been shown in \cite[Thm.~5.2]{MR3084424}.  If $X$ is a complex space with
arbitrary rational singularities, the proof given in \cite{MR3084424} applies
with minor modifications once the following obvious adjustments are made.
\begin{itemize}
\item Replace all references to the extension theorem \cite[Thm.~1.4]{GKKP11},
  which works for klt spaces only, by references to \theoremref{thm:main-new},
  which also covers the case of rational singularities.
  
\item Equation~\cite[(6.10.5)]{MR3084424} is shown for klt spaces using
  Hacon-McKernan's solution of Shokurov's rational connectivity conjecture.
  However, is has been shown by Namikawa, \cite[Lem.~1.2]{MR1819886}, that the
  equation holds more generally, for arbitrary complex spaces with rational
  singularities.
    
\item If $X$ in ${\sf RSing}$ is a complex space that does not necessarily carry
  an algebraic structure, then one also needs to modify the proof of
  \cite[Lem.~6.15]{MR3084424}, replacing the reference to
  \cite[Cor.~2.12(ii)]{GKK08} by its obvious generalisation to complex spaces.
\end{itemize}

For the convenience of the reader, we include a sketch of proof that summarises
the main ideas and simplifies \cite{MR3084424} a little.  Let $f \colon Z → X$
be any holomorphic map between normal complex spaces with rational
singularities.  Given any $σ ∈ H⁰ \bigl( X,\, Ω^{[p]}_X \bigr)$, we explain the
construction of an appropriate pull-back form
$τ ∈ H⁰ \bigl( Z,\, Ω^{[p]}_Z \bigr)$ and leave it to the reader to check that
this $τ$ is independent of the choices made, and satisfies all required
properties.

\subsubsection*{Step 1}
\approvals{Christian & yes \\Stefan & yes}

To find a reflexive form $τ ∈ H⁰ \bigl( Z,\, Ω^{[p]}_Z \bigr)$, it is equivalent
to find a big, open subset $Z° ⊆ Z_{\reg}$ and an honest form
$τ° ∈ H⁰ \bigl( Z°,\, Ω^p_{Z°} \bigr)$.  We can therefore assume from the outset
that $Z$ is smooth.  Next, let $T := \overline{f(Z)}$ denote the Zariski closure
of the image, and let $\wtilde{T}$ be a desingularisation.  The morphism $f$
factors as
\[
  \begin{tikzcd}[column sep=huge]
    Z \rar[dashed,swap]{\text{meromorphic}} \arrow[bend left=15]{rrr}{f} & \wtilde{T}
    \rar[swap]{\text{desingularisation}} & T \rar[swap]{\text{inclusion}} & X
  \end{tikzcd}
\]
Now, if we can find an appropriate pull-back form
$τ_{\wtilde{T}} ∈ H⁰ \bigl( \wtilde{T},\, Ω^p_{\wtilde{T}} \bigr)$, we could
use the standard fact \cite[Rem.~1.8(1)]{MR1326624} that the meromorphic map
$Z \dasharrow \wtilde{T}$ is well-defined on a big, Zariski-open subset of $Z$
to find the desired form $τ$ by pulling back.  Replacing $Z$ by $\wtilde{T}$, if
need be, we may therefore assume without loss of generality that $Z$ is smooth
and that the image $T := f(Z)$ is closed in Zariski topology.

\subsubsection*{Step 2}
\approvals{Christian & yes \\Stefan & yes}

Next, choose a desingularisation $π \colon \wtilde{X} → X$ such that
$E := \supp π^{-1}(T)$ is an snc divisor.  We will then find a Zariski open
subset $T° ⊆ T_{\reg}$ with preimage $E° := \supp π^{-1}(T°)$ such that
$E° → T°$ is relatively snc.  The assumption that $X$ has rational singularities
is used in the following claim\footnote{The paper \cite{MR3084424} uses
  Hacon-McKernan's solution of Shokurov's rational connectivity conjecture and
  the more involved technique ``projection to general points of $T$'' to prove
  this result.}.

\begin{claim}\label{cl:1}
  If $t ∈ T°$ is any point with fibre $E_t := \supp π^{-1}(t)$, then
  $$
  H⁰ \left( E_t,\; \factor{Ω^{p}_{E_t}}{\tor} \right)= 0.
  $$
\end{claim}
\begin{proof}[Proof of \claimref{cl:1}]
  In case where $E_t ⊂ \wtilde{X}$ is a divisor, this is a result of Namikawa,
  \cite[Lem.~1.2]{MR1819886}.  If $E_t$ is not a divisor, we can blow up and
  apply Namikawa's result upstairs.  The claim then follows from the elementary
  fact that sheaves of ``Kähler differentials modulo torsion'' have good
  pull-back properties, \cite[§2.2]{MR3084424}.  \qedhere\mbox{\quad(\claimref{cl:1})}
\end{proof}

\subsubsection*{Step 3}
\approvals{Christian & yes \\Stefan & yes}

Again using that $X$ has rational singularities, \theoremref{thm:main-new}
yields a form
$τ_{\wtilde{X}} ∈ H⁰ \bigl( \wtilde{X},\, Ω^p_{\wtilde{X}} \bigr)$.  The
following claim asserts that its restriction to $E°$ comes from a form $τ_{T°}$
on $T°$.

\begin{claim}\label{cl:2}
  There exists a unique differential form
  $τ_{T°} ∈ ∈ H⁰ \bigl( T°,\, Ω^p_{T°} \bigr)$ such that
  $τ_{\wtilde{X}}|_{E°}$ and $\dK(π|_{E°})(τ_{T°})$ agree up to torsion.
\end{claim}
\begin{proof}[Proof of \claimref{cl:2}]
  Almost immediate from \claimref{cl:1} and standard relative differential
  sequences for sheaves of Kähler differentials modulo torsion,
  \cite[Prop.~3.11]{MR3084424}.  \qedhere\mbox{\quad(\claimref{cl:2})}
\end{proof}

Pulling the form $τ_{T°}$ back to $Z° := f^{-1}(T°)$, we find a form $τ°$ on the
open set $Z° := f^{-1}(T°)$, which is a non-empty subset of $Z$ since
$T := f(Z)$ is closed in Zariski topology, but need not be big.  We leave it to
the reader to follow the arguments in \cite[§6 and 7]{MR3084424} to see that
this $τ°$ extends to a form $τ$ on all of $Z$ that it is independent of the
choices made and satisfies all required properties.  \qed

\appendix
\phantomsection\addcontentsline{toc}{part}{Appendix}
%
%
\svnid{$Id: SAA-weaklyRatl.tex 242 2018-12-03 15:15:41Z kebekus $}

\section{Weakly rational singularities}
\label{sec:wratlSings}
\subversionInfo

\subsection{Definition and examples}
\label{ssec:wratlDef}
\approvals{Christian & yes \\Stefan & yes}

Let $X$ be a normal complex space.  The main result of this paper asserts that
if top-forms on $X_{\reg}$ extend to regular top-forms on one desingularisation,
then the same will hold for reflexive $p$-forms, for all values of $p$ and all
desingularisations.  Spaces whose top-forms extend therefore seem to play an
important role.  We refer to them as spaces with \emph{weakly rational}
singularities and briefly discuss their main properties in this appendix.

\begin{defn}[Weakly rational singularities]
  Let $X$ be a normal complex space.  We say that $X$ has \emph{weakly rational
    singularities} if the Grauert-Riemenschneider sheaf $ω_X^{\GR}$ is
  reflexive.  In other words, $X$ has weakly rational singularities if for every
  (equivalently: one) resolution of singularities, $r \colon \wtilde{X} → X$,
  the sheaf $r_* ω_{\wtilde{X}}$ is reflexive.  We say that a variety has weakly
  rational singularities if its underlying complex space does.
\end{defn}

\begin{example}[Rational singularities]
  Recall from \parref{sssec:wratl} that rational singularities are weakly
  rational.  For a concrete example, let $X$ be the affine cone over a Fano
  manifold $Y$ with conormal bundle $L := ω_Y^{-1}$, as discussed in
  \cite[§3.8]{MR3057950}.  By \cite[Prop.~3.13]{MR3057950}, this implies that
  $X$ has rational singularities because $L^m$ is the tensor product of $ω_Y$
  with the ample line bundle $ω_Y^{-1} ⊗ L^m$.  A perhaps more surprising
  example is that any affine cone over an Enriques surface has rational
  singularities.
\end{example}

\begin{example}[Varieties with small resolutions]
  If a normal complex space $X$ admits a small resolution, then $X$ has weakly
  rational singularities.  For a concrete example of a non-rational singularity
  of this form, consider an elliptic curve $E$ and a very ample line bundle
  $L ∈ \Pic(E)$.  Let $\wtilde{X} → E$ be the total space of the vector bundle
  $L^{-1} ⊕ L^{-1}$ and identify $E$ with the zero-section in $\wtilde{X}$.  We
  claim that there exists a normal, affine variety $X$ and a birational morphism
  $r \colon \wtilde{X} → X$ that contracts $E ⊂ \wtilde{X}$ to a normal point
  $x ∈ X$ and is isomorphic elsewhere.  An elementary computation shows that
  $R¹r_* 𝒪_{\wtilde{X}} \ne 0$, so $X$ does not have rational singularities.
  \Publication{The preprint version of this paper spells out more
    details.}\Preprint{
    
    To construct the contraction in detail, one might either invoke \cite[Thm.~3
    on p.~59]{MR673560}, or argue directly as follows.  Write $ℒ$ for the sheaf
    of holomorphic sections in $L$ and consider the nef, locally free sheaf
    $ℰ := ℒ ⊕ ℒ ⊕ 𝒪_E$.  The space $ℙ(ℰ)$ is a natural compactification of
    $\wtilde{X}$, the bundle $𝒪_{ℙ(ℰ)}(1)$ is nef and big on $ℙ(ℰ)$, and its
    restriction to $\wtilde{X}$ is trivial.  We can therefore identify sections
    in $𝒪_{ℙ(ℰ)}(m)$ with functions on $\wtilde{X}$, set
    $$
    X := \Spec \bigoplus_{m ∈ ℕ} H⁰ \bigl( ℙ(ℰ),\, 𝒪_{ℙ(ℰ)}(m) \bigr)
    $$
    and obtain the desired map $r \colon \wtilde{X} → X$.  Denoting the ideal
    sheaf of $E ⊂ V$ by $𝒥_E$ and the $m$th infinitesimal neighbourhood of $E$
    in $\wtilde{X}$ by $E_m$, the cohomology of the standard sequence
    $$
    0 → \underbrace{\factor{𝒥_E^m}{𝒥_E^{m+1}}}_{≅ \Sym^m (ℒ ⊕
      ℒ)} → 𝒪_{E_m} → 𝒪_{E_{m-1}} → 0
    $$
    then shows that the restrictions
    $H¹\bigl( E,\, 𝒪_{E_m} \bigr) → H¹\bigl( E,\, 𝒪_{E_{m-1}}\bigr)$ are
    isomorphic for all $m$, so that
    $$
    (R¹ r_*𝒪_{\wtilde{X}})^{\what\ }_x = \underset{\leftarrow}{\lim}\: H¹\bigl( E,\, 𝒪_{E_m} \bigr) \ne
    0,
    $$
    as required.}
\end{example}

Perhaps somewhat counter-intuitively, there are example of log-canonical
varieties $X$ whose singularities are weakly rational but not rational.  If
$K_X$ is Cartier and $ω_X$ is locally generated by one element, this can of
course not happen, so that the canonical divisors of the examples will never be
Cartier.

\begin{example}[Some log canonical singularities are weakly rational, not rational]
  To start, let $E$ be a smooth projective variety of positive irregularity
  whose canonical divisor is torsion, but not linearly trivial.  Let
  $L ∈ \Pic(E)$ be very ample, and let $X$ be the affine cone over $E$ with
  conormal bundle $L$.  By \cite[§3.8]{MR3057950}, $X$ is log canonical and does
  \emph{not} have rational singularities.  Yet, \propositionref{prop:extCone}
  asserts that the singularities of $X$ are weakly rational.  \Publication{The
    preprint version of this paper discusses a concrete example.}

  \Preprint{For a concrete example, let $S$ be a K3 surface obtained as a double
    cover of the projective plane branched along a non-singular degree six.
    Observe that the Galois involution $σ ∈ \Aut(S)$ acts non-trivially on
    $H⁰(S,\, ω_S) ≅ ℂ$.  Let $C$ be an elliptic curve, and let
    $τ ∈ \Aut(C)$ be a translation by a torsion element of degree two, so that
    $τ$ is again an involution.  Consider the involution $(σ,τ) ∈ \Aut(S ⨯ C)$,
    which is fixed point free, and choose $E$ to be the quotient,
    $E := (S ⨯ C)/ℤ_2$.  The threefold $E$ admits no global top-form by choice
    of $σ$, and has positive irregularity since it admits a morphism to the
    elliptic curve $C/ℤ_2$.}
\end{example}

\begin{rem}[Incompatible definitions in the literature]
  There already exists a notion of ``weakly rational'' in the
  literature.  Andreatta-Silva \cite{MR765916} call a variety $X$ weakly
  rational if $R^{\dim X-1}r_* 𝒪_{\wtilde{X}} = 0$ for one (or equivalently, any)
  resolution of singularities.  They seem to be assuming implicitly that $X$ has
  isolated singularities, although they do not include this assumption into the
  definition.  (For a complex space with isolated singularities, both definitions are
  equivalent.)
\end{rem}

\subsection{Behaviour with respect to standard constructions}
\approvals{Christian & yes \\Stefan & yes}

In view of their importance for our result, we briefly review the main
properties of weakly rational singularities, in particular their behaviour under
standard operations of birational geometry.

\subsubsection{Positive results}
\approvals{Christian & yes \\Stefan & yes}

In the positive direction, we show that weakly rational singularities are
stable under general hyperplane sections, and that a space has weakly rational
singularities if it is covered by a space with weakly rational singularities.

\begin{prop}[Stability under general hyperplane sections]
  Let $X$ be a quasi-projective variety with weakly rational singularities, let
  $L ∈ \Pic(X)$ be a line bundle and $\bL ⊆ |L|$ be a finite-dimensional,
  basepoint free linear system whose general member is connected.  Then, there
  exists a dense, Zariski-open subset $\bL° ⊆ \bL$ such that any hyperplane
  $H ∈ \bL°$ has weakly rational singularities, and satisfies the adjunction
  formula
  \begin{equation}\label{eq:adjunction}
    ω_H^{\GR} ≅ ω_X^{\GR} ⊗ 𝒪_X(H) ⊗ 𝒪_H.
  \end{equation}
\end{prop}
\begin{proof}
  Choose a resolution of singularities, $r \colon \wtilde{X} → X$.  There exists
  a dense, Zariski-open $\bL° ⊆ L$ such that any hyperplane $H ∈ \bL°$ satisfies
  the following properties.
  \begin{enumerate}
  \item\label{il:A3} The hypersurface $H$ is normal, connected and
    $H_{\sing} = X_{\sing} ∩ H$: Seidenberg's theorem, \cite{Seidenberg50}, and
    the fact that a variety is smooth along a Cartier divisor if the divisor
    itself is smooth.
  \item\label{il:B3} The preimage $\wtilde{H} := r^{-1} H$ is smooth: Bertini's
    theorem.
  \item\label{il:C3} The restriction $ω_X^{\GR}|_H$ is reflexive:
    \cite[Thm.~12.2.1]{EGA4-3}.
  \end{enumerate}
  We claim that the adjunction formula~\eqref{eq:adjunction} holds for $H$,
  which together with \ref{il:C3} implies that $H ∈ \bL°$ has weakly rational
  singularities.  The setup is summarised in the following diagram
  \[
    \begin{tikzcd}[row sep=large, column sep=18 ex]
      \wtilde{H} \dar[swap]{r_H\text{, resolution}} \rar{\wtilde{ι}\text{, closed embedding}} &
      \wtilde{X} \dar{r\text{, resolution}} \\
      H \rar[swap]{ι\text{, closed embedding}} & X
    \end{tikzcd}
  \]
  We obtain an adjunction morphism,
  \begin{equation}\label{eq:ad2}
    \begin{aligned}
      ι^* \bigl( ω^{\GR}_X(H) \bigr) & ≅ ι^* r_* \bigl( ω_{\wtilde{X}}(\wtilde{H}) \bigr) && \text{Projection formula} \\
      & → (r_H)_* \wtilde{ι}^* \bigl( ω_{\wtilde{X}}(\wtilde{H}) \bigr) && \text{Cohomology and base change} \\
      & ≅ (r_H)_* ω_{\wtilde{H}}
      ≅ ω^{\GR}_H && \text{Adjunction and smoothness of } \wtilde{H}
    \end{aligned}
  \end{equation}
  which is clearly an isomorphism over the big open subset of $H$ where $H$ and
  $X$ are both smooth.  More can be said.  Item~\ref{il:C3} implies that the
  left hand side of \eqref{eq:ad2} is reflexive, while the right hand side of
  \eqref{eq:ad2} is a push forward of a torsion free sheaf, hence torsion free.
  As a morphism from a reflexive to a torsion free sheaf that is isomorphic in
  codimension one, the adjunction morphism must then in fact be isomorphic.
\end{proof}

As a second positive result, we show that images of weakly rational
singularities under arbitrary finite morphisms are again weakly rational.  This
can be seen as an analogue of the fact that quotients of rational singularities
under the actions of finite groups are again rational.  \Publication{The proof
  follows along the lines of \cite[proof of Cor.~3.2]{GKK08} and is therefore
  omitted here.  The preprint version of this paper spells out all details.}

\begin{prop}[Stability under finite quotients]
  Let $γ \colon X → Y$ be a proper, surjective morphism between normal complex
  spaces.  Assume that $γ$ is finite, or that it bimeromorphic and small.  If
  $X$ has weakly rational singularities, then so does $Y$.  \Publication{\qed}
\end{prop}
\Preprint{%
  \begin{proof}
    The case of a small morphism is rather trivial, so we consider finite
    morphisms only.  We assume without loss of generality $Y$ is Stein.  Let
    $r_Y \colon \wtilde{Y} → Y$ be a log-resolution, with exceptional set
    $E ⊂ \wtilde{Y}$.

    Since $Y$ is Stein, to prove that $Y$ has weakly rational singularities, it
    suffices to show that for any given section $σ ∈ H⁰(Y,\, ω_Y)$, the
    associated rational form $\wtilde{σ}$ on $\wtilde{Y}$, which might a priori
    have poles along $E$, does in fact not have any poles.  To this end, let
    $\wtilde{X}$ be a strong resolution of the normalised fibre product
    $X ⨯_Y \wtilde{Y}$.  The following diagram summarises the
    situation:
	 \[
	 \begin{tikzcd}[row sep=huge, column sep=18 ex,ampersand replacement=\&]
      \wtilde{X} \rar{Γ\text{, generically finite}} \dar[swap]{r_X\text{, desing.}}
			\& \wtilde{Y} \dar{r_Y\text{, desing.}} \\
      X \rar[swap]{γ\text{, finite}} \& Y.
	\end{tikzcd}
	\]

    Set $F := \supp Γ^{-1}E$ and consider the rational differential form
    $\wtilde{τ}$ on $\wtilde{X}$, which might a priori have poles along $F$.
    Since $Γ$ is generically finite, \cite[Cor.~2.12(ii)]{GKK08}
    applies\footnote{The reference \cite{GKK08} works in the algebraic setting.
      However, the result quoted here (and its proof) will also be true for
      complex spaces.} to show that $\wtilde{σ}$ is without poles along $E$ if
    and only if $\wtilde{τ}$ is without poles along $F$, or more precisely:
    without poles along those components of $F$ that dominate components of $E$.
    
    To show that $\wtilde{τ}$ has no pole indeed, observe that finiteness of $γ$
    and reflexivity of $ω_X$ imply that there exists a section
    $τ ∈ H⁰(X,\, ω_X)$ that agrees with $dγ(σ)$ wherever $X$ and $Y$ are smooth.
    The assumption that $X$ has weakly rational singularities will then give a
    regular differential form on $\wtilde{X}$, without poles, that agrees with
    $dr_X(τ)$ wherever $X$ is smooth.  This form clearly equals $\wtilde{τ}$.
	 \end{proof}}

\subsubsection{Negative results}
\approvals{Christian & yes \\Stefan & yes}

In spite of the positive results above, the following examples show that the
class of varieties with weakly rational singularities does not remain invariant
when taking quasi-étale covers or special hyperplane sections, even in the
simplest cases.

\begin{example}[Instability under special hyperplane sections]
  Grauert-Riemenschneider construct a normal, two-dimensional, isolated
  hypersurface singularity where $ω_X^{\GR}$ is not reflexive,
  \cite[p.~280f]{GR70}.  In particular, $X$ does \emph{not} have weakly rational
  singularities and a naive adjunction formula for the Grauert-Riemenschneider
  sheaf as in \eqref{eq:adjunction} does not hold in this case.
\end{example}

\begin{example}[Instability under quasi-étale covers]
  Any cone $Y$ over an Enriques surface has rational singularities and admits a
  quasi-étale cover by a cone $X$ over a K3 surface, which is Cohen-Macaulay,
  but does \emph{not} have rational singularities, \cite[Ex.~3.6]{MR3057950}.
  As we saw in \parref{sssec:wratl}, this implies that $X$ does \emph{not}
  have weakly rational singularities.  We obtain examples of quasi-étale maps
  $X → Y$ between isolated, log-canonical singularities where $Y$ is weakly
  rational while $X$ is not.
\end{example}

%
%
\svnid{$Id: SAB-cones.tex 263 2020-01-14 12:37:18Z kebekus $}

\section{Cones over projective manifolds}
\approvals{Christian & yes \\Stefan & yes}
\label{app:cones}

Cones over projective manifolds are a useful class of examples to illustrate how
the extension problem for $p$-forms is related to the behaviour of the canonical
sheaf.  We follow the notation introduced in Kollár's book \cite{MR3057950} and
work in the following setting.

\begin{setting}[\protect{Cones over projective manifolds, compare \cite[§3.1]{MR3057950}}]\label{setting:cone}
  Fix a number $n ≥ 2$ and a smooth projective variety $Y$ of dimension
  $\dim Y = n-1$, together with an ample line bundle $L ∈ \Pic(Y)$.  Following
  \cite[§3.8]{MR3057950}, we define the \define{affine cone over $Y$ with
    conormal bundle $L$} as the affine algebraic variety
  \[
    X := \Spec \bigoplus_{m ≥ 0} H⁰ \bigl(Y,\, L^m\bigr)
  \]
  The ring is finitely generated since $L$ is ample.  The variety $X$ is normal
  of dimension $n$ and smooth outside of the \emph{vertex} $\vec{v}$, which is
  the point corresponding to the zero ideal.  Unless $Y = ℙ^{n-1}$ and
  $L = 𝒪_{ℙ^{n-1}}(1)$, the vertex will always be an isolated singular point.

  Since $Y$ is smooth, the partial resolution of singularities constructed in
  \cite[§3.8]{MR3057950}, say $r \colon \wtilde{X} → X$, is in fact a log
  resolution of singularities.  The variety $\wtilde{X}$ is isomorphic to the
  total space of the line bundle $L^{-1}$ and the $r$-exceptional set
  $E ⊊ \wtilde{X}$ is identified with the zero-section of that bundle.
\end{setting}

\Preprint{
  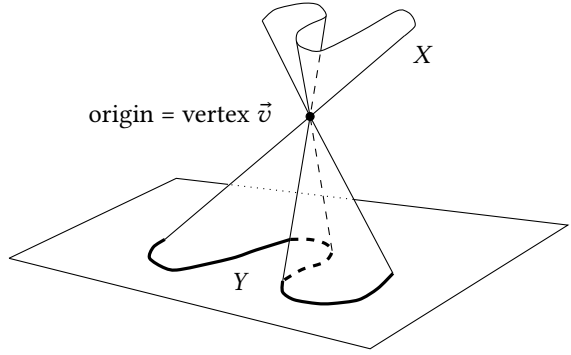
\begin{figure}[h]
    \centering
    \begin{tikzpicture}
      \node[align=center,black] at (4,5.9) {origin = vertex $\vec{v}$};
      \node[align=center,black] at (4.8,3.7) {$Y$};
      \node[align=center,black] at (7.2,6.7) {$X$};
      \draw (1.75,4) -- (6.5,2.8) -- (9.1,4.44) -- (6.28,4.78);
      \draw [dotted] (6.28,4.78) -- (4.63,4.98);
      \draw (4.63,4.98) -- (3.95,5.06) -- (1.75,4);
      \draw [very thick] plot [smooth] coordinates
      {(3.8,4.25) (3.65,4.15) (3.6,4) (3.8,3.87) (4,3.85) (4.68,4) (5,4.11) (5.45,4.25)};
      \draw [very thick,dashed] plot [smooth] coordinates
      {(5.45,4.25) (5.7,4.25) (5.9,4.2) (6,4.1) (5.9,3.95) (5.6,3.85) (5.35,3.7)};
      \draw [very thick] plot [smooth] coordinates
      {(5.35,3.7) (5.35,3.55) (5.5,3.45) (5.8,3.4) (6,3.4) (6.35,3.45) (6.6,3.6) (6.8,3.8)};
      \filldraw (5.71,5.88) circle (1.5pt);
      \draw (3.8,4.25) -- (5.71,5.88) -- (6.95,6.95);
      \draw (6.8,3.8) -- (5.71,5.88) -- (5.08,7.15);
      \draw (5.35,3.7) -- (5.71,5.88) -- (5.53,6.88);
      \draw [dashed] (6,4.1) -- (5.71,5.88) -- (5.85,6.75);
      \draw (5.85,6.75) -- (5.92,7.21);
      \draw plot [smooth] coordinates
      {(5.08,7.15) (5.25,7.3) (5.6,7.38) (5.82,7.36) (5.92,7.21) (5.7,7.1) (5.55,7)
        (5.53,6.88) (5.6,6.8) (5.85,6.75) (6.36,7) (6.85,7.28) (7,7.26) (7.07,7.18) (7.09,7.08)
        (6.99,6.98) (6.95,6.95)};
    \end{tikzpicture}
    
    \caption{Cone over a smooth variety}
    \label{fig:cone}
  \end{figure}

  \begin{note}
    The definition is motivated by the geometric construction of cones, as
    illustrated in Figure~\vref{fig:cone}.  Suppose that $Y$ is a submanifold of
    $ℙ^d$.  The affine cone over $Y$, with vertex the origin in $ℂ^{d+1}$, is
    the union of all the lines in $ℂ^{d+1}$ corresponding to the points of $Y$.
    Its coordinate ring is the graded $ℂ$-algebra
    \[
      ℂ[x_0, x_1, …, x_d] / I_Y,
    \]
    where $I_Y$ is the homogeneous ideal of $Y$.  The affine cone is not always
    normal, but it is easy to see that the coordinate ring of its normalisation
    is $\Spec R$, where $R$ is the section ring of the very ample line bundle
    $𝒪_Y(1)$.  Our definition is slightly more general, because $L$ is only
    assumed to be ample.
  \end{note}
}

\subsection{Extension of differential forms}
\approvals{Christian & yes \\Stefan & yes}

Now we turn out attention to the extension problem for differential forms.  The
following result can be summarised very neatly by saying that if $n$-forms
extend, then $p$-forms extend for every $0 ≤ p ≤ n$.

\begin{prop}[Extension of differential forms on cones]\label{prop:extCone}
  Assume \settingref{setting:cone}.  Then, $p$-forms extend for all $p ≤ n-2$.
  The following equivalences hold in addition.
  \begin{align}
    \label{il:X1} \text{$(n-1)$-forms extend} & ⇔ H⁰\bigl(Y,\, ω_Y ⊗ L^{-m}\bigr) = 0, \forall m ≥ 1.  \\
    \label{il:X2} \text{$n$-forms extend} & ⇔ H⁰\bigl(Y,\, ω_Y ⊗ L^{-m}\bigr) = 0, \forall m ≥ 0.
  \end{align}
\end{prop}
\begin{proof}
  Since $\wtilde{X} ∖ E$ is isomorphic to $X_{\reg}$, the question is simply
  under what conditions on $Y$ and $L$ the restriction mapping
  \[
    H⁰ \bigl( \wtilde{X}, Ω_{\wtilde{X}}^p \bigr) → H⁰ \bigl( \wtilde{X} ∖ E,
    Ω_{\wtilde{X}}^p \bigr)
  \]
  is an isomorphism for different values of $p ∈ \{0, 1, …, n\}$.  We use the
  identification of $\wtilde{X}$ with the total space of the line bundle
  $L^{-1}$ and denote the projection by $q \colon \wtilde{X} → Y$.  The sequence
  of differentials and the sequence of $p$th exterior powers now read as
  follows,
  \[
    0 → q^* Ω_Y¹ → Ω_{\wtilde{X}}¹ → q^* L → 0 %
    \quad\text{and}\quad %
    0 → q^* Ω_Y^p → Ω_{\wtilde{X}}^p → q^* \bigl( Ω_Y^{p-1} ⊗ L \bigr) → 0.
  \]
  Now both $q \colon \wtilde{X} → Y$ and its restriction $q|_{\wtilde{X} ∖ E}$
  are affine, and
  \[
    q_* 𝒪_{\wtilde{X}} ≅ \bigoplus_{m ≥ 0} L^m \quad \text{and} \quad
    (q|_{\wtilde{X} ∖ E})_* 𝒪_{\wtilde{X} ∖ E} ≅ \bigoplus_{m ∈ ℤ} L^m.
  \]
  We therefore obtain the following commutative diagram with exact
  rows:
\[
\begin{tikzpicture}[baseline= (a).base]
\node[scale=0.87] (a) at (0,0){
\begin{tikzcd}[column sep=small,row sep=small]
	0 \rar & \displaystyle \bigoplus_{m ≥ 0} H⁰(Y, Ω_Y^p ⊗ L^m)
		\rar \dar[start anchor={[yshift=2.2ex]},end anchor={[yshift=-0.3ex]}]{α}
			& H⁰ \bigl( \wtilde{X}, Ω_{\wtilde{X}}^p \bigr)
		\rar \dar & \displaystyle \bigoplus_{m ≥ 1} H⁰(Y, Ω_Y^{p-1} ⊗ L^m)
		\rar \dar[start anchor={[yshift=2.2ex]},end anchor={[yshift=-0.3ex]}]{β}
			& \displaystyle \bigoplus_{m ≥ 0} H¹(Y, Ω_Y^p ⊗ L^m)
		\dar[start anchor={[yshift=2.2ex]},end anchor={[yshift=-0.3ex]}] \\
	0 \rar & \displaystyle \bigoplus_{m ∈ ℤ} H⁰(Y, Ω_Y^p ⊗ L^m)
		\rar & H⁰ \bigl( \wtilde{X} ∖ Y, Ω_{\wtilde{X}}^p \bigr)
		\rar & \displaystyle \bigoplus_{m ∈ ℤ} H⁰(Y, Ω_Y^{p-1} ⊗ L^m)
		\rar & \displaystyle \bigoplus_{m ∈ ℤ} H¹(Y, Ω_Y^p ⊗ L^m)
\end{tikzcd}
};
\end{tikzpicture}
\]
  Consider the first vertical arrow, labelled $α$, in the commutative diagram
  above.  By the Nakano vanishing theorem, we have
  $H⁰ \bigl(Y,\, Ω_Y^p ⊗ L^m\bigr) = 0$ for $m ≤ -1$ and $p ≤ \dim Y - 1$, and
  so $α$ is an isomorphism if and only if
  \begin{equation}\label{eq:cones-n-1}
    H⁰\bigl(Y,\, ω_Y ⊗ L^m\bigr) = 0, \quad \forall m ≤ -1.
  \end{equation}
  Consider next the third vertical arrow, labelled $β$, in the commutative
  diagram.  For the same reason as before, we have
  $H⁰\bigl(Y,\, Ω_Y^{p-1} ⊗ L^m\bigr) = 0$ for $m ≤ -1$ and $p-1 ≤ \dim Y-1$.
  For $m = 0$, the horizontal arrow
  \[
    H⁰\bigl(Y,\, Ω_Y^{p-1}\bigr) → H¹\bigl(Y,\, Ω_Y^p\bigr)
  \]
  in the second row is cup product with the first Chern class of the ample line
  bundle $L$; by the Hard Lefschetz Theorem, it is injective as long as
  $p-1 ≤ \dim Y-1$.  Consequently, $β$ is an isomorphism if and only if
  \begin{equation}\label{eq:cones-n}
    H⁰\bigl(Y,\, ω_Y ⊗ L^m\bigr) = 0, \quad \forall m ≤ 0.
  \end{equation}
  The conclusion is that $p$-forms extend for $p ≤ n-2$ without any extra
  assumptions on $(Y, L)$; since the cone over $(Y, L)$ has an isolated
  singularity at the vertex, this is consistent with the result by Steenbrink
  and van Straten \cite[Thm.~1.3]{SS85}.  Moreover, $(n-1)$-forms extend iff the
  condition in \eqref{eq:cones-n-1} is satisfied, and $n$-forms extend iff the
  condition in \eqref{eq:cones-n} is satisfied.
\end{proof}

\subsection{Characterisation of standard singularity types}
\approvals{Christian & yes \\Stefan & yes}

The following summary of several well-known results relates different classes of
singularities to properties of the line bundle $L$, in particular to the
vanishing of higher cohomology for $L$ and its powers.

\begin{prop}[Classes of singularities on cones] \label{pro:singCone}
  Assume \settingref{setting:cone}.  Then, the following equivalences hold.
  \begin{align}
    \text{$X$ has rational singularities} & ⇔ Hⁱ\bigl(Y,\, L^m\bigr) = 0, \forall i > 0, \forall m ≥ 0.  \\
    \text{$X$ has Du Bois singularities} & ⇔ Hⁱ\bigl(Y,\, L^m\bigr) = 0, \forall i > 0, \forall m > 0.  \\
    \text{$X$ is Cohen-Macaulay} & ⇔ Hⁱ\bigl(Y,\, L^m\bigr) = 0, \forall \dim Y > i > 0, \forall m ≥ 0.  \\
    \intertext{The singularity types of the minimal model program are described as follows.}
    \text{$X$ is $ℚ$-Gorenstein} & ⇔ ∃ m : K_Y \sim_{ℚ} L^m.  \\
    \text{$X$ is klt} & ⇔ ∃ m < 0 : K_Y \sim_{ℚ} L^m.  \\
    \text{$X$ is log canonical} & ⇔ ∃ m ≤ 0 : K_Y \sim_{ℚ} L^m.
  \end{align}
\end{prop}
\begin{proof}
  See \cite[Lem.~3.1, Cor.~3.11, Prop.~3.13 and Prop.~3.14]{MR3057950} and
  \cite[Thm~2.5]{MR3247804}.
\end{proof}

Comparing \propositionref{prop:extCone} and \ref{pro:singCone}, we find that the
extension property of $p$-forms is a comparatively mild condition on $(Y, L)$.
It is not as cohomological in nature as ``rational'', ``Du Bois'' and
``Cohen-Macaulay'', and certainly not nearly as restrictive as being klt, which
only happens in the special case where $Y$ is a Fano manifold and $L$ is
$ℚ$-linearly equivalent to a positive multiple of $-K_Y$.  This suggests looking
for an extension theorem that goes beyond the class of singularities used in the
Minimal Model Program.



\end{document}